\documentclass[10pt]{amsart}

\usepackage[latin2]{inputenc}
\usepackage{amsmath}
\usepackage{graphicx}
\usepackage{amssymb}
\usepackage{esint}
\usepackage{color}
\usepackage{amsthm}
\usepackage{epsfig}
\usepackage{enumitem}
\usepackage{mathtools}
\usepackage{esvect}
\usepackage{dsfont}

\usepackage{longtable}
\usepackage{tabu}

\usepackage[dvipsnames]{xcolor}
\usepackage{tikz}
\usepackage{xxcolor}
\usepackage[linktocpage=true,colorlinks=true,linkcolor=Blue,citecolor=Green]{hyperref}

\usepackage{tikz}
\usetikzlibrary{calc,intersections,through,backgrounds,patterns,arrows.meta, math,through}
\usepackage[labelformat = brace, position=top]{subcaption}

\usepackage[english]{babel}

\newtheorem{theorem}{Theorem}
\newtheorem{proposition}[theorem]{Proposition}
\newtheorem{lemma}[theorem]{Lemma}

\newtheorem{definition}[theorem]{Definition}
\newtheorem{remark}[theorem]{Remark}
\newtheorem*{theorem*}{Theorem}

\theoremstyle{definition}
\newtheorem{construction}[theorem]{Construction}

\def\XXint#1#2#3{{\setbox0=\hbox{$#1{#2#3}{\int}$ }
\vcenter{\hbox{$#2#3$ }}\kern-.6\wd0}}

\definecolor{Yellow}{rgb}{0.95,0.9,0.0} 
\definecolor{Red}{rgb}{0.8,0.1,0.1}
\definecolor{Green}{rgb}{0.1,0.65,0.2}
\definecolor{Blue}{rgb}{0.1,0.1,0.8}
\definecolor{Purple}{rgb}{0.7,0.1,0.7}
\definecolor{Grey}{rgb}{0.6,0.6,0.6}

\definecolor{YELLOW}{rgb}{0.95,0.9,0.0} 
\definecolor{RED}{rgb}{0.8,0.1,0.1}
\definecolor{GREEN}{rgb}{0.25,0.65,0.1}
\definecolor{BLUE}{rgb}{0.1,0.1,0.8}
\definecolor{PURPLE}{rgb}{0.7,0.1,0.7}

\newcommand{\supp}{\operatorname{supp}}

\newcommand{\dist}{\operatorname{dist}}

\DeclareMathOperator*{\argmin}{arg\,min}

\newcommand{\Rd}[1][d]{{\mathbb{R}^{#1}}}
\allowdisplaybreaks[4]

\newcommand{\barI}{\bar{I}}
\newcommand{\no}{\bar{\vec{n}}}
\newcommand{\ta}{\bar{\vec{\tau}}}
\newcommand{\taTrJ}{\bar{\vec{t}}}
\newcommand{\trLine}{\bar\Gamma}

\renewcommand{\vec}[1]{{\operatorname{#1}}}

\def\onedot{$\mathsurround0pt\ldotp$}
\def\cdddot#1{
  \mathbin{\vcenter{\baselineskip.67ex
    \hbox{\onedot}\hbox{\onedot}\hbox{\onedot}%
  }}%
}

\begin{document}

\title[Weak-strong uniqueness for MCF of double bubbles]
{Weak-strong uniqueness for the mean curvature flow of double bubbles}

\author{Sebastian Hensel}
\address{Institute of Science and Technology Austria (IST Austria), Am~Campus~1, 
3400 Klosterneuburg, Austria}
\email{sebastian.hensel@ist.ac.at}
\curraddr{Hausdorff Center for Mathematics, Universit{\"a}t Bonn, Endenicher Allee 62, 53115 Bonn, Germany
(\texttt{sebastian.hensel@hcm.uni-bonn.de})}
\author{Tim Laux}
\address{Hausdorff Center for Mathematics, Universit{\"a}t Bonn, Endenicher Allee 62, 53115 Bonn, Germany}
\email{tim.laux@hcm.uni-bonn.de}

\begin{abstract}
We derive a weak-strong uniqueness principle for BV solutions to multiphase
mean curvature flow of triple line clusters in three dimensions. Our proof is based on the explicit construction of a gradient-flow calibration in the sense of the recent work of Fischer et al.~[arXiv:2003.05478v2] for any such cluster. This extends the two-dimensional construction to the three-dimensional case of surfaces meeting along triple junctions.
\end{abstract}


\maketitle


\tableofcontents

\section{Introduction}

Multiphase mean curvature flow (MCF) arises as the $L^2$-gradient flow of the area functional and has been studied intensively over the last decades. 
Its earliest motivation comes from materials science where MCF models the slow relaxation of grain boundaries in polycrystals.
 
 \medskip
 
While a lot of progress has been achieved in the two-dimensional case, often referred to as network flow, in the physically most relevant case of three spatial dimensions, results concerning strong solutions are just beginning to emerge.
The short-time existence for the MCF of three surfaces coming together at a triple junction has been established by Freire \cite{Freire2010} when all three surfaces can be parametrized as graphs over a single plane--a condition which then was relaxed by Depner, Garcke, and Kohsaka  \cite{DepnerGarckeKohsaka} who derived the local well-posedness without relying on this graphical geometry. 
In this work, they parametrize the surface cluster over a fixed reference surface cluster and phrase MCF as a nonlocal, quasilinear parabolic system of free boundary problem.
Independently and as the result of an improved compactness property of Brakke flows, Schulze and White \cite{SchulzeWhite} established short-time existence in a similar geometric setting. 
Recently, Baldi, Haus, and Mantegazza \cite{BaldiHausMantegazza} derived the existence of a self-similarly shrinking ``lens-shaped" surface cluster describing a solution to MCF just before the disappearance of the smaller bubble in the cluster. 
We refer the interested reader to \cite{FischerHenselLauxSimon2D} for a more detailed discussion and further relevant references. 
The construction of regular solutions starting from non-regular initial surface clusters has not yet been accomplished, but the recent microlocal approach \cite{stenf} might give new insights. However, also this approach relies on an explicit construction of gluing in self-similar expanders, which does not immediately carry over to the three-dimensional case.

\medskip

Most results on weak solutions of MCF, however, are often quite general and in particular apply in our present three-dimensional case. While the theory of viscosity solutions is not available for surface clusters (even in two dimensions), Brakke's solution concept, and in particular the more refined version of Kim and Tonegawa \cite{TonegawaKim} apply, and so do the conditional convergence results in \cite{LauxOtto} and \cite{LauxOttoBrakke}.

\medskip

In the present work we prove the stability and weak-strong uniqueness of regular surface clusters with triple junctions moving by MCF in three dimensions. 
The key step is the explicit construction of a gradient-flow calibration in the sense of our recent work  \cite{FischerHenselLauxSimon2D} with Fischer and Simon. 
Therein, we constructed such a gradient-flow calibration in the planar case of networks moving by curve shortening flow and proved that, in arbitrary dimension, any calibrated MCF is stable. The main contribution of the present work is the extension of the first part to the three dimensional case, which then immediately implies the weak-strong uniqueness. 
The concept of gradient-flow calibrations is the time-dependent counterpart of calibrations and paired calibrations for minimal surfaces and minimal surface clusters, respectively, see in particular~\cite{Lawlor-Morgan}.

\medskip

There are two model cases in which our new construction (essentially) reduces to the two-dimensional case: if the three-dimensional configuration is (i) rotationally symmetric or (ii) translation invariant in one direction. However, in the general case, a direct slicing argument across normal planes to the triple junction of course does not yield torsion-free tangent frames of the interfaces. 
By introducing suitable gauge rotations, we correct this ad hoc construction. We prove that these gauged tangent frames then give a natural extension of the respective normal vector fields to a vicinity of each interface and even across the triple line using Herring's condition. Furthermore, we show that these constructions are regular and compatible to first order along the triple line. 
Although the present method would immediately carry over to more general surface clusters only containing smooth surfaces coming together along triple lines, we restrict ourselves to the case of a ``double bubble'', i.e., a cluster of three surfaces as displayed in Figure~\ref{fig:global picture}. 
For general surface clusters in $\Rd[3]$, triple lines could meet in quadruple points some of which will persist over time; cf.~\cite{Taylor1976} for the static case. For these systems, even the short-time existence of regular solutions has not been established. It would be interesting to generalize our present work to construct a gradient-flow calibration in this more general setting.


\begin{figure}
  \centering      
  \includegraphics{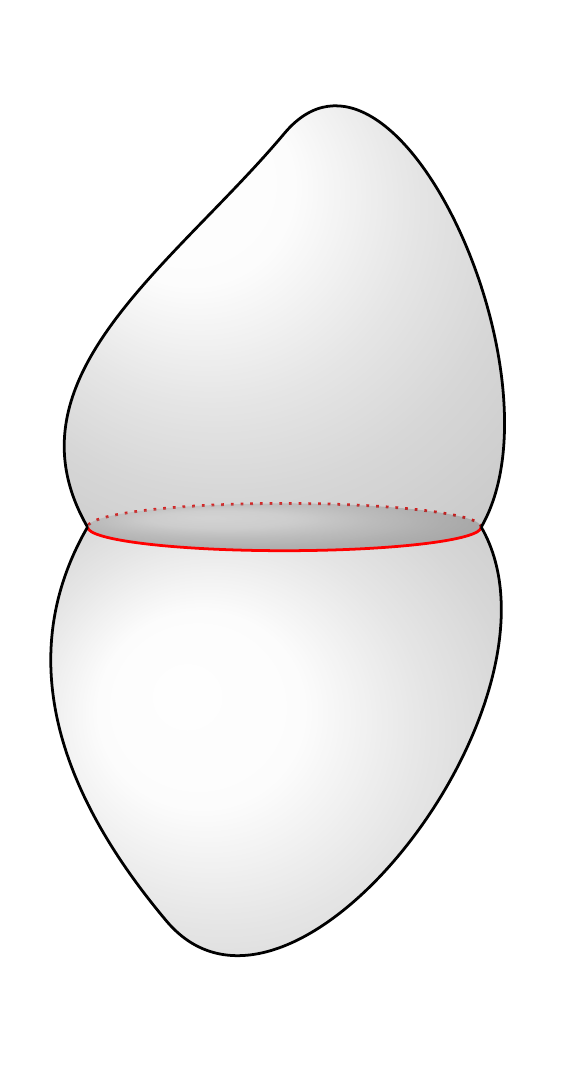}
	\caption{A simplified illustration of a double bubble in three dimensions. The triple line $\trLine$ along which all three interfaces meet is marked in red. We emphasize that neither flatness of an interface nor symmetry
	of the triple line is required for our results.}
	\label{fig:global picture}
\end{figure}
\section{Main results}
In our recent contribution~\cite{FischerHenselLauxSimon2D}
with Fischer and Simon, we developed a general
approach to the question of weak-strong uniqueness of BV~solutions to multiphase
mean curvature flow in arbitrary ambient dimension~$d\geq 2$. This approach splits
into a two-step procedure. 

In a first step, we introduced a novel concept of
calibrated flows with respect to the gradient flow of the interface energy functional
given by the (weighted) sum of the surface areas of the interfaces,
cf.\ \eqref{def:energyFunctional} below. This concept can be interpreted
as the evolutionary analogue of the well-known notion of paired calibrations
due to Lawlor and Morgan~\cite{Lawlor-Morgan} from their study of the
minimization problem of interfacial surface area of networks.
Indeed, the main merit of a calibrated flow is that its existence
(essentially) implies qualitative uniqueness and quantitative stability 
of BV~solutions to multiphase mean curvature flow in arbitrary ambient dimension~$d\geq~2$.

In a second step, we then put this theory to use by showing that any sufficiently regular network of
interfaces in the plane~$\Rd[2]$, which in addition is subject to the correct angle condition
at triple junctions, is in fact calibrated in the precise sense of~\cite{FischerHenselLauxSimon2D}. 
The purpose of the present work is to extend this second step of the approach
to the three-dimensional setting of mean curvature flow of sufficiently regular
double bubbles (again with the correct angle condition along the triple line).
The main contributions are summarized in the following result.

\begin{theorem}
\label{theo:mainResultsCombined}
Let~$T\in (0,\infty)$ be a time horizon, and let
$(\bar\Omega_1,\bar\Omega_2,\bar\Omega_3)$ be a regular 
double bubble smoothly evolving by~MCF on~$[0,T]$ in the sense of 
Definition~\ref{def:smoothSolution}.
The evolution of $(\bar\Omega_1,\bar\Omega_2,\bar\Omega_3)$
on $[0,T]$ is then calibrated in the sense that
there exists an associated gradient-flow calibration
$((\xi_i)_{i\in\{1,2,3\}},B)$ on~$[0,T]$, cf.\ Definition~\ref{def:gradientFlowCalibration}.
Moreover, the smoothly evolving regular double bubble $(\bar\Omega_1,\bar\Omega_2,\bar\Omega_3)$
admits a family of transported weights~$(\vartheta_i)_{i\in\{1,2,3\}}$
on~$[0,T]$ in the sense of Definition~\ref{prop:existenceWeights}.

As a corollary, we obtain a weak-strong uniqueness and stability of evolutions
principle for \emph{BV}~solutions $(\Omega_1,\Omega_2,\Omega_3)$
to multiphase~MCF on $[0,T]$ 
(cf.\ \cite[Definition~13]{FischerHenselLauxSimon2D})
with respect to the class of regular double bubbles smoothly evolving by~MCF on~$[0,T]$
in the sense of Definition~\ref{def:smoothSolution}.
We refer to Theorem~\ref{theorem:mainResult}
for a more detailed statement of this corollary,
and to the discussion right below it
for an account on the general regime of~$P\geq 3$ phases on the level of the BV~solution.
\end{theorem}

\begin{proof}
The existence of a gradient-flow calibration $((\xi_i)_{i\in\{1,2,3\}},B)$ on~$[0,T]$
is the content of Theorem~\ref{prop:existenceGradientFlowCalibration}. Its proof
occupies almost the whole paper and is carried out from Section~\ref{sec:localCalibrationInterface} 
to Section~\ref{sec:globalCalibrations}. We emphasize in this context that 
the local construction at a triple line performed in Section~\ref{sec:existenceCalibrationTripleLine}
represents the core contribution of the present work. The existence of 
transported weights~$(\vartheta_i)_{i\in\{1,2,3\}}$ on~$[0,T]$ is
proven in Section~\ref{sec:existenceWeights}
in form of Proposition~\ref{prop:existenceWeights}.

These two existence results in turn realize the assumptions of the 
general conditional weak-strong uniqueness and stability of evolutions
principle~\cite[Proposition~5]{FischerHenselLauxSimon2D}
for BV~solutions to multiphase mean curvature flow
(with respect to the setting of~$P=3$ phases and~$d=3$ dimensions),
which therefore establishes the claim of the corollary. 
\end{proof}

The results of~\cite{FischerHenselLauxSimon2D} together with Theorem~\ref{theo:mainResultsCombined}
admittedly only cover two thirds of the story concerning weak-strong uniqueness
for general clusters in~$\mathbb{R}^3$ evolving by multiphase mean curvature flow.
Indeed, one also has to allow for quadruple junctions at which four distinct phases meet 
(cf.\ the structure result on minimizer of interfacial surface energy by Taylor~\cite{Taylor1976}). 
We expect that a suitable generalization of our ideas
for the construction at a triple line
should also lead to the correct construction in the case of a quadruple junction,
and thus to a full-fledged weak-strong uniqueness result in~$\mathbb{R}^3$.

\subsection{Existence of gradient-flow calibrations}
For the sake of completeness, let us first restate
the precise definition of the concept of a gradient-flow calibration.

\begin{definition}[Gradient-flow calibration]
\label{def:gradientFlowCalibration}
Let~$T\in (0,\infty)$ be a time horizon, and let
$\sigma\in\Rd[P{\times}P]$ be an admissible matrix of surface
tensions, cf.\ Remark \ref{remark:surface tensions},
for~$P\geq 2$ phases.
Moreover, let $(\bar\Omega_1,\ldots,\bar\Omega_P)$ be an evolving partition of finite
interface energy on~$\Rd{\times}[0,T]$ in the sense of Definition~\ref{def:partition}
in dimension~$d\geq 2$, and denote by~$\bigcup_{i\neq j}\barI_{i,j}$
the associated network of evolving interfaces.

A tuple of vector fields
\begin{align*}
(\xi_i)_{i\in\{1,\ldots,P\}} &\colon \Rd{\times}[0,T] \to (\Rd)^P,
\\
B &\colon \Rd{\times}[0,T] \to \Rd
\end{align*}
is called a \emph{calibration for the $L^2$-gradient flow 
of the interface energy~\eqref{def:energyFunctional} on~$[0,T]$} with 
respect to the evolving partition $(\bar\Omega_1,\ldots,\bar\Omega_P)$---or 
in short a \emph{gradient-flow calibration}---if 
it is subject to the following requirements:

\begin{subequations}
\begin{itemize}[leftmargin=0.7cm]
\item[i)] It holds $\xi_i,B\in C^0([0,T];C^0_{\mathrm{cpt}}(\Rd[d];\Rd[d]))$
for all $i\in\{1,\ldots,P\}$. Moreover, for each time $t\in [0,T]$, there exists
an $\mathcal{H}^{d-1}$ null set $\Gamma_t\subset\Rd[d]$ such that
for $\Gamma:=\bigcup_{t\in[0,T]}\Gamma_t{\times}\{t\}$
it holds $\xi_i \in (C^0_tC^1_x \cap C^1_tC^0_x)(\Rd[d]{\times}[0,T] \setminus \Gamma)$
for all $i\in\{1,\ldots,P\}$ and~$B\in C^0_tC^1_x(\Rd[d]{\times}[0,T] \setminus \Gamma)$. 
Finally, there exists $C>0$ such that $$\sup_{t\in[0,T]}\sup_{x\in\Rd[d]\setminus\Gamma_t}
|\nabla B(x,t)| + |\nabla\xi_i(x,t)| + |\partial_t\xi_i(x,t)| \leq C.$$

\item[ii)] For~$i,j\in\{1,\ldots,P\}$ with~$i\neq j$, define the vector field
\begin{align}
\label{Calibrations}
\xi_{i,j}:=\frac{1}{\sigma_{i,j}}(\xi_i-\xi_j) \quad\text{in } \Rd{\times}[0,T].
\end{align}
Denoting by~$\bar{\vec{n}}_{i,j}$ the unit normal vector field along the interface~$\barI_{i,j}$
(pointing from the $i$\emph{th} into the $j$\emph{th} phase), it is then required that
\begin{align}
\label{ExtensionProperty}
\xi_{i,j} = \bar{\vec{n}}_{i,j} \quad\text{along } \barI_{i,j}.
\end{align}
Moreover, there exists~$c\in (0,1)$ such that
a coercivity estimate in terms of the length of the vector field~$\xi_{i,j}$
holds true:
\begin{align}
\label{LengthControlXi}
|\xi_{i,j}(x,t)|\leq 1- c \min\{\dist^2(x,{\bar{I}}_{i,j}(t)),1\},
\quad (x,t) \in \Rd{\times}[0,T].
\end{align}

\item[iii)] The vector field~$B$ represents a velocity field for the partition
$(\bar\Omega_1,\ldots,\bar\Omega_P)$  in the sense that the following
two approximate evolution equations hold true for the vector fields~$\xi_{i,j}$,
$i,j\in\{1,\ldots,P\}$ with~$i\neq j$,
\begin{align}
\label{TransportEquationXi}
\big|\partial_t \xi_{i,j} {+} (B\cdot \nabla) \xi_{i,j} {+} (\nabla B)^{\mathsf{T}} \xi_{i,j} \big|(x,t)
&\leq C \min\{\dist(x,{\bar{I}}_{i,j}(t)),1\},
\\
\label{LengthConservation}
\big|\partial_t |\xi_{i,j}|^2 + (B\cdot \nabla) |\xi_{i,j}|^2 \big|(x,t)
&\leq C \min\{\dist^2(x,{\bar{I}}_{i,j}(t)),1\},
\end{align}
for some $C>0$ and all $(x,t)\in\Rd{\times}[0,T]$.

\item[iv)] The velocity~$B$ represents motion by multiphase mean curvature (i.e.,
the $L^2$-gradient flow with respect to the interface energy~\eqref{def:energyFunctional}) in the sense that
there exists a constant $C>0$ such that
\begin{align}
\label{Dissip}
\big|\xi_{i,j}\cdot B+\nabla \cdot \xi_{i,j}\big|
\leq
C \min\{\dist(x,{\bar{I}}_{i,j}(t)),1\},
\quad (x,t) \in \Rd{\times}[0,T].
\end{align}
\end{itemize}
\end{subequations}

If a gradient-flow calibration exists, we say that the evolving
partition $(\bar\Omega_1,\ldots,\bar\Omega_P)$ is \emph{calibrated on~$[0,T]$}.
\end{definition}

Note that the required regularity from the first item of the above 
definition is slightly less than what is actually stated 
in~\cite[Definition~2]{FischerHenselLauxSimon2D}. However,
it is easy to see that this regularity is still sufficient to ensure
the validity of~\cite[Theorem~3]{FischerHenselLauxSimon2D}.

The main result of the present work is now that
any sufficiently regular and smoothly evolving
double bubble admits an associated gradient-flow calibration.

\begin{theorem}[Existence of gradient-flow calibrations]
\label{prop:existenceGradientFlowCalibration}
Let~$T\in (0,\infty)$, let $\sigma\in \Rd[3\times 3]$ be an admissible matrix of surface tensions, and let
$(\bar\Omega_1,\bar\Omega_2,\bar\Omega_3)$ be a regular 
double bubble smoothly evolving by~MCF on~$[0,T]$ in the sense of 
Definition~\ref{def:smoothSolution}. 
Then $(\bar\Omega_1,\bar\Omega_2,\bar\Omega_3)$
is calibrated on~$[0,T]$ in the sense of Definition~\ref{def:gradientFlowCalibration}.
\end{theorem}

It turns out that the existence of a gradient-flow calibration 
already implies a quantitative inclusion principle for the 
surface cluster of general BV~solutions to multiphase mean curvature flow,
see~\cite[Theorem~3]{FischerHenselLauxSimon2D}. More precisely,
if at the initial time each interface of a BV~solution is contained
in the corresponding interface of a calibrated flow, then this
inclusion property remains to be satisfied as long as the calibrated flow exists.
Furthermore, this qualitative property is in fact a consequence of
a quantitative stability estimate for the interface error between
a general BV~solution and a calibrated flow (formulated in terms
of an error functional of the form~\eqref{eq:defInterfaceError} below).

The inclusion principle, however, is of course consistent with
the vanishing of a phase in the BV~solution, so that
weak-strong uniqueness cannot be derived by means of a
gradient-flow calibration alone. In order to get a control on the
bulk deviations of the phases, one relies on an additional input
which can be formalized as follows.

\begin{definition}[Family of transported weights]
\label{def:weights}
Let~$T\in (0,\infty)$ be a time horizon, and let
$\sigma\in\Rd[P{\times}P]$ be an admissible matrix of surface tensions
satisfying the strict triangle inequality
for~$P\geq 2$ phases. Let~$d\geq 2$, and let $(\bar\Omega_1,\ldots,\bar\Omega_P)$ be an evolving partition of finite
interface energy on~$\Rd{\times}[0,T]$ in the sense of Definition~\ref{def:partition}, 
and denote by~$(\bar\chi_1,\ldots,\bar\chi_P)$ the associated
family of indicator functions. We then in addition assume that
the measure~$\partial_t\bar\chi_i$ is absolutely continuous with 
respect to the measure~$|\nabla\bar\chi_i|$, and that~$\partial\bar\Omega_i(\cdot,t)$
is Lipschitz regular for all~$t\in [0,T]$. Consider finally a
velocity vector field~$B\in C^0([0,T];C^1_{\mathrm{cpt}}(\Rd[d];\Rd[d]))$.

A map~$\vartheta=(\vartheta_i)_{i\in\{1,\ldots,P\}}\colon\Rd\times[0,T]\to[-1,1]^P$
is called a \emph{family of transported weights for~$((\bar\Omega_1,\ldots,\bar\Omega_P),B)$} if it
satisfies the following list of properties:
\begin{itemize}[leftmargin=0.7cm]
\item[i)] In terms of regularity, we require $\vartheta_i\in (W^{1,1} \cap W^{1,\infty})(\Rd{\times}[0,T];[-1,1])$
					for all~$i\in\{1,\ldots,P\}$.
\item[ii)] We require that~$\vartheta_i(\cdot,t)=0$ on~$\partial\bar\Omega_i(t)$,
					 and $\vartheta_i(\cdot,t)>0$ in the essential exterior resp.\ $\vartheta_i(\cdot,t)<0$
					 in the essential interior of~$\bar\Omega_i(\cdot,t)$ for all~$i\in\{1,\ldots,P\}$
					 and all~$t\in [0,T]$.
\item[iii)] Each weight is approximately advected by the velocity~$B$ in form of
						\begin{align}
						\label{eq:advectionWeights}
						|\partial_t\vartheta_i + (B\cdot\nabla)\vartheta_i| \leq C|\vartheta_i|
						\quad\text{on } \Rd{\times} [0,T],\,i\in\{1,\ldots,P\}.
						\end{align}
\end{itemize}
\end{definition}

The existence of a family of transported weights is precisely what is
needed to derive a quantitative stability estimate for the bulk
error between a general BV~solution and a calibrated flow
(formulated in terms of an error functional 
of the form~\eqref{eq:defBulkError} below), which
together with the already mentioned quantitative inclusion principle
then implies a weak-strong uniqueness principle
for BV~solutions of multiphase mean curvature flow,
see~\cite[Proposition~5]{FischerHenselLauxSimon2D}.

It is therefore of interest to extend the 2D~existence result
from~\cite{FischerHenselLauxSimon2D} to the  3D~setting of 
any sufficiently regular and smoothly evolving double bubble.

\begin{proposition}[Existence of a family of transported weights]
\label{prop:existenceWeights}
Let~$T\in (0,\infty)$ be a time horizon, and let
$(\bar\Omega_1,\bar\Omega_2,\bar\Omega_3)$ be a regular 
double bubble smoothly evolving by~MCF on~$[0,T]$ in the sense of 
Definition~\ref{def:smoothSolution}. Let~$B$ denote the
velocity field from the gradient-flow calibration on~$[0,T]$
associated with $(\bar\Omega_1,\bar\Omega_2,\bar\Omega_3)$,
whose existence in turn is guaranteed by
Theorem~\ref{prop:existenceGradientFlowCalibration}.
Then there exists an associated family of transported weights~$(\vartheta_i)_{i\in\{1,2,3\}}$
on~$[0,T]$ with respect to the data $((\bar\Omega_1,\bar\Omega_2,\bar\Omega_3),B)$
in the precise sense of Definition~\ref{def:weights}.
\end{proposition}

\subsection{Weak-strong uniqueness and stability of evolutions}
Combining Theorem~\ref{prop:existenceGradientFlowCalibration} 
and Proposition~\ref{prop:existenceWeights} with the conditional stability 
of any calibrated MCF in arbitrary dimensions~\cite[Proposition~5]{FischerHenselLauxSimon2D}, 
we obtain the following weak-strong uniqueness principle for
distributional (i.e., BV) solutions to multiphase~MCF in three dimensions.

\begin{theorem}[Weak-strong uniqueness and quantitative stability]
	\label{theorem:mainResult}
Let~$T\in (0,\infty)$ be a time horizon, $d=3$, $P=3$,
and~$\sigma\in\Rd[3{\times}3]$ be a surface tension matrix
satisfying the strict triangle inequality. Let
$\bar\Omega=(\bar\Omega_1,\bar\Omega_2,\bar\Omega_3)$ be a regular 
double bubble smoothly evolving by~MCF on~$[0,T]$ in the sense of 
Definition~\ref{def:smoothSolution} (with respect to~$\sigma$), 
and let $\Omega=(\Omega_1,\Omega_2,\Omega_3)$ be a $BV$ solution
to multiphase~MCF in the sense of~\cite[Definition~13]{FischerHenselLauxSimon2D}
(again with respect to~$\sigma$).

If the initial conditions of the regular double bubble and the BV~solution coincide,
then the solutions also coincide for later times on~$[0,T]$. More precisely, 
\begin{align*}
&\mathcal{L}^3\big(\big(\Omega_i(0) \setminus \bar\Omega_i(0)\big)
\cup \big(\bar\Omega_i(0) \setminus \Omega_i(0)\big)\big) 
= 0 \text{ for all } i\in\{1,2,3\}
\\
&\Rightarrow \,\,
\mathcal{L}^3\big(\big(\Omega_i(t) \setminus \bar\Omega_i(t)\big)
\cup \big(\bar\Omega_i(t) \setminus \Omega_i(t)\big)\big)  = 0
\text{ for a.e.\ } t\in[0,T] \text{ and all } i\in\{1,2,3\}.
\end{align*}

Moreover, we have quantitative stability estimates in the following sense.
Denote by~$(\xi:=(\xi_i)_{i\in\{1,2,3\}},B)$ the gradient-flow calibration
on~$[0,T]$ from Theorem~\ref{prop:existenceGradientFlowCalibration}
with respect to $(\bar\Omega_1,\bar\Omega_2,\bar\Omega_3)$,
and denote by~$(\vartheta_i)_{i\in\{1,2,3\}}$ the corresponding family of transported
weights on~$[0,T]$ from Proposition~\ref{prop:existenceWeights}.
Let~$\vec{n}_{i,j}(\cdot,t)$ be the measure theoretic unit normal
along the interface $\partial^*\Omega_i(t)\cap\partial^*\Omega_j(t)$
pointing from~$\Omega_i(t)$ into~$\Omega_j(t)$, $t\in [0,T]$.
Then, the error functionals defined for all~$t\in [0,T]$ by
\begin{align}
\label{eq:defInterfaceError}
E[\Omega|\xi](t)
&:= \sum_{i,j\in\{1, 2, 3\},\,i\neq j} \sigma_{i,j}
\int_{\partial^*\Omega_i(t)\cap\partial^*\Omega_j(t)}
1{-}\vec{n}_{i,j}(\cdot,t)\cdot\xi_{i,j}(\cdot,t)
\,\mathrm{d}\mathcal{H}^2,
\\
\label{eq:defBulkError}
E[\Omega|\bar\Omega](t)
&:= \sum_{i=1}^{3} \int_{(\Omega_i(t) \setminus \bar\Omega_i(t))
\cup (\bar\Omega_i(t) \setminus \Omega_i(t))}
|\vartheta_i(\cdot,t)| \,\mathrm{d}x
\end{align}
satisfy the stability estimates
\begin{align}
\label{eq:stabilityEstimateInterfaceError}
E[\Omega|\xi](t) &\leq E[\Omega|\xi](0)e^{Ct},
\\
\label{eq:stabilityEstimateBulkError}
E[\Omega|\bar\Omega](t) &\leq \big(E[\Omega|\xi](0)
{+}E[\Omega|\bar\Omega](0)\big) e^{Ct}
\end{align}
for almost every~$t\in [0,T]$.
The constant~$C>0$ in these estimates depends only on the data of the 
smoothly evolving regular double bubble~$(\bar\Omega_1,\bar\Omega_2,\bar\Omega_3)$ on~$[0,T]$
through the explicit constructions~$((\xi_i)_{i\in\{1,2,3\}},B)$ and~$(\vartheta_i)_{i\in\{1,2,3\}}$.
\end{theorem}

\begin{proof}
As mentioned above, this is a straightforward application of
Theorem~\ref{prop:existenceGradientFlowCalibration}, 
Proposition~\ref{prop:existenceWeights} and~\cite[Proposition~5]{FischerHenselLauxSimon2D}.
\end{proof}

\begin{remark}[Admissible surface tensions]\label{remark:surface tensions}
	Let us briefly comment on the matrix of surface tensions $\sigma\in \Rd[P\times P]$. 
	We say $\sigma$ is admissible if it satisfies precisely the assumption in~\cite[Definition~9]{FischerHenselLauxSimon2D}. 
	More concretely, we require that there exists a non-degenerate $(P{-}1)$-simplex $(q_1,\ldots,q_P)$
	in~$\Rd[P{-}1]$ which represents
	the surface tensions in form of~$\sigma_{i,j}=|q_i{-}q_j|$. 
	
	In the framework of the present paper, i.e., the case $P=3$, this is equivalent to the strict triangle inequality
	\begin{align}\label{eq:triangle inequality surface tensions}
		\sigma_{i,j} < \sigma_{i,k} + \sigma_{k,j} \quad \text{for all choices } \{i,j,k\} = \{1,2,3\}.
	\end{align}
	
	In the general case $P\geq 3$, the $\ell^2$-embeddability is in 
	fact stronger than \eqref{eq:triangle inequality surface tensions}, and it
	constitutes the key ingredient to construct the missing calibration vector 
	fields $(\xi_i)_{i\in\{4,\ldots,P\}}$, for which one may in fact argue
	along the same lines as in the proof of~\cite[Lemma~35]{FischerHenselLauxSimon2D}
	without requiring any additional ingredients from the constructions performed in this work.
\end{remark}

We emphasize that only for simplicity, we considered in Theorem~\ref{theorem:mainResult} 
the case of~$P=3$ phases on the level of the \emph{BV}~solution.
Let us briefly outline the additional ingredients which are needed
to establish the stability estimates~\eqref{eq:stabilityEstimateInterfaceError}
and~\eqref{eq:stabilityEstimateBulkError} in terms of general \emph{BV} solutions
$(\Omega_1,\ldots,\Omega_P)$, $P > 3$, defined on~$\Rd[3]{\times}[0,T]$
with respect to a given $\ell^2$-embeddable matrix of surface tensions~$\sigma=(\sigma_{i,j})_{i,j\in\{1,\ldots,P\}}\in\Rd[P{\times}P]$,
and a fixed regular double bubble~$(\bar\Omega_1,\bar\Omega_2,\bar\Omega_3)$
smoothly evolving by~MCF with respect to the restriction~$(\sigma_{i,j})_{i,j\in\{1,2,3\}}$
of the surface tension matrix~$\sigma\in\Rd[P{\times}P]$.

Recalling the definitions~\eqref{eq:defInterfaceError}
and~\eqref{eq:defBulkError}
of the error functionals (in which one only needs to replace $3$ by $P$ in the case $P>3$), it is clear that we have to augment
the gradient-flow calibration provided by Theorem~\ref{prop:existenceGradientFlowCalibration}
with additional calibrating vector fields~$(\xi_i)_{i\in\{4,\ldots,P\}}$,
and the family of transported weights by Proposition~\ref{prop:existenceWeights}
with additional weights~$(\vartheta_i)_{i\in\{4,\ldots,P\}}$,
such that the resulting augmented families adhere to
the requirements of Definition~\ref{def:gradientFlowCalibration} and Definition~\ref{def:weights}, respectively,
in order to allow for the desired application of~\cite[Proposition~5]{FischerHenselLauxSimon2D}.
For consistency with our definitions, let us interpret to this end the smoothly evolving regular double bubble
as a partition~$(\bar\Omega_1,\ldots,\bar\Omega_P)$ with the convention that $\bar\Omega_i:=\emptyset$
for all additional phases $i\in\{4,\ldots,P\}$ in the BV~solution.

Extending the family of transported weights is trivial since
we may define~$\vartheta_i:=1$ for all $i\in\{4,\ldots,P\}$, thus
representing consistently the fact that the additional phases
on the level of the smoothly evolving regular double bubble are empty.

Furthermore, the missing calibration vector fields
can be constructed along the lines of the proof 
of~\cite[Lemma~35]{FischerHenselLauxSimon2D}.
It is then straightforward 
that the associated additional vector fields
\begin{align*}
\sigma_{i,j}\xi_{i,j} := \xi_i - \xi_j,
\quad i\in \{4,\ldots,P\} \text{ or } j \in \{4,\ldots,P\},
\end{align*}
satisfy~\eqref{ExtensionProperty}--\eqref{Dissip} (together with the desired regularity).
Indeed, except for the coercivity condition~\eqref{LengthControlXi},
all these properties are trivially satisfied in terms of the
relevant additional pairs of indices since the associated interfaces on 
the level of the smoothly evolving regular double bubble are empty.
With respect to~\eqref{LengthControlXi}, the proof 
of~\cite[Lemma~37]{FischerHenselLauxSimon2D} applies verbatim
without requiring any additional ingredients from the constructions of this work.

We decided to restrict ourselves to the case~$P=3$ in the formulation of Theorem~\ref{theorem:mainResult}
because we view the main contribution of this paper to be the first part
of Theorem~\ref{theo:mainResultsCombined} (i.e., the combination of
Theorem~\ref{prop:existenceGradientFlowCalibration} and Proposition~\ref{prop:existenceWeights}),
and thus aim to shift the focus on the required $3D$~generalization of those results of~\cite{FischerHenselLauxSimon2D}
which are concerned with the given strong solution only (i.e., in the present work a
regular double bubble smoothly evolving by~MCF).

\subsection{Definition of a regular double bubble smoothly evolving by MCF}
This part is concerned with the formulation of a ``strong solution concept''
for a (topologically standard) double bubble moving by mean curvature flow,
for which we are then able to show that its flow is calibrated
in the precise sense of Definition~\ref{def:gradientFlowCalibration}.
We start with the associated energy functional. 

\begin{definition}[Partition with finite interface energy, see {\cite[Definition~12]{FischerHenselLauxSimon2D}}]
\label{def:partition}
Consider~$d\geq 2$, $P\geq 2$, and an admissible matrix of surface tensions  $\sigma\in\Rd[P{\times}P]$. 
Let $(\bar\Omega_1,\ldots,\bar\Omega_P)$ be
a family of measurable subsets of~$\Rd$ such that $\mathcal{L}^d(\Rd\setminus\bigcup_{i=1}^P\bar\Omega_i)=0$
and $\mathcal{L}^d(\bar\Omega_i\cap\bar\Omega_j)=0$ for all~$i,j\in\{1,\ldots,P\}$ with~$i\neq j$.
We then call $(\bar\Omega_1,\ldots,\bar\Omega_P)$ a \emph{partition of~$\Rd$ with finite interface
energy} if
\begin{align}
\label{def:energyFunctional}
E[(\bar\Omega_1,\ldots,\bar\Omega_P)] := \sum_{i,j\in\{1,\ldots,P\},\,i\neq j}
\sigma_{i,j} \mathcal{H}^{d-1}(\partial^*\bar\Omega_i\cap\partial^*\bar\Omega_j) < \infty.
\end{align}

Let next~$T\in (0,\infty)$ be a time horizon, and consider 
a family $(\bar\Omega_1,\ldots,\bar\Omega_P)$ of open subsets of~$\Rd{\times}[0,T]$
in the form of~$\bar\Omega_i=\bigcup_{t\in [0,T]}\bar\Omega_i(t){\times}\{t\}$
for all~$i\in\{1,\ldots,P\}$. In this evolutionary setting,
we call $(\bar\Omega_1,\ldots,\bar\Omega_P)$ an \emph{evolving partition
on~$\Rd{\times}[0,T]$ with finite interface energy}, if for all~$t\in [0,T]$
the family~$(\bar\Omega_1(t),\ldots,\bar\Omega_P(t))$ is a partition of~$\Rd$ with
finite interface energy in the above sense and it holds
\begin{align}
\sup_{t\in[0,T]} E[(\bar\Omega_1(t),\ldots,\bar\Omega_P(t))] < \infty.
\end{align}

For such an evolving partition, we denote the associated evolving
interfaces by $\barI_{i,j}:=\bigcup_{t\in [0,T]}\barI_{i,j}(t){\times}(t)$,
where~$\barI_{i,j}(t):=\partial^*\bar\Omega_i(t)\cap\partial^*\bar\Omega_j(t)$
for all~$t\in [0,T]$ and all pairs~$i,j\in\{1,\ldots,P\}$, $i\neq j$.
\end{definition}

In a next step, we formalize the topological setup as well as the main regularity
assumptions. We also state the main compatibility condition in form of the Herring angle condition.

\begin{definition}[Regular double bubble]
\label{def:tripleLineCluster}
Let $\sigma\in\Rd[3{\times}3]$ be an admissible matrix of surface tensions,
and consider a partition $(\bar{\Omega}_1,\bar\Omega_2,\bar{\Omega}_3)$ of~$\Rd[3]$ with finite interface 
energy in the sense of Definition~\ref{def:partition}. Assume in addition that~$\bar\Omega_i$
is an open, non-empty and simply connected subset of~$\Rd[3]$ such that $ \partial \bar \Omega_i$ 
is the closure of $\partial^* \bar \Omega_i$
for all $i\in\{1,2,3\}$. Define then for each~$i,j\in\{1,2,3\}$ with~$i\neq j$
the associated interface ${\bar{I}}_{i,j}:=\partial \bar{\Omega}_i\cap\partial \bar{\Omega}_j$,
which is assumed to be non-empty.

We call $(\bar\Omega_1,\bar\Omega_2,\bar\Omega_3)$ a \emph{regular double bubble} 
if the following additional regularity conditions are satisfied:
\begin{itemize}[leftmargin=0.7cm]
	\item[i)] Each interface~${\bar{I}}_{i,j}$ is a two-dimensional, compact and simply connected manifold 
	with boundary of class~$C^5$. The interior of each interface is embedded.
	\item[ii)] The three interfaces ${\bar{I}}_{1,2}$, ${\bar{I}}_{2,3}$, and ${\bar{I}}_{3,1}$ 
	intersect precisely along their respective boundary, which in turn is a non-empty one-dimensional, compact and 
	connected manifold~$\trLine$ without boundary of class~$C^5$. 
	\item[iii)] Along the triple line~$\trLine$, the Herring angle condition has to be satisfied:
	\begin{align}\label{HerringAngleCondition}
	\sigma_{1,2}\bar{\vec{n}}_{1,2} + \sigma_{2,3}\bar{\vec{n}}_{2,3} + \sigma_{3,1}\bar{\vec{n}}_{3,1} = 0,
	\end{align}
	where we denote by~$\bar{\vec{n}}_{i,j}$ the associated unit normal vector field along~$\bar I_{i,j}$
	pointing from~$\bar\Omega_i$ into~$\bar\Omega_j$.
	%
\end{itemize}
\end{definition}

With the notion of a regular double bubble in place, we finally clarify
what we mean by a (sufficiently) smooth evolution of a regular double bubble
with respect to mean curvature flow. It turns out that the construction
of an associated gradient-flow calibration in the vicinity of the
evolving triple line requires two additional higher-order
compatibility conditions. For a sufficiently smooth evolution
of a regular double bubble, these two compatibility conditions
are consequences of differentiating in time the assumed zeroth
order compatibility condition (i.e., the triple line being the 
common boundary of the three interfaces) 
or first order compatibility condition
(i.e., the Herring angle condition), respectively.
Since we will require regularity down to time $t=0$,
we have to include the resulting compatibility conditions
for the initial double bubble.

\begin{definition}[Regular double bubble smoothly evolving by~MCF]
\label{def:smoothSolution}
Let $\sigma\in\Rd[3{\times}3]$
be an admissible matrix of surface tensions.
Consider an associated initial partition~$(\bar\Omega_1^0,\bar\Omega_2^0,\bar\Omega_3^0)$
of~$\Rd[3]$ representing a regular double bubble in the sense of Definition~\ref{def:tripleLineCluster}.
Assume in addition that~$(\bar\Omega_1^0,\bar\Omega_2^0,\bar\Omega_3^0)$ satisfies
the following two higher-order compatibility conditions:

First, we require for the scalar mean curvatures in form of~$H^0_{i,j}:=-\nabla^\mathrm{tan}\cdot\no_{i,j}^0$ that
along the initial triple line~$\trLine^0$ it holds
\begin{align}
\label{InitialSecondOrderCompatibility}
\sigma_{1,2}H^0_{1,2} + \sigma_{2,3}H^0_{2,3} + \sigma_{3,1}H^0_{3,1} &= 0,
\end{align}
which by~\eqref{HerringAngleCondition} is equivalent to the existence
of a unique vector field~$\vec{V}_{\trLine^0}$ along~$\trLine^0$, which takes
values in the normal bundle~$\mathrm{Tan}^\perp\trLine^0$ such that
\begin{align*}
\bar{\vec{n}}^0_{i,j}\cdot\vec{V}_{\trLine^0} = H^0_{i,j}
\quad\text{along } \trLine^0 \text{ for all } i,j\in\{1,2,3\} \text{ with } i\neq j.
\end{align*}

Second, denoting by~$\taTrJ^{0}$ a unit length tangent vector field
along the initial triple line~$\trLine^0$ and defining~$\ta^0_{i,j}:=\no^0_{i,j}{\times}\taTrJ^0$
along~$\trLine^0$ for all~$i,j\in\{1,2,3\}$ with~$i\neq j$,
we require that along~$\trLine^0$ the quantity
\begin{align}
\label{InitialThirdOrderCompatibility}
-\big(\ta^0_{i,j}\otimes\ta^0_{i,j}:\nabla^\mathrm{tan} \no_{i,j}^0\big)
\big(\ta^0_{i,j}\cdot\vec{V}_{\trLine^0}\big)
+ (\ta^0_{i,j}\cdot\nabla^\mathrm{tan}) H^0_{i,j}
\end{align}
is independent of the choice of distinct~$i,j\in\{1,2,3\}$ at each point on~$\trLine^0$.

Let now~$T\in (0,\infty)$ be a time horizon, and consider
an evolving partition $(\bar\Omega_1,\bar\Omega_2,\bar\Omega_3)$
on~$\Rd[3]{\times}[0,T]$ with finite interface energy in the sense of Definition~\ref{def:partition}.
We call $(\bar\Omega_1,\bar\Omega_2,\bar\Omega_3)$ a \emph{regular double bubble
smoothly evolving by~MCF on~$[0,T]$} with initial data~$(\bar\Omega_1^0,\bar\Omega_2^0,\bar\Omega_3^0)$
if it satisfies:
\begin{itemize}[leftmargin=0.7cm]
\item[i)] For each~$t\in [0,T]$, the family $(\bar\Omega_1(t),\bar\Omega_2(t),\bar\Omega_3(t))$
					is a regular double bubble in the sense of Definition~\ref{def:tripleLineCluster}.
					Furthermore, the initial condition is attained in the sense that
					$(\bar\Omega_1(0),\bar\Omega_2(0),\bar\Omega_3(0))=(\bar\Omega_1^0,\bar\Omega_2^0,\bar\Omega_3^0)$.
\item[ii)] There exists a family of diffeomorphisms $\psi^t\colon\Rd[3]\to\Rd[3]$, $t\in [0,T]$,
					 such that it holds~$\psi^0(x)=x$ for all~$x\in\Rd[3]$, and $\bar\Omega_i(t)=\psi^t(\bar\Omega_i^0)$
					 as well as $\barI_{i,j}(t)=\psi^t(\barI_{i,j}^0)$ for all~$t\in [0,T]$
					 and all~$i,j\in\{1,2,3\}$, $i\neq j$. In addition, the map
					 \begin{align*}
					 \psi_{i,j}\colon\barI_{i,j}^0{\times}[0,T] \to \barI_{i,j},
					 \quad (x,t) \mapsto (\psi^t(x),t)
					 \end{align*}
					 is a diffeomorphism of class $(C^0_tC^5_x\cap C^1_tC^3_x)(\barI_{i,j}^0{\times}[0,T])$.
\item[iii)] For each $i,j\in\{1,2,3\}$ with $i\neq j$	and each~$(x,t)\in\barI_{i,j}$
						denote by~$\vec{V}_{\barI_{i,j}}(x,t)$ the normal velocity vector of~$\barI_{i,j}(t)$
						at~$x\in\barI_{i,j}(t)$. We then require motion by~MCF for each interface, i.e., 
						\begin{align}
						\label{NormalVelocityAndMeanCurvatureOnEachInterface}
						\big(\no_{i,j}\cdot\vec{V}_{\barI_{i,j}}\big)(x,t) = H_{i,j}(x,t),
						\quad (x,t) \in \barI_{i,j},\, i,j\in\{1,2,3\} \text{ with } i\neq j.
						\end{align}
\end{itemize}
\end{definition}

\begin{remark}
	The existence of regular double bubbles smoothly evolving by MCF according to Definition \ref{def:smoothSolution} is basically contained in the work of Depner, Garcke, and Kohsaka \cite{DepnerGarckeKohsaka}. 
	However, here we ask for higher regularity of the interfaces $\barI_{i,j}$ and higher compatibility along the triple line $\trLine$.
	The estimate away from the triple line corresponds to standard interior Schauder estimates while the  behavior close to the triple line can be achieved by an adaptation of the argument in \cite{DepnerGarckeKohsaka} similar to the one by Garcke and G\"{o}\ss wein \cite{GarckeGoesswein}, who derived these higher-order compatibility conditions in the context of the surface diffusion flow, a fourth-order geometric evolution equation.
	Of course one needs to assume that the initial conditions satisfy the corresponding regularity and the higher compatibility along the triple line $\trLine^0$. 
	To then generalize the construction, the main challenge is the adaptation of the linearized problem, where one now adds these higher-order compatibility conditions to the operator ${B}^{i,j}$ and therefore to the vector $(b^1,b^2,b^3)$ in \cite[Eq.~(27)]{DepnerGarckeKohsaka}, say, as two components $b^{4},\,b^{5}$.
	Then one can check the Lopatinskii-Shapiro conditions \cite[Lemma 3]{DepnerGarckeKohsaka} as we only added more rows to the output. 
\end{remark}

\subsection{Notation}
We briefly review the standard notation employed throughout the present work.
The notation of geometric quantities will be introduced in the course of the paper.

We write~$\mathcal{L}^d$ for the $d$-dimensional Lebesgue measure,
$\mathcal{H}^{s}$ for the $s$-dimensional Hausdorff measure,
as well as~$\partial^*D$ for the reduced boundary of a set of finite perimeter.
The standard Lebesgue spaces with respect to the Lebesgue measure
are denoted as always by~$L^p(D)$ for any~$p\in [0,\infty]$
and any measurable~$D\subset\Rd$, whereas in addition for any~$k\in\mathbb{N}$
we denote by~$W^{k,p}(D)$ the standard Sobolev space.
We further write~$C^k(D)$, $k\geq 0$, for the space of functions
with bounded and continuous derivatives up to order~$k$ on~$D\subset\Rd$. The intersection
with the space~$C^0_{\mathrm{cpt}}(D)$ of continuous and compactly supported functions on~$D$
is denoted by~$C^k_{\mathrm{cpt}}(D)$. Vector-valued versions of these function spaces
are denoted by~$L^p(D;\Rd)$ and so on. For a differentiable function~$f\colon D\to\Rd[m]$
we write~$\nabla f\in\Rd[m{\times}d]$ for its Jacobian matrix, i.e., it holds $(\nabla f)_{i,j}=\partial_j f_i$.
If $f\colon M\to\Rd[m]$ is a differentiable function along a given $C^1$ manifold $M$,
we denote by~$\nabla^\mathrm{tan}$ its tangential gradient.

For a space-time domain $D\subset\Rd{\times}[0,T]$
of the form $D=\bigcup_{t\in [0,T]} D(t){\times}\{t\}$ we write $C^l_tC^k_x(D)$, $l,k\geq 0$,
for the space of continuous functions~$f$ on~$D$ with continuous and bounded
partial derivatives~$\partial_t^{l'}\partial_x^{k'}f$ on~$D$ for
any~$0\leq l'\leq l$ and any multi-index~$k'$ such that~$0\leq |k'|\leq k$.
With a slight abuse of notation, the distance function~$\dist(\cdot,D)$ with respect to such a space-time domain~$D$
is understood as he distance to the corresponding time slice, i.e.,~$(x,t)\mapsto\dist(x,D(t))$ 
for all~$(x,t)\in \Rd{\times}[0,T]$.

In terms of vector and tensor notation, we denote by~$v{\times} w$ the cross product
between two vectors~$v,w\in\Rd[3]$, by~$v \wedge w := v\otimes w - w\otimes v$
the exterior product of~$v,w\in\Rd[3]$, and by~$A:B:= \sum_{i,j}A_{i,j}B_{i,j}$
the complete contraction of two matrices~$A,B\in\Rd[m{\times}n]$.
Abusing notation we will also write~$a\wedge b:=\min\{a,b\}$
for the minimum of two numbers~$a,b\in\Rd[]$; however, it will always be
perfectly clear from the context what the symbol~$\wedge$ represents.
We also occasionally use~$a\vee b:=\max\{a,b\}$ for the maximum
of two numbers~$a,b\in\Rd[]$.

\section{Local gradient-flow calibration at a smooth interface}
\label{sec:localCalibrationInterface}
The aim of this section is to provide the local building block
of a gradient-flow calibration in the vicinity of an interface
present in a smoothly evolving double bubble. To this end, we introduce the following geometric setup.

\begin{definition}[Localization radius for interface]
\label{def:locRadiusInterface}
Let $(\bar\Omega_1,\bar\Omega_2,\bar\Omega_3)$ be a 
regular double bubble smoothly evolving by~MCF in the
sense of Definition~\ref{def:smoothSolution} on a time interval~$[0,T]$.
Fix $i,j\in\{1,2,3\}$ with $i\neq j$. We call a scale~$r_{i,j}\in (0,1]$ an \emph{admissible
localization radius for the interface~$\barI_{i,j}$} if
\begin{align}
\label{eq:diffeoInterface}
\Psi_{i,j}\colon\barI_{i,j}\times (-r_{i,j},r_{i,j})\to\Rd[3]\times [0,T],
\quad
(x,t,s) \mapsto (x {+} s\no_{i,j}(x,t),t)
\end{align}
is bijective onto its image~$\mathrm{im}(\Psi_{i,j}):=\Psi_{i,j}(\barI_{i,j}{\times}(-r_{i,j},r_{i,j}))$.
Moreover, it is required that the inverse~$\Psi^{-1}$ is a diffeomorphism of
class $(C^0_tC^4_x\cap C^1_tC^2_x)(\mathrm{im}(\Psi_{i,j}))$,
and that it splits in form of
\begin{align*}
\Psi^{-1}_{i,j}\colon \mathrm{im}(\Psi_{i,j}) \to \bar I_{i,j} \times (-r_{i,j},r_{i,j}),
\quad (x,t) \mapsto (P_{i,j}(x,t),t,s_{i,j}(x,t)),
\end{align*}
where the map~$s_{i,j}$ represents a signed distance function
(oriented by means of~$\no_{i,j}$,
i.e., $\nabla s_{i,j} = \no_{i,j}$ along the interface~$\bar I_{i,j}$)
\begin{align}
\label{eq:defSignedDistance}
s_{i,j}(x,t) = \begin{cases}
							 \phantom{+}\dist(x,\bar I_{i,j}(t)) & \text{if } (x,t) \in \Psi_{i,j}(\bar I_{i,j}{\times} [0,r_{i,j})), 
							 \\
							 -\dist(x,\bar I_{i,j}(t)) & \text{if } (x,t) \in \Psi_{i,j}(\bar I_{i,j}{\times} (-r_{i,j},0)),
							 \end{cases}
\end{align}
and the map~$P_{i,j}$ being (in each time slice) the nearest-point projection onto~$\bar I_{i,j}$
\begin{align}
\label{eq:defProjection}
P_{i,j}(x,t) = \argmin_{y\in\bar I_{i,j}(t)} \,|y{-}x|,
\quad (x,t) \in \mathrm{im}(\Psi_{i,j}).
\end{align}
\end{definition}

In view of Definition~\ref{def:smoothSolution} of a regular double bubble smoothly evolving by MCF,
it follows from the tubular neighborhood theorem
that all interfaces admit an admissible localization radius in
the sense of Definition~\ref{def:locRadiusInterface}.

We introduce some further notation and consequences with respect to Definition~\ref{def:locRadiusInterface}.
First, the nearest-point projection onto the interface admits the representation
\begin{align}
\label{eq:representationProjection}
P_{i,j}(x,t) = x - s_{i,j}(x,t)\nabla s_{i,j}(x,t),
\quad (x,t) \in \mathrm{im}(\Psi_{i,j}).
\end{align}
Second, it holds in terms of regularity
\begin{align}
\label{eq:regProjectionSignedDistance}
s_{i,j} \in (C^0_tC^5_x\cap C^1_tC^3_x)(\mathrm{im}(\Psi_{i,j})),
\quad
P_{i,j} \in (C^0_tC^4_x\cap C^1_tC^2_x)(\mathrm{im}(\Psi_{i,j})).
\end{align} 
The scalar mean curvature of the interface~$\barI_{i,j}$
with respect to the orientation induced by~$\no_{i,j}$ is denoted by~$H_{i,j}$. 
We extend these geometric quantities away from the interface,
performing a slight abuse of notation, by means of
\begin{align}
\label{eq:extensionNormal}
\no_{i,j}\colon \mathrm{im}(\Psi_{i,j}) \to \mathbb{S}^2,
\quad &(x,t) \mapsto \nabla s_{i,j}(x,t),
\\
\label{eq:extensionCurvature}
H_{i,j}\colon \mathrm{im}(\Psi_{i,j}) \to  \mathbb{R},
\quad &(x,t) \mapsto -\Delta s_{i,j}(P_{i,j}(x,t),t).
\end{align}
Observe that these definitions immediately imply that
\begin{align}
\label{eq:regularityNormalCurvature}
\no_{i,j} \in (C^0_tC^4_x\cap C^1_tC^2_x)(\mathrm{im}(\Psi_{i,j})),
\quad
H_{i,j} \in (C^0_tC^3_x\cap C^1_tC^1_x)(\mathrm{im}(\Psi_{i,j})).
\end{align}

\begin{construction}[gradient-flow calibration along smooth interfaces]
\label{gradientFlowCalibrationInterface}
Let the assumptions and notation of Definition~\ref{def:locRadiusInterface}
be in place, and let~$\mathcal{Y}_{i,j}\colon\mathrm{im}(\Psi_{i,j})\to\mathbb{R}^3$ 
be an arbitrary vector field of class~$C^0_tC^1_x(\mathrm{im}(\Psi_{i,j}))$.
We then define a pair of vector fields
$(\xi_{i,j},B)\colon \mathrm{im}(\Psi_{i,j})\to\mathbb{S}^2\times \mathbb{R}^3$
as follows:
\begin{align}
\label{eq:localGradFlowCalibrationInterface}
\xi_{i,j} := \no_{i,j}, \quad 
B := H_{i,j}\no_{i,j} + (\mathrm{Id}{-}\no_{i,j}\otimes\no_{i,j})\mathcal{Y}_{i,j}.
\end{align}
We call~$(\xi_{i,j},B)$ a \emph{local gradient-flow calibration 
for the interface~$\barI_{i,j}$}. 
\hfill$\diamondsuit$
\end{construction}

We now register the properties of the pair of vector fields~$(\xi_{i,j},B)$,
i.e., that it satisfies locally the requirements of Definition~\ref{def:gradientFlowCalibration}
with the exception of~\eqref{LengthControlXi}. The latter will only be satisfied
once we glued together the local constructions in Section~\ref{sec:globalCalibrations} by means of a
suitable family of cutoff functions.

\begin{lemma}
\label{lemma:gradientFlowCalibrationInterface}
Let the assumptions and notation of Construction~\ref{gradientFlowCalibrationInterface}
be in place. Then it holds 
\begin{align}
\label{eq:regularityXiVelInterface}
\xi_{i,j} \in (C^0_tC^4_x\cap C^1_tC^2_x)(\mathrm{im}(\Psi_{i,j})),
\quad
B \in C^0_tC^1_x(\mathrm{im}(\Psi_{i,j})).
\end{align}
Moreover, there exists a constant~$C>0$ which depends only on
the data of the smoothly evolving regular double bubble~$(\bar\Omega_1,\bar\Omega_2,\bar\Omega_3)$ on~$[0,T]$, such that
we have throughout the space-time domain~$\mathrm{im}(\Psi_{i,j})$
\begin{align}
\label{eq:regXiInterface}
|\nabla\xi_{i,j}| + |\partial_t\xi_{i,j}| &\leq C,
\\
\label{eq:regVelocityInterface}
|B| + |\nabla B| &\leq C,
\\
\label{eq:evolEquXiInterface}
\partial_t\xi_{i,j}
+ (B\cdot\nabla)\xi_{i,j}
+ (\nabla B)^\mathsf{T}\xi_{i,j} 
&= 0,
\\
\label{eq:divConstraintXiInterface}
|\nabla\cdot \xi_{i,j} 
+ B\cdot \xi_{i,j}|
&\leq C\dist(\cdot\bar I_{i,j}),
\\
\label{eq:evolEquLengthXiInterface}
\partial_t|\xi_{i,j}|^2
+ (B\cdot\nabla)|\xi_{i,j}|^2
&= 0.	
\end{align}
\end{lemma}

\begin{proof}
The asserted regularity follows immediately from
the definitions~\eqref{eq:localGradFlowCalibrationInterface}
and the regularity~\eqref{eq:regularityNormalCurvature} of
the constituents. The equation~\eqref{eq:evolEquXiInterface} 
for the time evolution of~$\xi_{i,j}$ follows from
differentiating in the spatial variable the~PDE satisfied
by the signed distance function~$s_{i,j}$, i.e.,
\begin{align}
\label{eq:evolutionSignedDistance}
\partial_t s_{i,j}
=-H_{i,j}=-(B\cdot\nabla) s_{i,j},
\end{align}
relying in the process on 
the product rule and~$\no_{i,j}=\nabla s_{i,j}$. The divergence constraint~\eqref{eq:divConstraintXiInterface}
is a direct consequence of the definitions~\eqref{eq:extensionNormal},
\eqref{eq:extensionCurvature} and~\eqref{eq:localGradFlowCalibrationInterface}
in combination with the regularity~\eqref{eq:regProjectionSignedDistance} 
of the signed distance. Finally, equation~\eqref{eq:evolEquLengthXiInterface}
is satisfied for trivial reasons since~$\xi_{i,j}\in\mathbb{S}^2$.
\end{proof}

\section{Local gradient-flow calibration at a triple line}
\label{sec:existenceCalibrationTripleLine}
This section constitutes the core of the present work. We establish
the existence of a gradient-flow calibration in the vicinity of
the triple line for a double bubble smoothly evolving by~MCF 
in the sense of Definition~\ref{def:smoothSolution}.
The main result of this section reads as follows.

\begin{proposition}[Existence of gradient-flow calibration at triple line]
\label{prop:gradientFlowCalibrationTripleLine}
Consider a regular double bubble $(\bar\Omega_1,\bar\Omega_2,\bar\Omega_3)$ 
smoothly evolving by~MCF on a time interval~$[0,T]$ 
in the sense of Definition~\ref{def:smoothSolution}.
Let $r\in (0,1]$ be an associated admissible localization radius 
for the triple line in the sense of Definition~\ref{def:locRadius} below.
There then exists a potentially smaller radius $\hat r\in (0,r]$, only depending
on the data of the smoothly evolving regular double bubble~$(\bar\Omega_1,\bar\Omega_2,\bar\Omega_3)$ on~$[0,T]$, 
which gives rise to the following assertions:

Denote by~$\mathcal{N}_{\hat r}(\trLine):=\bigcup_{t\in[0,T]}
B_{\hat r}(\trLine(t)){\times}\{t\}$ the neighborhood if the evolving triple line. 
For all $i,j\in\{1,2,3\}$ with $i\neq j$ 
there exists a continuous local extension
\begin{align*}
\xi_{i,j}\colon\mathcal{N}_{\hat r}(\trLine)\to \overline{B_1(0)}
\end{align*}
of the unit normal vector field $\no_{i,j}|_{\bar I_{i,j}}$ of $\barI_{i,j}$,
and a continuous local extension
\begin{align*}
B\colon\mathcal{N}_{\hat r}(\trLine)\to\Rd[3]
\end{align*}
of the velocity vector field of the network $\mathcal{I}=\bigcup_{i,j\in\{1,2,3\},i\neq j}\barI_{i,j}$,
such that the pair $((\xi_{i,j})_{i,j\in\{1,2,3\},i\neq j},B)$ satisfies the following list of requirements:
\begin{itemize}[leftmargin=0.7cm]
\item[i)] For all~$i,j\in\{1,2,3\}$ with~$i\neq j$
					it holds $\xi_{i,j}\in (C^0_tC^1_x \cap C^1_tC^0_x)
					(\mathcal{N}_{\hat r}(\trLine)\setminus\trLine)$ and
					$B\in C^0_tC^1_x(\mathcal{N}_{\hat r}(\trLine)\setminus\trLine)$,
					with corresponding estimates throughout~$\mathcal{N}_{\hat r}(\trLine)\setminus\trLine$
					\begin{align}
					\label{eq:regXiTripleLine}
					|\nabla\xi_{i,j}| + |\partial_t\xi_{i,j}| &\leq C,
					\\
					\label{eq:regVelocityTripleLine}
					|B| + |\nabla B| &\leq C
					\end{align}
					for some constant~$C>0$ which depends only on the data of the
					smoothly evolving regular double bubble~$(\bar\Omega_1,\bar\Omega_2,\bar\Omega_3)$ on~$[0,T]$.
\item[ii)] We have for all~$i,j\in\{1,2,3\}$ with~$i\neq j$
					\begin{align}
					\label{eq:extensionPropertyNormal}
					\xi_{i,j}(\cdot,t) &= \no_{i,j}(\cdot,t) 
					&& \text{along } \barI_{i,j}(t)\cap B_{\hat r}(\trLine(t)),
					\\
					\label{eq:extensionPropertyVelocity}
					B(\cdot,t) &= \vec{V}_{\bar\Gamma}(\cdot,t)
					&& \text{along } \bar\Gamma(t)
					\end{align}
					for all $t\in [0,T]$, where~$\vec{V}_{\bar\Gamma}$ denotes the normal velocity
					of the triple line~$\bar\Gamma$. Moreover,
					the skew-symmetry relation $\xi_{i,j}=-\xi_{j,i}$ holds true.
\item[iii)] The Herring angle condition is satisfied in the whole space-time 
					 tubular neighborhood~$\mathcal{N}_{\hat r}(\trLine)$ of the triple line, i.e.,
					 \begin{align}
					 \label{eq:HerringXi}
					 \sigma_{1,2}\xi_{1,2} + \sigma_{2,3}\xi_{2,3} + \sigma_{3,1}\xi_{3,1} = 0 
					 \quad \text{in } \mathcal{N}_{\hat r}(\trLine).
					 \end{align}
\item[iv)] There exists a constant $C>0$, depending only on the data of the
						smoothly evolving regular double bubble~$(\bar\Omega_1,\bar\Omega_2,\bar\Omega_3)$ on~$[0,T]$, such that
						for all~$i,j\in\{1,2,3\}$ with~$i\neq j$ the estimates
						\begin{align}
						\label{eq:timeEvolutionXiTripleLine}
						|\partial_t\xi_{i,j} + (B\cdot\nabla)\xi_{i,j} + (\nabla B)^\mathsf{T}\xi_{i,j}|
						&\leq C\dist(\cdot,\barI_{i,j}),
						\\
						\label{eq:motionByMeanCurvatureTripleLine}
						|B\cdot\xi_{i,j} + \nabla\cdot\xi_{i,j}|
						&\leq C\dist(\cdot,\barI_{i,j}),
						\\
						\label{eq:timeEvolutionLengthXiTripleLine}
						\partial_t|\xi_{i,j}|^2 + (B\cdot\nabla)|\xi_{i,j}|^2
						&\leq C\dist^2(\cdot,\barI_{i,j})
						\end{align}
						hold true within~$\mathcal{N}_{\hat r}(\trLine)\setminus\trLine$.
\end{itemize}
A pair~$((\xi_{i,j})_{i,j\in\{1,2,3\},i\neq j},B)$ subject to these
conditions is called a \emph{local gradient-flow calibration at the 
triple line~$\trLine$}.
\end{proposition}

\begin{figure}
  \centering      
  \includegraphics{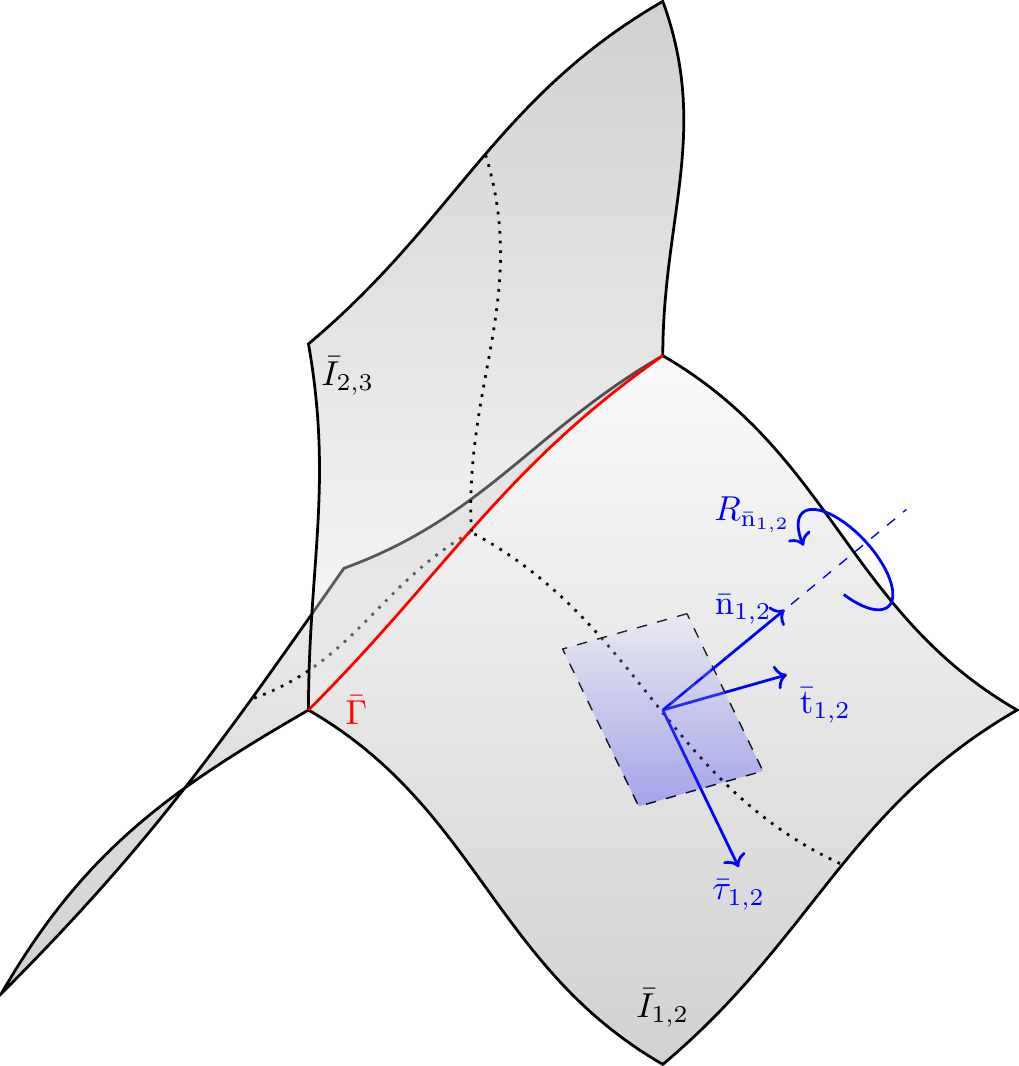}
	\caption{The smooth solution close to the triple line $\bar \Gamma$. Three sheets come together at fixed angles along $\bar \Gamma$ (here $120^\circ$). In this general situation, one needs to introduce an additional gauge rotation field $R_{ \vec{\bar n}_{1,2}}$. At each point, this matrix is a rotation in the tangent plane spanned by $\bar \tau_{1,2}$ and $ \vec{\bar t}_{1,2}$, illustrated here by a shaded (blue) rectangle.}
	\label{fig:curved triple line}
\end{figure}

The remainder of this section is organized as follows. In Subsection~\ref{sec:notationTripleLine}
we introduce the necessary notation employed in the construction of the desired vector fields.
Subsection~\ref{subsection:step_1} implements the construction of the main building blocks for
the vector fields $((\xi_{i,j})_{i\neq j},B)$, which will then be glued together in Subsection~\ref{subsection:step_2}.
Subsection~\ref{sec:proofGradientFlowCalibrationTripleLine}
contains the proof of Proposition~\ref{prop:gradientFlowCalibrationTripleLine}.
In the final Subsection~\ref{subsec:localCompatibility}, we formalize the fact
that the local gradient-flow calibration at the triple line due to Proposition~\ref{prop:gradientFlowCalibrationTripleLine}
represents an admissible perturbation of the local gradient-flow calibrations at the interfaces
in a suitable sense.

\begin{figure}
  \centering      
  \includegraphics{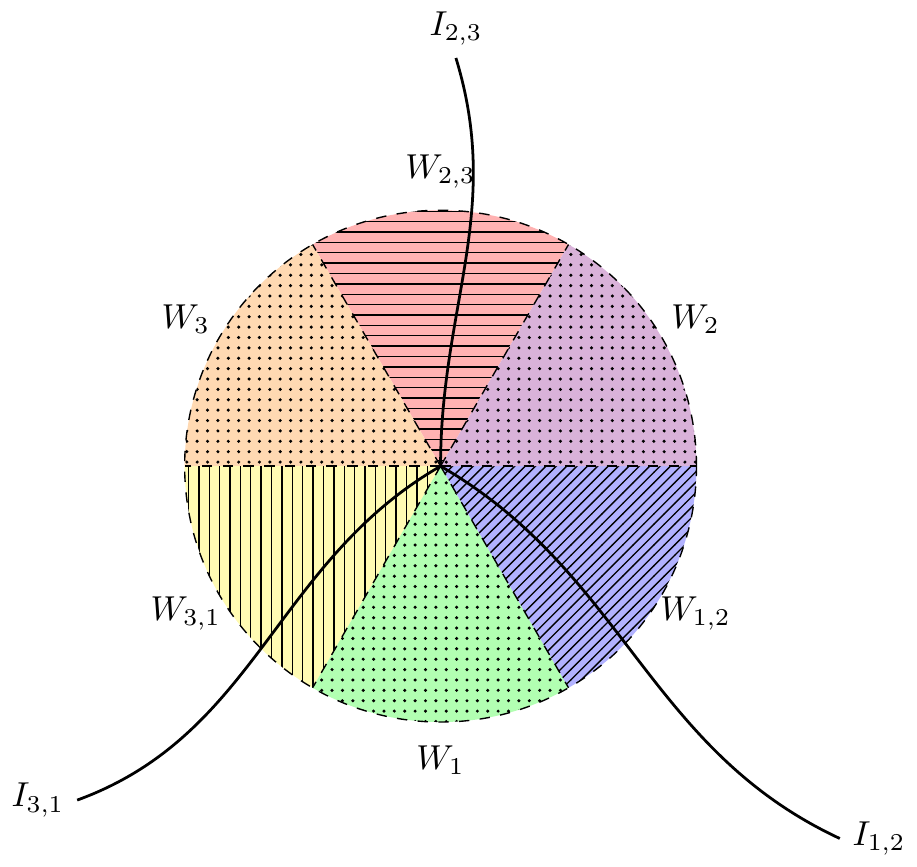}
	\caption{A cross-section othorgonal to the triple line illustrating the wedge decomposition in Definition~\ref{def:locRadius}. The ``interpolation wedges'' are marked with a dotted pattern, the ``interface wedges'' with striped patterns.}
	\label{fig:slice}
\end{figure}

\subsection{Local geometry at a triple line}
\label{sec:notationTripleLine}
We first provide a suitable decomposition of the space-time 
tubular neighborhood of the triple line of a smoothly evolving regular double bubble in 
the sense of Definition~\ref{def:smoothSolution}. The main ingredient is given by the
following notion of an admissible localization radius for the triple line, cf.\ Figure \ref{fig:slice}.

\begin{definition}[Localization radius for triple line]
\label{def:locRadius}
Let $(\bar\Omega_1,\bar\Omega_2,\bar\Omega_3)$ be a regular double 
bubble smoothly evolving by MCF in the sense of Definition~\ref{def:smoothSolution} 
on a time interval~$[0,T]$. For each~$i,j\in\{1,2,3\}$ with~$i\neq j$,
let~$r_{i,j}\in (0,1]$ be an admissible localization radius for the interface~$\barI_{i,j}$
in the sense of Definition~\ref{def:locRadiusInterface}.
We call $r=r_{\trLine}\in \big(0,\min\{r_{i,j}\colon i,j\in\{1,2,3\},\,i\neq j\}\big]$ an 
\emph{admissible localization radius for the triple line~$\trLine$} 
if the following properties are satisfied:
\begin{itemize}[leftmargin=0.7cm]
\item[i)] (Regularity of triple line) Define 
					$\mathcal{N}_{r}(\trLine):=\bigcup_{t\in [0,T]}B_{r}(\trLine(t)){\times}\{t\}$.
					The squared distance to~$\trLine$
					satisfies $\dist^2(\cdot,\trLine)\in C^0_tC^4_x(\mathcal{N}_{r}(\trLine))
					\cap C^1_tC^2_x(\mathcal{N}_{r}(\trLine))$, and for the nearest-point
					projection onto~$\trLine$ it holds $P_{\trLine}\in C^0_tC^4_x(\mathcal{N}_{r}(\trLine))
					\cap C^1_tC^2_x(\mathcal{N}_{r}(\trLine))$.
\item[ii)] (Wedge decomposition) For each~$i,j\in\{1,2,3\}$ with~$i\neq j$, there exist sets
					 $W_{\barI_{i,j}}:=\bigcup_{t\in [0,T]} W_{\barI_{i,j}(t)}{\times}\{t\}$, $W_{\barI_{j,i}}:=W_{\barI_{i,j}}$,
					 and $W_{\bar\Omega_i}:=\bigcup_{t\in [0,T]} W_{\bar\Omega_i}(t){\times}\{t\}$
					 subject to the following conditions:
					
					 First, for each~$t\in [0,T]$ the sets~$(W_{\barI_{i,j}}(t))_{i,j\in\{1,2,3\},i\neq j}$
					 and~$(W_{\bar\Omega_i}(t))_{i\in\{1,2,3\}}$ are non-empty, pairwise disjoint 
					 open subsets of~$B_{r}(\trLine(t))$. For each~$x\in\trLine(t)$, the slice of each of these
					 sets in the normal plane~$x{+}\mathrm{Tan}^\perp_x\trLine(t)$ is the intersection of $B_{r}(\trLine(t))$ and 
					 a cone with apex~$x$, cf.\ Figure \ref{fig:slice}.
					 More precisely, there exist unit length vector 
					 fields~$(X^\pm_{\barI_{i,j}})_{i,j\in\{1,2,3\},i\neq j}$ and~$(X^\pm_{\bar\Omega_i})_{i\in\{1,2,3\}}$
					 along $\trLine$, taking values for each~$t\in [0,T]$ in the normal bundle 
					 $\mathrm{Tan}^\perp\trLine(t)$
					 and being of class $C^0_tC^4_x(\trLine)\cap C^1_tC^2_x(\trLine)$,
					 so that for all~$i,j\in\{1,2,3\}$ with~$i\neq j$ and all~$(x,t)\in\trLine$ it holds
					 \begin{equation}
						\begin{aligned}
						\label{eq:defWedge}
						&W_{\barI_{i,j}}(t)\cap\big(x{+}\mathrm{Tan}_x^\perp\trLine(t)\big) 
						\\
						&= \big(x{+}\{\alpha X^+_{\barI_{i,j}}(x,t) + 
						\beta X^-_{\barI_{i,j}}(x,t)\colon\alpha,\beta\in (0,\infty)\}\big)
						\cap B_{r}(\trLine(t)),
						\end{aligned}
						\end{equation}
						as well as
						\begin{equation}
						\begin{aligned}
						\label{eq:defInterpolWedge}
						&W_{\bar\Omega_i}(t)\cap\big(x{+}\mathrm{Tan}_x^\perp\trLine(t)\big)
						\\
						&= \big(x{+}\{\alpha X^+_{\bar\Omega_i}(x,t) 
						+ \beta X^-_{\bar\Omega_i}(x,t)\colon\alpha,\beta\in (0,\infty)\}\big)
						\cap B_{r}(\trLine(t)).
						\end{aligned}
						\end{equation}
						Moreover, $X^\pm_{\barI_{i,j}}=X^\pm_{\barI_{j,i}}$ and
						$(X^+_{\bar\Omega_i},X^-_{\bar\Omega_i})\in
						\big\{(X^+_{\barI_{i,j}},X^-_{\barI_{k,i}}),(X^+_{\barI_{k,i}},X^-_{\barI_{i,j}})\big\}$
						for all $i,j,k\in\{1,2,3\}$ such that $\{i,j,k\}=\{1,2,3\}$.
						The opening angles of these cones are constant along~$\trLine$
						and take values in~$(0,\pi)$.
					
					 Second, for each~$t\in [0,T]$ these sets provide a decomposition
					 of the tubular neighborhood of the triple line in the sense that
					 	\begin{align}
						\label{eq:decompTripleLine}
						\overline{B_{r}(\trLine(t))} = \overline{W_{\barI_{1,2}}(t)} \cup
						\overline{W_{\barI_{2,3}}(t)} \cup  \overline{W_{\barI_{3,1}}(t)} \cup
						\bigcup_{i\in\{1,2,3\}}\overline{W_{\bar\Omega_i}(t)}.
						\end{align}
						
					 Third, for all~$t\in [0,T]$ and all distinct~$i,j\in\{1,2,3\}$  
					 it holds
					 	\begin{align}
						\label{eq:interfaceWedge}
						\barI_{i,j}(t)\cap B_{r}(\trLine(t))
						&\subset W_{\barI_{i,j}}(t) \cup \trLine(t)
						\subset\{x\in\mathbb{R}^3\colon (x,t)\in\mathrm{im}(\Psi_{i,j})\},
						\\
						\label{eq:interpolWedgeHalfspace}
						W_{\bar\Omega_i}(t) &\subset \bigcap_{j\in\{1,2,3\}\setminus\{i\}}
						\{x\in\mathbb{R}^3\colon (x,t)\in\mathrm{im}(\Psi_{i,j})\},
						\end{align}
						where we refer to Definition~\ref{def:locRadiusInterface} 
						for the diffeomorphisms~$\Psi_{i,j}$.
\item[iii)] (Comparability of distances) There exists~$C>0$ such that for all
						pairwise distinct $i,j,k\in\{1,2,3\}$ it holds
						(recall that~$\mathcal{I}=\bigcup_{i,j\in\{1,2,3\},i\neq j}\barI_{i,j}$)
						\begin{align}
						\label{eq:compDistances}
						\dist(\cdot,\trLine) + \dist(\cdot,\bar I_{i,j}) + \dist(\cdot,\bar I_{k,i})
						&\leq C\dist(\cdot,\mathcal{I})
						&& \text{in } W_{\bar\Omega_i},
						\\
						\label{eq:compDistances2}
						\dist(\cdot,\trLine) &\leq C\dist(\cdot,\barI_{i,j})
						&& \text{in } W_{\barI_{j,k}} \cup W_{\barI_{k,i}},
						\\
						\label{eq:compDistances3}
						\dist(\cdot,\barI_{i,j}) &\leq C\dist(\cdot,\mathcal{I})
						&& \text{in } W_{\barI_{i,j}}.
						\end{align}
\end{itemize}
We refer from here onwards to the sets~$(W_{\barI_{i,j}})_{i,j\in\{1,2,3\},i\neq j}$ 
as \emph{interface wedges}, and to the sets~$(W_{\bar\Omega_i})_{i\in\{1,2,3\}}$ as \emph{interpolation wedges}. 
\end{definition}

Equations~\eqref{eq:defWedge} and~\eqref{eq:defInterpolWedge} simply mean that the 
domains~$W_{\bar\Omega_i}(t)$ and~$W_{\bar I_{i,j}}(t)$ are ``wedges'' in the sense 
that their respective slices across the normal space~$x+\mathrm{Tan}^\perp\trLine(t)$ 
of the triple line have a cone structure close to~$\bar\Gamma(t)$.
The comparability \eqref{eq:compDistances}--\eqref{eq:compDistances3} of distance functions 
in the various slices can be already guessed from Figure \ref{fig:slice}.

Let us first briefly discuss the existence of an admissible localization radius.
\begin{lemma}
\label{lemma:existenceLocRadius}
Let the assumptions and notation of Definition~\ref{def:locRadius}
be in place. Then there exists an admissible localization
radius $r=r_{\trLine}\in (0,1]$ for the triple line. The radius~$r$ and the
associated data only depends on the data of the smoothly evolving regular
double bubble~$(\bar\Omega_1,\bar\Omega_2,\bar\Omega_3)$ on~$[0,T]$.
\end{lemma}

\begin{proof}
We provide details on how to arrange the vector fields
~$(X^\pm_{\barI_{i,j}})_{i,j\in\{1,2,3\},i\neq j}$ and~$(X^\pm_{\bar\Omega_i})_{i\in\{1,2,3\}}$
in order to ensure the properties~\eqref{eq:defWedge}--\eqref{eq:decompTripleLine}.
The remaining conditions are a consequence of exploiting the uniform space-time regularity of the
interfaces present in the smoothly evolving regular double bubble~$(\bar\Omega_1,\bar\Omega_2,\bar\Omega_3)$
on~$[0,T]$, cf.\ Definition~\ref{def:smoothSolution},
and choosing the scale~$r\in (0,r_{1,2}\wedge r_{2,3} \wedge r_{3,1}]$ sufficiently small.

Fix~$(x,t)\in \trLine$, and up to a translation and rotation we may
assume without loss of generality that~$x=0$ 
and~$\mathrm{Tan}_{x}^\perp\trLine(t) = \{0\}{\times}\Rd[2] = \langle e_2,e_3 \rangle$,
where~$\{e_1,e_2,e_3\}$ denotes the standard basis of~$\Rd[3]$
and~$\langle e_2,e_3 \rangle$ the $\Rd[]$-linear span of~$\{e_2,e_3\}$.
Denote then by~$\ta_{1,2},\ta_{2,3},\ta_{3,1}\in \langle e_2,e_3 \rangle$ the
tangent vectors at~$x=0$ to the interfaces~$\barI_{1,2}$, $\barI_{2,3}$
and~$\barI_{3,1}$, respectively, with the orientation chosen so 
that along $\bar \Gamma$ all of them point in the direction of the associated interface
(which is also described in more detail in Construction~\ref{tangentVectorFields} below).
These tangent vectors define associated half-spaces
\begin{align}
\label{eq:defHalfSpaceSlice}
\mathbb{H}_{1,2} := \big\{y \in \langle e_2,e_3 \rangle\colon
y\cdot\ta_{1,2} > 0\big\},
\end{align}
where $\mathbb{H}_{2,3}$ and~$\mathbb{H}_{3,1}$ are defined analogously.

We now construct a set of pairwise disjoint open 
cones~$C_{\bar\Omega_1},C_{\bar\Omega_2},C_{\bar\Omega_3}\subset\langle e_2,e_3\rangle$,
which will provide the cone structure of the interpolation wedges,
by means of the following procedure: If the cone given
by $\mathbb{H}_{1,2}\cap\mathbb{H}_{3,1}$ has an opening angle
strictly greater than~$90^\circ$, we define
$C_{\bar\Omega_1}:=\mathbb{H}_{1,2}\cap\mathbb{H}_{3,1}$.
In the other case, we define $C_{\bar\Omega_1}$ to be the
middle third of the cone with opening vectors~$\ta_{1,2}$
and~$\ta_{3,1}$. The remaining two cones~$C_{\bar\Omega_2}$
and~$C_{\bar\Omega_3}$ are defined in the same way.

Note that the opening angles of the cones~$(C_{\bar\Omega_i})_{i\in\{1,2,3\}}$
are always strictly smaller than~$180^\circ$ since the surface
tensions satisfy the strict triangle inequality.
We proceed by selecting 
cones~$C_{\barI_{1,2}}=:C_{\barI_{2,1}},C_{\barI_{2,3}}=:C_{\barI_{3,2}},
C_{\barI_{3,1}}=:C_{\barI_{1,3}}\subset\langle e_2,e_3\rangle$,
which are uniquely determined by the requirement that
together with $(C_{\bar\Omega_i})_{i\in\{1,2,3\}}$ they form a family
of pairwise disjoint open cones in~$\langle e_2,e_3\rangle$ such that
\begin{align}
\label{eq:tiling}
\langle e_2,e_3 \rangle &= \overline{C_{\barI_{1,2}}} \cup \overline{C_{\barI_{2,3}}}
\cup \overline{C_{\barI_{3,1}}} \cup \bigcup_{i\in\{1,2,3\}} \overline{C_{\bar\Omega_i}},
\\
\label{eq:interfaceOrientation}
\ta_{1,2} &\in C_{\barI_{1,2}},\quad
\ta_{2,3} \in C_{\barI_{2,3}},\quad
\ta_{3,1} \in C_{\barI_{3,1}}.
\end{align}
We finally define $(X^\pm_{\barI_{i,j}})_{i,j\in\{1,2,3\},i\neq j}$ and~$(X^\pm_{\bar\Omega_i})_{i\in\{1,2,3\}}$
by means of the unit length opening vectors of the cones 
$(C_{\barI_{i,j}})_{i,j\in\{1,2,3\},i\neq j}$ and~$(C_{\bar\Omega_i})_{i\in\{1,2,3\}}$,
respectively. The right hand sides of properties~\eqref{eq:defWedge} and~\eqref{eq:defInterpolWedge}
now serve as the defining objects for the interface and interpolation wedges, respectively.
\end{proof}

In a second preparatory step, we proceed with the definition of a preliminary
orthonormal frame along each of the three respective interfaces in the 
vicinity of the triple line, cf.\ Figure \ref{fig:curved triple line}.

\begin{construction}[Preliminary choice of tangent frame]
\label{tangentVectorFields}
Let the assumptions and notation of Definition~\ref{def:locRadius}
be in place. In particular, let~$r\in (0,r_{1,2}\wedge r_{2,3} \wedge r_{3,1}]$ be an associated 
admissible localization radius for the triple line~$\trLine$.
We then provide for all~$t\in[0,T]$ and all distinct phases $i,j\in\{1,2,3\}$ two tangent vector fields
$\ta_{i,j}(\cdot,t),\taTrJ_{i,j}(\cdot,t)\in\mathbb{S}^2$ along the local interface patch
$\barI_{i,j}(t)\cap B_{r}(\trLine(t))$ by means of the following procedure:

First, slicing the interface $\barI_{i,j}(t)$ along the planes $y{+}\mathrm{Tan}_y^\perp\trLine(t)$
produces a family of curves $\barI_{i,j}^{y}(t):=\barI_{i,j}(t)\cap\big(y{+}\mathrm{Tan}_y^\perp\trLine(t)\big)
\cap B_{r}(\trLine(t))$ for all $y\in\trLine(t)$. Second, for each~$t\in [0,T]$ and each~$y\in\trLine(t)$ denote 
by $\ta_{i,j}^{y}(\cdot,t)\in\mathbb{S}^2$ the tangent vector field along the curve~$\barI_{i,j}^{y}(t)$
which is oriented by $y{+}\frac{r}{2}\ta_{i,j}^{y}(y,t)\in 
W_{\barI_{i,j}}(t)\cap\big(y{+}\mathrm{Tan}_y^\perp\trLine(t)\big)$.
We then define two tangent vector fields on the local interface patch
$\barI_{i,j}\cap\mathcal{N}_{r}(\trLine)$ by means of
\begin{align*}
\ta_{i,j}(x,t) &:= \ta_{i,j}^{y}(x,t)\big|_{y=P_{\trLine}(x,t)} \in \mathbb{S}^2,
&& (x,t)\in \barI_{i,j}\cap\mathcal{N}_{r}(\trLine),
\\
\taTrJ_{i,j}(x,t) &:= (\no_{i,j}\times\ta_{i,j})(x,t) \in \mathbb{S}^2,
&& (x,t)\in \barI_{i,j}\cap\mathcal{N}_{r}(\trLine).
\end{align*}
This yields an orthonormal frame~$(\no_{i,j},\ta_{i,j},\taTrJ_{i,j})$
on $\barI_{i,j}\cap\mathcal{N}_{r}(\trLine)$.
Observe also that it holds
\begin{align}
\label{eq:orientationTangentTripleLine}
\taTrJ_{1,2}|_{\trLine}=\taTrJ_{2,3}|_{\trLine}=\taTrJ_{3,1}|_{\trLine}.
\end{align}

By a minor abuse of notation, we finally introduce extensions of these tangential
vector fields away from the interfaces. Namely, we define
\begin{align}
\label{eq:frameInterface}
(\ta_{i,j},\taTrJ_{i,j})(x,t) &:= 
(\ta_{i,j},\taTrJ_{i,j})(y,t)\big|_{y=P_{i,j}(x,t)},
\quad (x,t) \in  \mathrm{im}(\Psi_{i,j}) \cap\mathcal{N}_{r}(\trLine).
\end{align}
We refer to Definition~\ref{def:locRadiusInterface} 
for the diffeomorphism~$\Psi_{i,j}$ and the projection~$P_{i,j}$
onto the interface~$\barI_{i,j}$. We register in terms of regularity that
\begin{align}
\label{eq:regTangentFields}
\ta_{i,j},\taTrJ_{i,j} \in (C^0_tC^4_x \cap C^1_tC^2_x) 
(\mathrm{im}(\Psi_{i,j}) \cap \mathcal{N}_{r}(\trLine)).
\end{align}
This concludes our construction of \emph{orthonormal frames}~$(\no_{i,j},\ta_{i,j},\taTrJ_{i,j})$.
\hfill$\diamondsuit$
\end{construction}

In the sequel we will repeatedly rely on an explicit
representation of the gradients for the normal
and tangential vector fields. These
formulas are the content of the following result.

\begin{lemma}
\label{lem:gradientOrthonormalFrame}
Let the assumptions and notation of Definition~\ref{def:locRadius} 
and Construction~\ref{tangentVectorFields} be in place. To ease
notation, let~$\barI:=\barI_{1,2}$, $\barI':=\barI_{2,3}$ and~$\barI'':=\barI_{3,1}$
for the three interfaces present in the smoothly evolving regular double bubble~$(\bar\Omega_1,\bar\Omega_2,\bar\Omega_3)$. 
We proceed accordingly for the associated orthonormal frames~$(\no,\ta,\taTrJ)$, $(\no',\ta',\taTrJ')$
and~$(\no'',\ta'',\taTrJ'')$, respectively.

Using also the abbreviations~$\kappa_{\ta\ta}:=-\ta\otimes\ta:\nabla\no$, 
$\kappa_{\taTrJ\taTrJ}:=-\taTrJ\otimes\taTrJ:\nabla\no$
as well as~$\kappa_{\ta\taTrJ}:=-\ta\otimes\taTrJ:\nabla\no$, it holds
$\kappa_{\ta\taTrJ}=-\taTrJ\otimes\ta:\nabla\no$ and
	\begin{align}
	\label{eq:gradientNormal}
	\nabla\no 
	&= -\kappa_{\ta\ta}\,\ta\otimes\ta -\kappa_{\taTrJ\taTrJ}\,\taTrJ\otimes\taTrJ
	-\kappa_{\ta\taTrJ}\,(\taTrJ\otimes\ta + \ta\otimes\taTrJ),
	\\	\label{eq:gradientTau}
	\nabla\ta 
	&= \kappa_{\ta\ta}\,\no\otimes\ta - (\nabla\cdot\taTrJ)\,\taTrJ\otimes\ta
	+ \kappa_{\ta\taTrJ}\,\no\otimes\taTrJ + (\nabla\cdot\ta)\,\taTrJ\otimes\taTrJ,
	\\	\label{eq:gradientTaTrJ}
	\nabla\taTrJ 
	&= \kappa_{\taTrJ\taTrJ}\,\no\otimes\taTrJ 
	+ \kappa_{\ta\taTrJ}\,\no\otimes\ta
	+ (\nabla\cdot\taTrJ)\,\ta\otimes\ta
	- (\nabla\cdot\ta)\,\ta\otimes\taTrJ
	\end{align}
along the local interface patch~$\barI \cap \mathcal{N}_r(\trLine)$.
Analogous formulas of course hold true for~$(\no',\ta',\taTrJ')$ along~$\barI' \cap \mathcal{N}_r(\trLine)$
in terms of~$(\kappa'_{\ta'\ta'},\kappa'_{\taTrJ'\taTrJ'},\kappa'_{\ta'\taTrJ'})$,
and for~$(\no'',\ta'',\taTrJ'')$ along~$\barI'' \cap \mathcal{N}_r(\trLine)$
in terms of~$(\kappa''_{\ta''\ta''},\kappa''_{\taTrJ''\taTrJ''},\kappa''_{\ta''\taTrJ''})$.
\end{lemma}

\begin{proof}
The representation~\eqref{eq:gradientNormal} is essentially just a rephrasing
of the definition of the coefficients~$\kappa_{\ta\ta}$, $\kappa_{\taTrJ\taTrJ}$
and~$\kappa_{\ta\taTrJ}$. The only additional ingredients needed for
the validity of~\eqref{eq:gradientNormal} are 
$(\nabla\no)^\mathsf{T}\no =\frac{1}{2}\nabla|\no|^2=0$
and the symmetry of~$\nabla\no=\nabla^2 s_{1,2}$, cf.~\eqref{eq:extensionNormal}.

For a proof of~\eqref{eq:gradientTau}, we write $\ta=J\no$
where $J=\ta\wedge\no+\taTrJ\otimes\taTrJ$
denotes the associated rotation matrix around the $\taTrJ$-axis.
Based on~$(\no\cdot\nabla)\ta=0$, $(\nabla\ta)^\mathsf{T}\ta =\frac12 \nabla |\ta|^2 = 0$
and~\eqref{eq:gradientNormal} we then obtain
\begin{align*}
\nabla\ta &= \kappa_{\ta\ta} \, \no\otimes\ta
- \kappa_{\taTrJ\taTrJ} \, \taTrJ\otimes\taTrJ
- \kappa_{\ta\taTrJ} \, (\taTrJ\otimes\ta - \no\otimes\taTrJ) 
\\&~~~
+ \big((\ta\cdot\nabla)J\big)\no\otimes\ta 
+ \big((\taTrJ\cdot\nabla)J\big)\no\otimes\taTrJ 
\\&
= \kappa_{\ta\ta} \, \no\otimes\ta
- \kappa_{\taTrJ\taTrJ} \, \taTrJ\otimes\taTrJ
- \kappa_{\ta\taTrJ} \, (\taTrJ\otimes\ta - \no\otimes\taTrJ)
\\&~~~
+ \big(\no\otimes\no:(\ta\cdot\nabla)J\big) \, \no\otimes\ta 
+ \big(\taTrJ\otimes\no:(\ta\cdot\nabla)J\big) \, \taTrJ\otimes\ta 
\\&~~~
+ \big(\no\otimes\no:(\taTrJ\cdot\nabla)J\big) \, \no\otimes\taTrJ 
+ \big(\taTrJ\otimes\no:(\taTrJ\cdot\nabla)J\big) \, \taTrJ\otimes\taTrJ. 
\end{align*}
For the two appearing $(\no\otimes\no)$-components of~$\nabla J$,
it suffices to take the symmetric part of $J$ into account, which is
$\taTrJ\otimes\taTrJ$. It then follows from $\taTrJ\cdot\no=0$ that 
\begin{align*}
\no\otimes\no:(\ta\cdot\nabla)J
=\no\otimes\no:(\taTrJ\cdot\nabla)J=0. 
\end{align*}
Based on~\eqref{eq:gradientNormal}, $\taTrJ \cdot (\taTrJ \cdot \nabla) \taTrJ 
= \frac12 (\taTrJ \cdot \nabla)|\taTrJ|^2 =0$, and
$\taTrJ=\no\times\ta$ we may further compute
\begin{align*}
\taTrJ\otimes\no:(\taTrJ\cdot\nabla)J
&= (\taTrJ\otimes\no) : (\taTrJ\cdot\nabla)
(\ta\wedge\no) + \no\cdot(\taTrJ\cdot\nabla)\taTrJ
\\&
= (\taTrJ\otimes\no) : (\taTrJ\cdot\nabla)
(\ta\wedge\no) - \taTrJ\otimes\taTrJ:\nabla\no
\\&
= (\taTrJ\otimes\no) : (\taTrJ\cdot\nabla)
(\ta\wedge\no) + \kappa_{\taTrJ\taTrJ}.
\end{align*}
Based on~\eqref{eq:gradientNormal}, $\taTrJ\cdot\no=0$,
$(\taTrJ\otimes\no) : (\ta\cdot\nabla)
(\ta\wedge\no) = (\taTrJ\otimes\ta):\nabla\ta$, and $\taTrJ=\no\times\ta$
we in addition have
\begin{align*}
\taTrJ\otimes\no:(\ta\cdot\nabla)J
&= (\taTrJ\otimes\ta):\nabla\ta
+ \no\cdot(\ta\cdot\nabla)\taTrJ
\\&
= (\taTrJ\otimes\ta):\nabla\ta 
- \taTrJ\otimes\ta:\nabla\no
\\&
= (\taTrJ\otimes\ta):\nabla\ta
+ \kappa_{\ta\taTrJ}.
\end{align*}
The combination of the previous four displays yields
\begin{align}
\nabla\ta &= \kappa_{\ta\ta} \, \no\otimes\ta
+ \big((\taTrJ\otimes\ta):\nabla\ta \big) \, \taTrJ\otimes\ta
+ \kappa_{\ta\taTrJ} \, \no\otimes\taTrJ
+ (\nabla\cdot\ta) \, \taTrJ\otimes\taTrJ,
\\
\label{eq:repDivTangent}
\nabla\cdot\ta &=
(\taTrJ\otimes\no) : (\taTrJ\cdot\nabla)(\ta\wedge\no).
\end{align}
Moreover, exploiting that $\taTrJ=\no\times\ta$
yields by the product rule, \eqref{eq:gradientNormal} and the previous display
\begin{align}
\nabla\taTrJ &= \kappa_{\taTrJ\taTrJ} \, \no\otimes\taTrJ
+  (\nabla\cdot\taTrJ) \, \ta\otimes\ta
+ \kappa_{\ta\taTrJ} \, \no\otimes\ta
- (\nabla\cdot\ta) \, \ta\otimes\taTrJ,
\\
\label{eq:repDivTangentTripleLine}
\nabla\cdot\taTrJ &= -(\taTrJ\otimes\ta):\nabla\ta.
\end{align}
The previous two displays in turn directly imply~\eqref{eq:gradientTau} and~\eqref{eq:gradientTaTrJ}.
\end{proof}

\begin{figure}
  \centering      
  \includegraphics{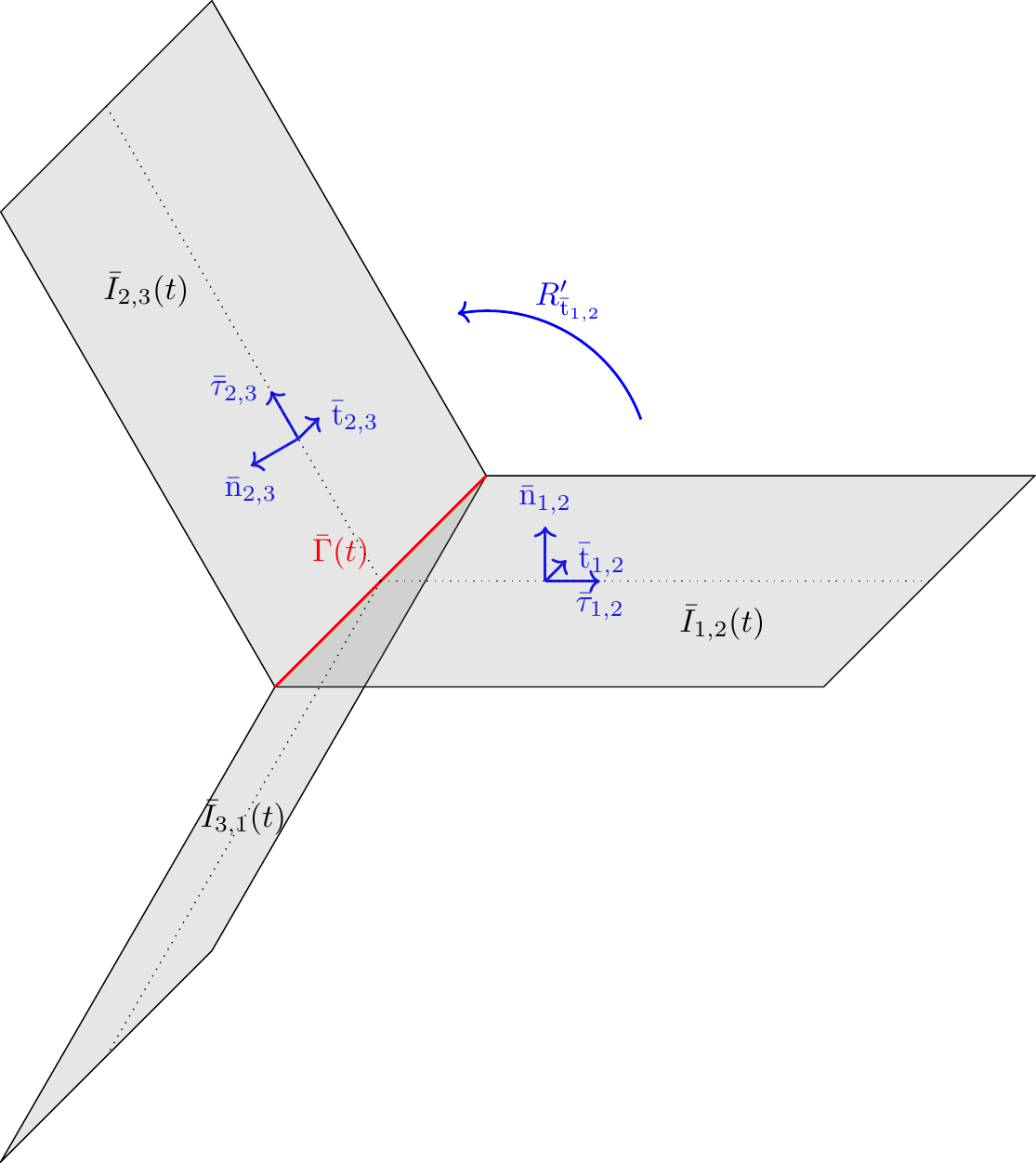}
	\caption{Local geometry at the triple line and preliminary construction of tangent frame. 
	For simplicity, we illustrate here the case of three flat sheets coming together at equal 
	angles of $120^\circ$ along a straight triple line $\trLine(t)$. In this case, the ``Herring'' 
	rotation $R'(y,t)$ is a rotation by $120^\circ$ about the axis given by the tangent 
	vector $\taTrJ= \taTrJ_{1,2}(y,t)$ of $\trLine(t)$. The dotted lines represent the 
	three slices $\barI_{i,j}^y(t)$ of the interfaces $\barI_{i,j}$.}
	\label{fig:flat sheets}
\end{figure}

The orthonormal frames provided by Construction~\ref{tangentVectorFields}
together with the signed distance functions~\eqref{eq:defSignedDistance}
constitute all ingredients for the construction
of a suitable building block~$\widetilde\xi_{i,j}$
for the vector field~$\xi_{i,j}$; at least in~$\mathcal{N}_r(\trLine)\cap\mathrm{im}(\Psi_{i,j})$,
see Construction~\ref{AnsatzAuxiliaryXiHalfSpace} below.
However, we also have to provide a construction of the vector field~$\xi_{i,j}$
outside of the domain~$\mathcal{N}_r(\trLine)\cap\mathrm{im}(\Psi_{i,j})$,
i.e., where this vector field a~priori does not have a ``natural'' definition.
The guiding principle is to mimic the Herring angle condition valid on the triple line: 
\begin{align*}
\sigma_{1,2}\no_{1,2} + \sigma_{2,3}\no_{2,3} + \sigma_{3,1}\no_{3,1} = 0.
\end{align*}
This condition motivates to appropriately rotate the already defined 
candidate vector fields~$\widetilde\xi_{j,k}$ and $\widetilde\xi_{k,i}$ 
to provide the building blocks for the vector field~$\xi_{i,j}$
throughout~$\mathcal{N}_r(\trLine)\cap\mathrm{im}(\Psi_{j,k})$
and $\mathcal{N}_r(\trLine)\cap\mathrm{im}(\Psi_{k,i})$, respectively. 
	
The rotations used in this procedure have to be chosen 
carefully so that our constructions will satisfy the requirements of a local gradient-flow calibration at the triple line, e.g., sufficiently high regularity 
(in particular, adequate compatibility along the triple line)
and the validity of the required evolution equations (up to a desired 
error in the distance to the interface). 

\begin{construction}[Gauged Herring rotation fields]
\label{RotationsTripleLine}
Let the assumptions and notation of Definition~\ref{def:locRadius}, 
Construction~\ref{tangentVectorFields} and Lemma~\ref{lem:gradientOrthonormalFrame} be in place.
Consistent with the notational conventions of the latter, denote by~$\Psi$, $\Psi'$ and~$\Psi''$
the diffeomorphisms from Definition~\ref{def:locRadiusInterface} with respect to the
interfaces~$\bar I$, $\bar I'$ and~$\bar I''$, respectively. We proceed accordingly
for the surface tensions $(\sigma,\sigma',\sigma'')$ and the projections $(P,P',P'')$.

We now define a pair of \emph{Herring rotation fields} 
\begin{align}
R'_{\taTrJ},R''_{\taTrJ}\colon \mathcal{N}_{r}(\trLine) \cap \mathrm{im}(\Psi) \to SO(3) \subset\Rd[3{\times}3]
\end{align} 
around the $\taTrJ$-axis by means of
\begin{align}
\label{eq:rotationTangent}
R'_{\taTrJ}(x,t) &:= \cos\theta'\,\mathrm{Id} + \sin\theta'\, (\ta \wedge \no)(x,t)
+ (1{-}\cos\theta')\,(\taTrJ\otimes\taTrJ)(x,t),
\\
R''_{\taTrJ}(x,t) &:= \cos\theta''\,\mathrm{Id} + \sin\theta''\, (\ta \wedge \no)(x,t)
+ (1{-}\cos\theta'')\,(\taTrJ\otimes\taTrJ)(x,t)
\end{align}
for all~$(x,t)\in\mathcal{N}_{r}(\trLine) \cap \mathrm{im}(\Psi)$, cf.\ Figure \ref{fig:flat sheets}.
The associated angles $\theta',\theta''\in (0,\pi)$ are independent 
of~$(x,t)\in\mathcal{N}_{r}(\trLine) \cap \mathrm{im}(\Psi)$ 
and chosen based on the triple of surface tensions 
$(\sigma,\sigma',\sigma'')$ such that the relations
\begin{align}
\label{eq:rotationTangentAngle1}
R'_{\taTrJ}\no &= \no', 
\\
\label{eq:rotationTangentAngle2}
R''_{\taTrJ}\no &= \no''
\end{align}
hold true along the triple line~$\trLine$. Hence, the Herring condition \eqref{HerringAngleCondition} 
implies that for all~$(x,t)\in\mathcal{N}_{r}(\trLine) \cap \mathrm{im}(\Psi)$
and all~$v\in\Rd[3]$ such that~$v\cdot\taTrJ(x,t)=0$ it holds
\begin{align}
\label{eq:HerringConditionByRotation}
\sigma v + \sigma'R'_{\taTrJ}(x,t)v + \sigma''R''_{\taTrJ}(x,t)v = 0.
\end{align}
Analogously, one defines a pair of rotations $(R_{\taTrJ'},R''_{\taTrJ'})$, 
resp.\ $(R_{\taTrJ''},R'_{\taTrJ''})$, throughout the region~$\mathcal{N}_{r}(\trLine) \cap \mathrm{im}(\Psi')$,
resp.\ $\mathcal{N}_{r}(\trLine) \cap \mathrm{im}(\Psi'')$.

Apart from the Herring rotation fields, we also introduce the \emph{gauge rotation field}  
\begin{align}
\label{eq:defGaugeRotation}
R_{\no}:=R_{\no}^{(2)}R_{\no}^{(1)}\colon \mathcal{N}_{r}(\trLine) 
\cap \mathrm{im}(\Psi) \to SO(3) \subset\Rd[3\times 3]	
\end{align}
around the $\no$-axis, cf.\ Figure \ref{fig:curved triple line}. 
The auxiliary rotation fields $R_{\no}^{(1)}$ and $R_{\no}^{(2)}$ around the $\no$-axis are defined via
\begin{align}
\label{eq:rotationNormal}
R_{\no}^{(1)}(x,t) &:= \cos\delta(x,t)\,\mathrm{Id} 
+ \sin\delta(x,t)\,(\taTrJ \wedge \ta)(x,t)
\\&~~~~ \nonumber
+ (1{-}\cos\delta(x,t))\,(\no\otimes\no)(x,t),
\label{eq:rotationNormal2}
\\
R_{\no}^{(2)}(x,t) &:= \cos\omega(x,t)\,\mathrm{Id} 
+ \sin\omega(x,t)\,(\taTrJ \wedge \ta)(x,t)
\\&~~~~ \nonumber
+ (1{-}\cos\omega(x,t))\,(\no\otimes\no)(x,t).
\end{align}
Here the rotation angle $\delta(x,t)$ is given explicitly by
\begin{align}
\label{eq:rotationNormalAngle}
\delta(x,t) &:= s(x,t)\kappa_{\ta\taTrJ}(x,t),
\quad (x,t)\in\mathcal{N}_{r}(\trLine) \cap \mathrm{im}(\Psi),
\end{align}
and the angle $\omega(x,t)$ is given by the extension
\begin{align}
\label{eq:rotationNormalAngle2}
\omega(x,t) &:= \widehat\omega(P(x,t),t),
\quad (x,t)\in\mathcal{N}_{r}(\trLine) \cap \mathrm{im}(\Psi)
\end{align}
of $\widehat \omega(x,t)$, which in turn is defined by the one-parameter family of ODEs
\begin{align}
\label{ODE_omega}
&
  \begin{cases}
	\widehat\omega(x,t) = 0, & (x,t)\in\trLine, \\
	\left(\ta(x,t) \cdot \nabla \right) \widehat\omega(x,t)  = (\nabla\cdot\taTrJ)(x,t),
	& (x,t) \in \barI \cap \mathcal{N}_r(\trLine).
  \end{cases}
\end{align}

Analogously, one defines a gauge rotation $R_{\no'}:=R_{\no'}^{(2)}R_{\no'}^{(1)}$, 
resp.\ $R_{\no''}:=R_{\no''}^{(2)}R_{\no''}^{(1)}$, throughout the 
region~$\mathcal{N}_{r}(\trLine) \cap \mathrm{im}(\Psi')$,
resp.\ $\mathcal{N}_{r}(\trLine) \cap \mathrm{im}(\Psi'')$.

We finally define via conjugation a pair of \emph{gauged Herring rotation fields}
\begin{align}
\label{eq:defGaugedHerringRotation}
\widetilde R'_{\bar I} := R_{\no}R'_{\taTrJ}R_{\no}^\mathsf{T}&\colon
\mathcal{N}_{r}(\trLine) \cap \mathrm{im}(\Psi) \to SO(3) \subset\Rd[3{\times}3],
\\
\label{eq:defGaugedHerringRotation2}
\widetilde R''_{\bar I} := R_{\no}R''_{\taTrJ}R_{\no}^\mathsf{T}&\colon
\mathcal{N}_{r}(\trLine) \cap \mathrm{im}(\Psi) \to SO(3) \subset\Rd[3{\times}3],
\end{align}
and analogously a pair~$(\widetilde R_{\bar I'},\widetilde R''_{\bar I'})$,
resp.\ $(\widetilde R_{\bar I''},\widetilde R'_{\bar I''})$,
of gauged Herring rotation fields throughout the region~$\mathcal{N}_{r}(\trLine) \cap \mathrm{im}(\Psi')$,
resp.\ $\mathcal{N}_{r}(\trLine) \cap \mathrm{im}(\Psi'')$.
\hfill$\diamondsuit$
\end{construction}

In a symmetric setting with either rotational or translational symmetry, cf.\ Figure~\ref{fig:flat sheets},
the gauge rotations~$R_{\no}$, $R_{\no'}$, and~$R_{\no''}$
are not needed and, in fact, reduce to the identity matrix. In the general case, cf.\ Figure~\ref{fig:curved triple line}, 
they account for the fact that, for instance,
the normal vector field~$\no(\cdot,t)$ evaluated along a slice $\barI(t)\cap(x{+}\mathrm{Tan}_x^\perp\trLine(t))$
for some $x\in\trLine(t)$ will in general rotate out of the plane $x{+}\mathrm{Tan}_x^\perp\trLine(t)$
as one moves away from the triple line point~$x$.

We conclude this section with the derivation of compatibility
conditions along the triple line. These represent the last missing ingredients
to ensure compatibility of the main building blocks~$\widetilde\xi_{i,j}$
(cf.\ Construction~\ref{AnsatzAuxiliaryXiHalfSpace} below)
for the vector field~$\xi_{i,j}$ and its rotated counterparts
along the triple line (see Lemma~\ref{lem:firstOrderCompXi} below).

\begin{lemma}
\label{lem:compatibilityConditionsTripleLine}
Let the assumptions and notation of Definition~\ref{def:locRadius}, 
Construction~\ref{tangentVectorFields}, Lemma~\ref{lem:gradientOrthonormalFrame},
and Construction~\ref{RotationsTripleLine} be in place. Consistently with
the notational conventions of the latter two, denote by~$H$, $H'$ and~$H''$
the extended scalar mean curvatures defined by~\eqref{eq:extensionCurvature} 
with respect to the interfaces~$\barI$, $\barI'$ and~$\barI''$, respectively.
Denote by~$\vec{V}_{\trLine}$ the normal velocity vector field of the triple line.

Then, the following compatibility conditions are satisfied along the triple line~$\trLine$:
\begin{align}
\label{eq:compConditionTripleLine}
\ta' &= R'_{\taTrJ}\ta, \quad \ta'' = R''_{\taTrJ}\ta,
\\
\label{eq:compConditionTripleLine1}
\kappa'_{\ta'\taTrJ'} &= \kappa_{\ta\taTrJ},
\quad \kappa''_{\ta''\taTrJ''} = \kappa_{\ta\taTrJ},
\\
\label{eq:compConditionTripleLine2}
\kappa'_{\taTrJ'\taTrJ'} &= (R'_{\taTrJ}\no\cdot\no)\kappa_{\taTrJ\taTrJ}
- (R'_{\taTrJ}\no\cdot\ta) \nabla\cdot\ta,
\\
\label{eq:compConditionTripleLine3}
\kappa''_{\taTrJ''\taTrJ''} &= (R''_{\taTrJ}\no\cdot\no)\kappa_{\taTrJ\taTrJ}
- (R''_{\taTrJ}\no\cdot\ta) \nabla\cdot\ta,
\\
\label{eq:compConditionTripleLine4}
\nabla\cdot\ta' &= (R'_{\taTrJ}\no\cdot\ta)\kappa_{\taTrJ\taTrJ}
+ (R'_{\taTrJ}\no\cdot\no) \nabla\cdot\ta,
\\
\label{eq:compConditionTripleLine5}
\nabla\cdot\ta'' &= (R''_{\taTrJ}\no\cdot\ta)\kappa_{\taTrJ\taTrJ}
+ (R''_{\taTrJ}\no\cdot\no) \nabla\cdot\ta,
\\
\label{eq:compConditionTripleLine6}
\sigma H {+} \sigma'H' {+} \sigma''H'' &= 0,
\\
\label{eq:compConditionTripleLine7}
\kappa''_{\ta''\ta''}(\ta''\cdot \vec{V}_{\trLine}) {+} (\ta''\cdot\nabla)H''
&= \kappa'_{\ta'\ta'}(\ta'\cdot \vec{V}_{\trLine}) {+} (\ta'\cdot\nabla)H'
\\&\nonumber 
= \kappa_{\ta\ta}(\ta\cdot \vec{V}_{\trLine}) {+} (\ta\cdot\nabla)H.
\end{align}
Of course, the analogues of~\eqref{eq:compConditionTripleLine}
as well as~\emph{\eqref{eq:compConditionTripleLine2}--\eqref{eq:compConditionTripleLine5}}
hold true for the appropriate relabellings of the associated data.

Introduce next a gauged orthonormal frame on~$\mathcal{N}_r(\trLine)\cap\mathrm{im}(\Psi)$
by means of
\begin{align}
\label{eq:defGaugedOrthonormalFrame}
(\no,\ta_*,\taTrJ_*) := (\no,R_{\no}\ta,R_{\no}\taTrJ).
\end{align}
Then, the following compatibility condition holds true:
\begin{align}
\label{eq:compConditionTripleLine8}
(\no,\ta_*,\taTrJ_*) &= (\no,\ta,\taTrJ)\quad \text{along the triple line }  \trLine.
\end{align}
The analogue of~\eqref{eq:compConditionTripleLine8} 
with respect to the gauged frame~$(\no',\ta'_*,\taTrJ'_*) := (\no',R_{\no'}\ta',R_{\no'}\taTrJ')$
on~$\mathcal{N}_r(\trLine)\cap\mathrm{im}(\Psi')$, resp.\ $(\no'',\ta''_*,\taTrJ''_*) 
:= (\no'',R_{\no''}\ta'',R_{\no''}\taTrJ'')$ on~$\mathcal{N}_r(\trLine)\cap\mathrm{im}(\Psi'')$,
is also satisfied.
\end{lemma}

\begin{proof}
Except for the conditions~\eqref{eq:compConditionTripleLine} 
and~\eqref{eq:compConditionTripleLine8}, 
the asserted compatibility conditions are consequences of differentiating
the existing zeroth and first order compatibility conditions along the triple line.

\textit{Step 1: Proof of~\eqref{eq:compConditionTripleLine}.}
By~\eqref{eq:orientationTangentTripleLine} and the choice of the
orientation for the tangent fields~$(\ta,\ta',\ta'')$ along
the triple line, cf.\ Construction~\ref{tangentVectorFields},
it holds
\begin{align}
\label{eq:defNinetyDegreeRotation}
\ta = J\no, \quad \ta' = J\no', \quad \ta'' = J\no''\quad \text{on } \trLine
\end{align}
in terms of a single $90^\circ$ rotation field around the $\taTrJ$-axis
\begin{align}
\label{eq:rotationNinetyDegreeTripleLine}
J = (\ta\wedge \no) + \taTrJ\otimes\taTrJ
  = (\ta'\wedge \no') + \taTrJ'\otimes\taTrJ'
	= (\ta''\wedge \no'') + \taTrJ''\otimes\taTrJ''
	\quad\text{on } \trLine.
\end{align}
Hence, it follows from~\eqref{eq:rotationTangentAngle1}
and the fact that the Herring rotation~$R'_{\taTrJ}$
is a rotation around the same axis
\begin{align*}
R'_{\taTrJ}\ta=R'_{\taTrJ}J\no=JR'_{\taTrJ}\no=J\no'=\ta'
	\quad\text{on } \trLine.
\end{align*} 
This proves the first asserted identity of~\eqref{eq:compConditionTripleLine};
the second of course follows analogously based on~\eqref{eq:rotationTangentAngle2}.

\textit{Step 2: Proof of~\emph{\eqref{eq:compConditionTripleLine1}--\eqref{eq:compConditionTripleLine3}}.}
Since the Herring rotation~$R'_{\taTrJ}$ defined by~\eqref{eq:rotationTangent}
is a rotation around the $\taTrJ$-axis with constant angle, the coefficients in 
the representation $R'_{\taTrJ}\no = (R'_{\taTrJ}\no \cdot \no) \no +(R'_{\taTrJ}\no \cdot \ta) \ta$ 
are constant. Hence  we may compute along~$\trLine$
together with the formulas~\eqref{eq:gradientNormal} and~\eqref{eq:gradientTau}
\begin{align*}
(\taTrJ\cdot\nabla)R'_{\taTrJ}\no &= (R'_{\taTrJ}\no\cdot\no) (\taTrJ\cdot\nabla)\no
+ (R'_{\taTrJ}\no\cdot\ta) (\taTrJ\cdot\nabla)\ta
\\
&= \big((R'_{\taTrJ}\no\cdot\ta)(\nabla\cdot\ta)
-(R'_{\taTrJ}\no\cdot\no)\kappa_{\taTrJ\taTrJ}\big)\taTrJ
- (R'_{\taTrJ}\no\cdot\no)\kappa_{\ta\taTrJ}\ta
+ (R'_{\taTrJ}\no\cdot\ta)\kappa_{\ta\taTrJ}\no.
\end{align*}
Furthermore, by the analogue of~\eqref{eq:gradientTau} for the tangent field~$\ta'$ 
as well as the identities~\eqref{eq:orientationTangentTripleLine}
and~\eqref{eq:compConditionTripleLine}, and again the fact that $R'_{\taTrJ}$ and $J$ commute, 
we obtain along the triple line~$\trLine$
\begin{align*}
(\taTrJ'\cdot\nabla)\no' &= -\kappa'_{\taTrJ'\taTrJ'}\taTrJ'
-\kappa'_{\ta'\taTrJ'}\ta'
\\
&= -\kappa'_{\taTrJ'\taTrJ'}\taTrJ
-(R'_{\taTrJ}\ta\cdot\ta)\kappa'_{\ta'\taTrJ'}\ta
-(R'_{\taTrJ}\ta\cdot\no)\kappa'_{\ta'\taTrJ'}\no
\\
&= -\kappa'_{\taTrJ'\taTrJ'}\taTrJ
-(R'_{\taTrJ}\no\cdot\no)\kappa'_{\ta'\taTrJ'}\ta
+(R'_{\taTrJ}\no\cdot\ta)\kappa'_{\ta'\taTrJ'}\no.
\end{align*}
Hence, the defining condition~\eqref{eq:rotationTangentAngle1}
of the Herring rotation~$R'_{\taTrJ}$ and matching coefficients
in the previous two displays implies the first identity of~\eqref{eq:compConditionTripleLine1}
as well as~\eqref{eq:compConditionTripleLine2} (note that of course, either~$(R'_{\taTrJ}\no\cdot\no)$
or~$(R'_{\taTrJ}\no\cdot\ta)$ is non-zero). The second identity of~\eqref{eq:compConditionTripleLine1}
as well as~\eqref{eq:compConditionTripleLine3} in turn follow from an analogous computation
based on~\eqref{eq:rotationTangentAngle2}.

\textit{Step 3: Proof of~\emph{\eqref{eq:compConditionTripleLine4}--\eqref{eq:compConditionTripleLine5}}.}
These two compatibility conditions are derived as in the previous step, this
time computing the tangential derivative along the triple line for both sides
of the identities from~\eqref{eq:compConditionTripleLine}, respectively.

\textit{Step 4: Proof of~\emph{\eqref{eq:compConditionTripleLine6}--\eqref{eq:compConditionTripleLine7}}.}
By \eqref{NormalVelocityAndMeanCurvatureOnEachInterface}, the 
normal velocity~$\vec{V}_{\trLine}$ of the triple line satisfies
along~$\trLine$
\begin{align}
\label{eq:uniqueVelocityAtTripleLine}
\vec{V}_{\trLine} \cdot \sigma\no = \sigma H,
\quad \vec{V}_{\trLine} \cdot \sigma'\no' = \sigma' H',
\quad \vec{V}_{\trLine} \cdot \sigma''\no'' = \sigma'' H''.
\end{align}
Summing these identities results in~\eqref{eq:compConditionTripleLine6}
thanks to the Herring angle condition~\eqref{HerringAngleCondition}
being satisfied at each time.

To derive the compatibility condition~\eqref{eq:compConditionTripleLine7}, we differentiate the Herring angle condition and obtain
\begin{align*}
(\partial_t + \vec{V}_{\trLine} \cdot \nabla ) \left( \sigma \no + \sigma'\no' + \sigma''\no''\right) = 0.
\end{align*}
Now we compute using \eqref{eq:extensionNormal} and
and~\eqref{eq:evolutionSignedDistance} for the first term and \eqref{eq:gradientNormal} for the second one
\begin{align}
\label{AdvectiveDerivativeOnHerringCondition}
 \partial_t \no + (\vec{V}_{\trLine} \cdot \nabla )\no
= -(\taTrJ\cdot\nabla H)  \taTrJ -(\ta\cdot\nabla H)  \ta   
- (\vec{V}_{\trLine}\cdot\ta)(\kappa_{\ta\ta}\ta{+}\kappa_{\ta\taTrJ}\taTrJ)
\end{align}
on $\trLine$. The analogous equations hold for $\no' $ and $\no''$. Plugging those into \eqref{AdvectiveDerivativeOnHerringCondition}, using~\eqref{eq:orientationTangentTripleLine} and \eqref{eq:compConditionTripleLine1}, we obtain
\begin{align*}
0 &=   \big(\taTrJ \cdot \nabla ( \sigma H + \sigma' H' + \sigma'' H'')\big) \taTrJ
		+ (\ta\cdot\nabla H) \sigma\ta
		+ (\ta'\cdot\nabla H') \sigma'\ta'
		+ (\ta''\cdot\nabla H'') \sigma''\ta'' 
\\&~~~+ \kappa_{\ta\ta} (\vec{V}_{\trLine}\cdot\ta) \sigma\ta
				 + \kappa'_{\ta'\ta'} (\vec{V}_{\trLine}\cdot\ta') \sigma'\ta'
				 + \kappa''_{\ta''\ta''} (\vec{V}_{\trLine}\cdot\ta'') \sigma''\ta''
\\&~~~
+   \vec{V}_{\trLine}\cdot(\sigma\ta{+}
  \sigma'\ta'{+}\sigma''\ta'') (\kappa_{\ta\taTrJ}\taTrJ)
\end{align*}
on $\trLine$. 
Differentiating \eqref{eq:compConditionTripleLine6} along $\trLine$, we see that the first term vanishes. 
The last term vanishes by applying the fixed rotation $J$ to the Herring condition \eqref{HerringAngleCondition}. 
Thus, since the three vectors $\ta$, $\ta'$, and $\ta''$ lie in one plane,
we deduce~\eqref{eq:compConditionTripleLine7} from the previous display.

\textit{Step 5: Proof of~\emph{\eqref{eq:compConditionTripleLine8}
}.}
The requirement~\eqref{eq:compConditionTripleLine8} is immediate from the 
definitions~\eqref{eq:defGaugeRotation}--\eqref{ODE_omega} in form of
\begin{align}
\label{eq:gaugeRotationOnTripleLine}
R_{\no} = \mathrm{Id}
\end{align}
along the triple line~$\trLine$.
\end{proof}

With all of these ingredients in place, we may eventually move on with the construction 
of a local gradient-flow calibration at a triple line.

\subsection{Extension of vector fields close to each interface}
\label{subsection:step_1}
The aim of this section is to provide auxiliary extensions of the unit normal
vector fields and an auxiliary extension of the normal velocity vector field
which are defined in the neighborhood~$\mathcal{N}_r(\trLine) \cap \mathrm{im}(\Psi_{i,j})$ for each interface~$\barI_{i,j}$,
respectively. These extensions constitute the main building blocks for the
desired extensions from Proposition~\ref{prop:gradientFlowCalibrationTripleLine}.

Throughout this whole subsection, let the assumptions of Proposition~\ref{prop:gradientFlowCalibrationTripleLine}
and the notation of Section~\ref{sec:localCalibrationInterface} and Subsection~\ref{sec:notationTripleLine} 
be in place. In particular, let us again make use of the
following notational conventions which basically aim to drop the indices $i,j\in\{1,2,3\}$.
We denote by $\barI:=\barI_{1,2},\barI':=\barI_{2,3},\barI'':=\barI_{3,1}$
the three interfaces present in the given smoothly evolving 
regular double bubble~$(\bar\Omega_1,\bar\Omega_2,\bar\Omega_3)$ on~$[0,T]$. 
We proceed accordingly for the associated orthonormal frames~$(\no,\ta,\taTrJ)$, $(\no',\ta',\taTrJ')$
$(\no'',\ta'',\taTrJ'')$ due to Construction~\ref{tangentVectorFields},
the surface tensions~$(\sigma,\sigma',\sigma'')$, the signed distances~$(s,s',s'')$,
the projections~$(P,P',P'')$, the scalar mean curvatures~$(H,H',H'')$ and the diffeomorphisms~$(\Psi,\Psi',\Psi'')$
from Definition~\ref{def:locRadiusInterface}. 

\begin{construction}[Extension of normal vector fields close to their associated interfaces]
\label{AnsatzAuxiliaryXiHalfSpace}
Define a coefficient function~$\alpha\colon\mathcal{N}_r(\trLine)\cap\mathrm{im}(\Psi)\to\mathbb{R}$ by
\begin{align}
\label{eq:alpha_coefficient_xi}
\alpha(x,t) &:= \alpha_{\mathrm{vel}}(x,t) + (\nabla\cdot\ta)(x,t),
\quad (x,t) \in \mathcal{N}_r(\trLine)\cap\mathrm{im}(\Psi),
\end{align}
where~$\alpha_{\mathrm{vel}}\colon\mathcal{N}_r(\trLine)\cap\mathrm{im}(\Psi)\to\mathbb{R}$
denotes, for the time being, an arbitrary coefficient function of 
class $C^0_tC^2_x(\mathcal{N}_r(\trLine)\cap\mathrm{im}(\Psi))$ 
such that along the triple line it holds
\begin{align}
\label{eq:initialValueODEalpha}
	\alpha_{\mathrm{vel}}(x,t) & = \ta(x,t)\cdot\vec{V}_{\trLine}(x,t),\quad (x,t)\in\trLine.
\end{align}
Here, $\vec{V}_{\trLine}$ denotes again the normal velocity vector field of the triple line~$\trLine$.
Recall finally the definition~\eqref{eq:defGaugedOrthonormalFrame} of the gauged
orthonormal frame~$(\no,\ta_*,\taTrJ_*)$.

We then define an initial extension~$\widetilde\xi\colon\mathcal{N}_r(\trLine)\cap\mathrm{im}(\Psi)\to\mathbb{R}^3$ 
for the normal vector field~$\no|_{\bar I}$ 
of the interface~$\barI$ by means of the \emph{gauged expansion ansatz}
\begin{align}
\label{eq:AnsatzAuxiliaryXiHalfSpace}
	\widetilde\xi(x,t)  
	& := 
	\no(x,t)
	\\&~~~~\nonumber
	+\alpha(P_{\trLine}(x,t),t) s(x,t) \ta_*(x,t) 
	\\&~~~~\nonumber 
	-\frac{1}{2} \alpha^2(P_{\trLine}(x,t),t) s^2(x,t) \no(x,t) 
\end{align}
for all $(x,t)\in \mathcal{N}_r(\trLine)\cap\mathrm{im}(\Psi)$.

Analogously, one defines initial extensions $\widetilde\xi'\colon\mathcal{N}_r(\trLine)\cap\mathrm{im}(\Psi')\to\mathbb{R}^3$ 
as well as $\widetilde\xi''\colon\mathcal{N}_r(\trLine)\cap\mathrm{im}(\Psi'')\to\mathbb{R}^3$ of the normal vector
fields~$\no'|_{\barI'}$ and $\no''|_{\barI''}$.
\hfill$\diamondsuit$
\end{construction}

The following result shows that, after applying the correct gauged
Herring rotation as provided by Construction~\ref{RotationsTripleLine}, 
the initial extensions of our normal vector fields are regular and compatible to first 
order along the triple line~$\trLine$.
%

\begin{lemma}
\label{lem:firstOrderCompXi}
Let~$(\widetilde\xi,\widetilde\xi',\widetilde\xi'')$
be the initial extensions from Construction~\ref{AnsatzAuxiliaryXiHalfSpace} 
of the normal vector fields~$(\no|_{\barI}, \no'|_{\barI'}, \no''|_{\barI''})$. Moreover,
let~$(\widetilde R'_{\barI},\widetilde R''_{\barI})$, $(\widetilde R_{\barI'},\widetilde R''_{\barI'})$
and~$(\widetilde R_{\barI''}, \widetilde R'_{\barI''})$ be the gauged Herring rotations
as provided by Construction~\ref{RotationsTripleLine}.

Then it holds $(\widetilde\xi,\widetilde R'_{\barI}\,\widetilde\xi,\widetilde R''_{\barI}\,\widetilde\xi\,) \in 
(C^0_tC^2_x\cap C^1_tC^0_x)(\mathcal{N}_r(\trLine)\cap\mathrm{im}(\Psi))$
with corresponding estimates
\begin{align}
\label{eq:regEstimateAuxiliarxExtensionsNormal}
|(\nabla,\nabla^2,\partial_t)(\widetilde\xi,\widetilde R'_{\barI}\,\widetilde\xi,\widetilde R''_{\barI}\,\widetilde\xi\,)|
\leq C \quad\text{in } \mathcal{N}_r(\trLine)\cap\mathrm{im}(\Psi),
\end{align}
where the constant~$C>0$ only depends on the data of the
smoothly evolving regular double bubble~$(\bar\Omega_1,\bar\Omega_2,\bar\Omega_3)$ on~$[0,T]$.
Moreover, the constructions are compatible to first order 
in the sense that along the triple line~$\trLine$
\begin{align}
\label{eq:compatibilityTripleLineXi}
\widetilde R'_{\barI}\,\widetilde\xi &= \widetilde\xi',
\quad
\widetilde R''_{\barI}\,\widetilde\xi = \widetilde\xi'',
\\
\label{eq:compatibilityTripleLineXi2}
\nabla \big(\widetilde R'_{\barI}\,\widetilde\xi\,\big) &= 
\nabla \widetilde\xi',
\quad
\nabla \big(\widetilde R''_{\barI}\,\widetilde\xi\,\big) = 
\nabla \widetilde\xi''.
\end{align}
Analogous claims are satisfied in terms of the vector fields~$(\widetilde R_{\barI'}\,\widetilde\xi',
\widetilde\xi',\widetilde R''_{\barI'}\,\widetilde\xi')$,
resp.\ the vector fields~$(\widetilde R_{\barI''}\,\widetilde\xi'',
\widetilde R'_{\barI''}\,\widetilde\xi'',\widetilde\xi'')$,
throughout the region~$\mathcal{N}_r(\trLine)\cap\mathrm{im}(\Psi')$, resp.\ the region
$\mathcal{N}_r(\trLine)\cap\mathrm{im}(\Psi'')$.
\end{lemma}

\begin{proof}
We split the proof into two steps.

\textit{Step 1: Regularity estimates.} 
We first claim that for each~$\mathcal{R}\in\{R_{\taTrJ}',R_{\taTrJ}'',R_{\no}\}$
\begin{align}
\label{eq:regEstimateRotations}
|(\nabla,\nabla^2,\partial_t)\mathcal{R}| \leq C
\quad\text{in } \mathcal{N}_r(\trLine)\cap\mathrm{im}(\Psi)
\end{align}
for some constant~$C>0$ which depends only on the data of the
smoothly evolving regular double bubble~$(\bar\Omega_1,\bar\Omega_2,\bar\Omega_3)$ on~$[0,T]$,
and that analogous estimates hold true for 
$\mathcal{R}\in\{R_{\taTrJ'},R''_{\taTrJ'},R_{\no'}\}$
in $\mathcal{N}_r(\trLine)\cap\mathrm{im}(\Psi')$,
or for $\mathcal{R}\in\{R_{\taTrJ''},R'_{\taTrJ''},R_{\no''}\}$
in $\mathcal{N}_r(\trLine)\cap\mathrm{im}(\Psi'')$.

For a Herring rotation~$\mathcal{R}\in\{R_{\taTrJ}',R_{\taTrJ}''\}$,
the claim~\eqref{eq:regEstimateRotations} follows directly from
the regularity of the frame~$(\no,\ta,\taTrJ)$, see~\eqref{eq:regularityNormalCurvature}
and~\eqref{eq:regTangentFields}, since the associated angles~$\theta',\theta''$
are independent of $(x,t)\in\mathcal{N}_r(\trLine)\cap\mathrm{im}(\Psi)$,
see Construction~\ref{RotationsTripleLine}. In terms of the
gauge rotation~$\mathcal{R}=R_{\no}$, it suffices to show that
\begin{align}
\label{eq:regGaugeRotationAngles}
|(\nabla,\nabla^2,\partial_t)(\delta,\omega)| \leq C
\quad\text{in } \mathcal{N}_r(\trLine)\cap\mathrm{im}(\Psi)
\end{align}
for the associated angles~$(\delta,\omega)$ defined in~\eqref{eq:rotationNormalAngle}
and~\eqref{eq:rotationNormalAngle2}, respectively. For the angle~$\delta$,
the regularity estimate from the previous display can be deduced 
from the regularity~\eqref{eq:regularityNormalCurvature} of the normal~$\no$.
The regularity estimate for the angle~$\omega$ in turn follows
from the regularity~\eqref{eq:regProjectionSignedDistance} of the projection onto the interface~$\barI$,
the regularity~\eqref{eq:regTangentFields} of the tangent vector fields~$(\ta,\taTrJ)$,
and from explicitly integrating (in each time slice) the ODE~\eqref{ODE_omega} along the 
integral lines of the tangent vector field~$\ta$.

We next claim that there exist constants~$c_1,c_2\in (-1,1)$
only depending on the surface tensions such that
\begin{align}
\label{eq:repTangentialVelocityAlongTripleLine}
\alpha_{\mathrm{vel}}(x,t)
= (1{-}c_1^2)^{-1}c_2\big(H'(x,t) - c_1 H(x,t)\big)
\end{align}
for all~$(x,t)\in\trLine$. For a proof of~\eqref{eq:repTangentialVelocityAlongTripleLine},
we define~$c_1:=\ta(x,t)\cdot\ta'(x,t)$
and $c_2:=\no'(x,t)\cdot\ta(x,t)
=-\no(x,t)\cdot\ta'(x,t)$, and
then simply observe from~\eqref{eq:uniqueVelocityAtTripleLine} and~\eqref{eq:initialValueODEalpha} that
\begin{align*}
\alpha_{\mathrm{vel}}(x,t)
&= c_2H'(x,t) + c_1\alpha'_{\mathrm{vel}}(x,t),
\\
\alpha'_{\mathrm{vel}}(P_{\trLine}(x,t),t)
&= -c_2H(x,t) + c_1\alpha_{\mathrm{vel}}(x,t) \quad \text{on } \trLine.
\end{align*}
Inserting the second identity of the previous display into the first one then 
directly yields the claim~\eqref{eq:repTangentialVelocityAlongTripleLine}.

The upshot of~\eqref{eq:regEstimateRotations} and~\eqref{eq:repTangentialVelocityAlongTripleLine}
is now the following. First, it follows from~\eqref{eq:alpha_coefficient_xi},
the regularity of the projection onto the triple line~$\trLine$ (cf.\ Definition~\ref{def:locRadius}~\textit{i)}),
the regularity~\eqref{eq:regTangentFields} of the tangent~$\ta$,
the representation~\eqref{eq:repTangentialVelocityAlongTripleLine} 
and finally the regularity~\eqref{eq:regularityNormalCurvature}
of the extended scalar mean curvatures that~$\alpha_{\trLine}(x,t):=\alpha(P_{\trLine}(x,t),t)$
satisfies
\begin{align*}
|\alpha_{\trLine}| + |(\nabla,\nabla^2,\partial_t)\alpha_{\trLine}| \leq C
\quad\text{in } \mathcal{N}_r(\trLine)\cap\mathrm{im}(\Psi).
\end{align*}
The previous display in combination with~\eqref{eq:regEstimateRotations}
and the expansion ansatz~\eqref{eq:AnsatzAuxiliaryXiHalfSpace} finally
implies the asserted regularity estimate~\eqref{eq:regEstimateAuxiliarxExtensionsNormal}.

\textit{Step 2: First order compatibility along triple line.} The zeroth
order conditions~\eqref{eq:compatibilityTripleLineXi} are immediate
from the definitions~\eqref{eq:AnsatzAuxiliaryXiHalfSpace}
as well as the identities~\eqref{eq:rotationTangentAngle1} and \eqref{eq:rotationTangentAngle2},
respectively. For a proof of the first order condition, we focus on deriving the first
identity of~\eqref{eq:compatibilityTripleLineXi2}. The second follows along the same lines.

Recalling the definition~\eqref{eq:defGaugedHerringRotation} of the gauged Herring rotation
and the gauged expansion ansatz~\eqref{eq:AnsatzAuxiliaryXiHalfSpace} we compute on the interface~$\barI$
(abbreviating~$\alpha_{\trLine}(\cdot,t):= \alpha(P_{\trLine}(\cdot,t),t)$ for~$t\in [0,T]$)
\begin{align}
\label{eq:derivGaugedRotatedAnsatz1}
\nabla \big(\widetilde R'_{\barI}\,\widetilde\xi\,\big) &=
(\nabla R_{\no}) R'_{\taTrJ}\no + R_{\no}\nabla \big(R'_{\taTrJ}\no\big)
+ \alpha_{\trLine}\big(R_{\no}R_{\taTrJ}'\ta\big)\otimes\no.
\end{align}
Let us now first compute $\nabla\big(R'_{\taTrJ}\no\big)$ and neglect the gauge rotations for a while. Recalling the fact that
$R'_{\taTrJ}$ is a rotation around the $\taTrJ$-axis with constant angle, see \eqref{eq:rotationTangent},
we obtain on the interface~$\bar I$
\begin{align*}
\nabla \big(R'_{\taTrJ}\no\big) = \nabla\big((R'_{\taTrJ}\no\cdot\no)\no+(R'_{\taTrJ}\no\cdot\ta)\ta\big)
= (R'_{\taTrJ}\no\cdot\no)\,\nabla\no + (R'_{\taTrJ}\no\cdot\ta)\,\nabla\ta. 
\end{align*}
Plugging in the identities \eqref{eq:gradientNormal} and \eqref{eq:gradientTau}, and using
in a second step that $R'_{\taTrJ}\no\cdot\no=R'_{\taTrJ}\ta\cdot\ta$
as well as $R'_{\taTrJ}\no\cdot\ta=-R'_{\taTrJ}\ta\cdot\no$, we further compute
\begin{align}
\nonumber
\nabla \big(R'_{\taTrJ}\no\big) &= 
-\kappa_{\ta\ta}\big((R'_{\taTrJ}\no\cdot\no)\,\ta\otimes\ta
-(R'_{\taTrJ}\no\cdot\ta)\,\no\otimes\ta\big)
\\&~~~\nonumber
-\big((R'_{\taTrJ}\no\cdot\no)\kappa_{\ta\taTrJ}
+(R'_{\taTrJ}\no\cdot\ta) (\nabla\cdot\taTrJ)\big)\,\taTrJ\otimes\ta
\\&~~~\nonumber
-\kappa_{\ta\taTrJ}\big((R'_{\taTrJ}\no\cdot\no)\,\ta\otimes\taTrJ
-(R'_{\taTrJ}\no\cdot\ta)\,\no\otimes\taTrJ\big)
\\&~~~\nonumber
-\big((R'_{\taTrJ}\no\cdot\no)\kappa_{\taTrJ\taTrJ}
-(R'_{\taTrJ}\no\cdot\ta)(\nabla\cdot\ta)\big)\,\taTrJ\otimes\taTrJ
\\&
\label{eq:auxRepDerivRotatedXiDiv}
= -\kappa_{\ta\ta}R'_{\taTrJ}\ta\otimes\ta
\\&~~~\nonumber
-\big((R'_{\taTrJ}\no\cdot\no)\kappa_{\ta\taTrJ}
+(R'_{\taTrJ}\no\cdot\ta)(\nabla\cdot\taTrJ)\big)\,\taTrJ\otimes\ta
\\&~~~\nonumber
-\kappa_{\ta\taTrJ}R'_{\taTrJ}\ta\otimes\taTrJ
\\&~~~\nonumber
-\big((R'_{\taTrJ}\no\cdot\no)\kappa_{\taTrJ\taTrJ}
-(R'_{\taTrJ}\no\cdot\ta)(\nabla\cdot\ta)\big)\,\taTrJ\otimes\taTrJ,
\end{align}
which holds true on the interface $\barI$.

Recalling the choice~\eqref{eq:alpha_coefficient_xi} for $\alpha$,  we may infer
from the formula~\eqref{eq:auxRepDerivRotatedXiDiv} for~$\nabla \big(R'_{\taTrJ}\no\big)$, 
substituting $\kappa_{\ta\ta}= H - \kappa_{\taTrJ\taTrJ}$ along~$\barI$,
the identity~\eqref{eq:gaugeRotationOnTripleLine}, and the formula~\eqref{eq:derivGaugedRotatedAnsatz1}
the following representation for the gradient of $\widetilde R'_{\barI}\,\widetilde\xi$
along the triple line~$\trLine$
\begin{align}
\nabla \big(\widetilde R'_{\barI}\,\widetilde\xi\big) 
\label{eq:gradientRotatedXi}
&= R'_{\taTrJ}\ta\otimes (-H\ta + \alpha_{\mathrm{vel}}\no)
+ R'_{\taTrJ}\ta\otimes (\kappa_{\taTrJ\taTrJ}\ta + (\nabla\cdot\ta)\no)
\\&~~~
\nonumber
+ \big((\taTrJ\cdot\nabla)\widetilde R'_{\barI}\,\widetilde\xi\,\big) \otimes \taTrJ
\\&~~~
\nonumber
-\big((R'_{\taTrJ}\no\cdot\no)\kappa_{\ta\taTrJ}
+(R'_{\taTrJ}\no\cdot\ta)(\nabla\cdot\taTrJ)\big)\,\taTrJ\otimes\ta
\\&~~~
\nonumber
+ (\nabla R_{\no}) R'_{\taTrJ}\no.
\end{align} 
A direct computation based on the ansatz~\eqref{eq:AnsatzAuxiliaryXiHalfSpace}, 
the identities~\eqref{eq:gradientNormal},~\eqref{eq:gradientTau}, and~\eqref{eq:gaugeRotationOnTripleLine}, and
substituting~$\kappa'_{\ta'\ta'}=H'-\kappa'_{\taTrJ'\taTrJ'}$
also yields along the triple line~$\trLine$
\begin{align}
\label{eq:gradientRotatedXiPrime}
\nabla \widetilde\xi' &= \ta'\otimes (-H'\ta' + \alpha'_{\mathrm{vel}}\no')
+ \ta'\otimes (\kappa'_{\taTrJ'\taTrJ'}\ta' + (\nabla\cdot\ta')\no')
\\&~~~
\nonumber
+ (\taTrJ'\cdot\nabla)\widetilde\xi'\otimes\taTrJ'
\\&~~~
\nonumber
- \kappa'_{\ta'\taTrJ'}\,\taTrJ'\otimes\ta'
\\&~~~
\nonumber
+ (\nabla R_{\no'}) \no'.
\end{align}

We proceed by comparing the respective formulas~\eqref{eq:gradientRotatedXi}
and~\eqref{eq:gradientRotatedXiPrime}.
Recalling that we denoted by~$\vec{V}_{\trLine}$ the normal velocity vector field of
the triple line, we obtain from~\eqref{eq:uniqueVelocityAtTripleLine}, 
the choice of $\alpha_{\mathrm{vel}}$~\eqref{eq:initialValueODEalpha},
the identities~\eqref{eq:defNinetyDegreeRotation} and~\eqref{eq:rotationNinetyDegreeTripleLine},
as well as the zeroth order compatibility~\eqref{eq:compConditionTripleLine} along the triple line that 
the first terms in~\eqref{eq:gradientRotatedXi} and~\eqref{eq:gradientRotatedXiPrime} are identical:
\begin{align*}
R'_{\taTrJ}\ta\otimes (-H\ta + \alpha_{\mathrm{vel}}\no)
= -\ta'\otimes J\vec{V}_{\trLine} = \ta'\otimes (-H'\ta' + \alpha'_{\mathrm{vel}}\no') \quad \text{along } \trLine. 
\end{align*}
Moreover, by the compatibility conditions~\eqref{eq:compConditionTripleLine}, 
\eqref{eq:compConditionTripleLine2} and~\eqref{eq:compConditionTripleLine4} along the triple line, 
as well as $R'_{\taTrJ}\ta\cdot\ta=R'_{\taTrJ}\no\cdot\no$ and
$R'_{\taTrJ}\ta\cdot\no=-R'_{\taTrJ}\no\cdot\ta$, we may infer that the second terms agree, too:
\begin{align*}
R'_{\taTrJ}\ta\otimes (\kappa_{\taTrJ\taTrJ}\ta + (\nabla\cdot\ta)\no) = 
\ta'\otimes (\kappa'_{\taTrJ'\taTrJ'}\ta' + (\nabla\cdot\ta')\no') \quad \text{along } \trLine.
\end{align*}
From the last two identities together with~\eqref{eq:gradientRotatedXi}, \eqref{eq:gradientRotatedXiPrime},
\eqref{eq:compatibilityTripleLineXi}, and~\eqref{eq:orientationTangentTripleLine} 
we therefore obtain along the triple line $\trLine$
\begin{align}
\label{eq:firstOrderCompWithoutRotationNormalAxis}
\nabla \big(\widetilde R_{\barI}\,\widetilde\xi\big) - \nabla\widetilde\xi'
&= -\big((R'_{\taTrJ}\no\cdot\no)\kappa_{\ta\taTrJ}
+(R'_{\taTrJ}\no\cdot\ta) (\nabla\cdot\taTrJ)\big)\,\taTrJ\otimes\ta
+ \kappa'_{\ta'\taTrJ'}\,\taTrJ\otimes\ta'
\\&~~~ \nonumber
+ (\nabla R_{\no})R'_{\taTrJ}\no - (\nabla R_{\no'})\no'.
\end{align}
In the rotationally symmetric case, 
the right hand side terms in the first line of~\eqref{eq:firstOrderCompWithoutRotationNormalAxis}
actually vanish. However, there is no reason
in general why these terms should vanish without assuming additional symmetry.
This is the motivation for the introduction of the additional gauge rotation matrices around the normal axis.
Their definition is arranged in such a way so that their contribution in~\eqref{eq:firstOrderCompWithoutRotationNormalAxis}
exactly cancels the right hand side terms of the first line.

First, we obtain from the definitions~\eqref{eq:defGaugeRotation}--\eqref{ODE_omega} 
along the triple line
\begin{align}
\label{eq:derivativeRotationNormal1}
(\nabla R_{\no})R'_{\taTrJ}\no &=
\big((\ta\cdot\nabla)R^{(2)}_{\no}\big)R'_{\taTrJ}\no\otimes\ta
+ \big((\no\cdot\nabla)R^{(1)}_{\no}\big)R'_{\taTrJ}\no\otimes\no.
\end{align}
Let us next compute the two relevant directional derivatives of the gauge
rotation matrices. We first observe that due to~\eqref{eq:rotationNormal}
and~\eqref{eq:rotationNormalAngle}
\begin{align}
\label{eq:derivativeRotationNormal2}
(\no\cdot\nabla)R^{(1)}_{\no} = \kappa_{\ta\taTrJ}\,\taTrJ\wedge\ta
\end{align} 
along the interface $\barI$. This in turn entails
by~$R'_{\taTrJ}\ta\cdot\no=-R'_{\taTrJ}\no\cdot\ta$ 
\begin{align}
\label{eq:derivativeRotationNormal3}
\big((\no\cdot\nabla)R^{(1)}_{\no}\big)R'_{\taTrJ}\no\otimes\no
= -\kappa_{\ta\taTrJ}(R'_{\taTrJ}\ta\cdot\no)\,\taTrJ\otimes\no
\quad\text{along } \trLine.
\end{align}
Moreover, we may compute based on~\eqref{eq:rotationNormal2},~\eqref{eq:rotationNormalAngle2}, and~\eqref{ODE_omega} on the triple line~$\trLine$ 
\begin{align}
\label{eq:derivativeRotationNormal4}
(\ta\cdot\nabla)R^{(2)}_{\no}
= (\nabla\cdot\taTrJ)\,\taTrJ\wedge\ta,
\end{align}
from which we deduce
\begin{align}
\label{eq:derivativeRotationNormal5}
\big((\ta\cdot\nabla)R^{(2)}_{\no}\big)R'_{\taTrJ}\no\otimes\ta
= \big((R'_{\taTrJ}\no\cdot\ta)(\nabla\cdot\taTrJ)\big)\,\taTrJ\otimes\ta
\quad\text{along } \trLine.
\end{align}
A straightforward computation  
shows that along the triple line~$\trLine$ it holds
\begin{align}
\label{eq:derivativeRotationNormal6}
(\nabla R_{\no'})\no' = \nabla \big(R_{\no'}\no'\big) - R_{\no'} \nabla \no' = \nabla \no' - \nabla \no' =0.
\end{align}
Combining~\eqref{eq:derivativeRotationNormal1},~\eqref{eq:derivativeRotationNormal3},~\eqref{eq:derivativeRotationNormal5}, 
and~\eqref{eq:derivativeRotationNormal6} 
with the compatibility conditions~\eqref{eq:compConditionTripleLine}
and~\eqref{eq:compConditionTripleLine1} finally yields
the desired cancellation
\begin{align*}
-\big((R'_{\taTrJ}\no\cdot\no)\kappa_{\ta\taTrJ}
+(R'_{\taTrJ}\no\cdot\ta) (\nabla\cdot\taTrJ)\big)\,\taTrJ\otimes\ta
+ \kappa'_{\ta'\taTrJ'}\,\taTrJ\otimes\ta'
+ (\nabla R_{\no})R'_{\taTrJ}\no - (\nabla R_{\no'})\no' = 0
\end{align*}
along the triple line~$\trLine$. By~\eqref{eq:firstOrderCompWithoutRotationNormalAxis}, 
this in turn concludes the proof of Lemma~\ref{lem:firstOrderCompXi}.
\end{proof}

We proceed with the construction of suitable candidate velocity fields.

\begin{construction}[Extension of velocity fields close to their associated interfaces]
\label{AnsatzAuxiliaryVelocityHalfSpace}
Recall that~$\vec{V}_{\trLine}$ denotes the normal velocity of the triple line~$\trLine$,
and recall the definition~\eqref{eq:defGaugedOrthonormalFrame} of the gauged
orthonormal frame~$(\no,\ta_*,\taTrJ_*)$.
We then define a coefficient function
\begin{align}
\label{eq:coefficientTangentialVelocity}
\alpha_{\mathrm{vel}}\colon\mathcal{N}_r(\trLine)\cap\mathrm{im}(\Psi)\to\mathbb{R},
\quad &(x,t) \mapsto \widehat\alpha_{\mathrm{vel}}(P(x,t),t),
\end{align}
where the coefficient~$\alpha_{\mathrm{vel}}$ is defined by projection onto the interface~$\barI$ in terms of the
solution of the following family of ODEs, solved along the integral lines of the 
tangent vector field~$\ta_*$ with initial condition posed on the triple line~$\trLine$
\begin{align}
\label{ODE_alpha2}
  \begin{cases}
	\widehat\alpha_{\mathrm{vel}}(x,t) & = (\ta_*\cdot\vec{V}_{\trLine})(x,t),
	\hspace*{0.43cm} (x,t)\in\trLine, 
	\\
	\left(\ta_* \cdot \nabla \right) \widehat\alpha_{\mathrm{vel}}(x,t) & = 
	(H\kappa_{\ta_*\ta_*})(x,t), 
	\quad (x,t)\in \barI \cap \mathcal{N}_r(\trLine).
  \end{cases}
\end{align}
Note that the choice of the initial value in~\eqref{ODE_alpha2} is consistent with~\eqref{eq:initialValueODEalpha}. 
Next, we define another coefficient function
\begin{align}
\label{eq:coefficientBeta}
\beta\colon\mathcal{N}_r(\trLine)\cap\mathrm{im}(\Psi)\to\mathbb{R},
\quad (x,t) \mapsto -\big((\ta_*\cdot\nabla)H\big)(x,t) 
- (\alpha_{\mathrm{vel}}\kappa_{\ta_*\ta_*})(x,t).
\end{align}

We now define a preliminary extension~$\widetilde B\colon\mathcal{N}_r(\trLine)\cap\mathrm{im}(\Psi)\to\mathbb{R}^3$
of the normal velocity vector field $(H\no)|_{\barI}$ 
for the interface~$\barI$ in terms of the \emph{gauged expansion ansatz}
\begin{align}
\label{eq:AnsatzAuxiliaryVelocityHalfSpace}
	\widetilde B(x,t)  
	& := 
	H(x,t) \, \no(x,t)
	\\&~~~~\nonumber
	+ \alpha_{\mathrm{vel}}(x,t) \, \ta_*(x,t) 
	\\&~~~~\nonumber 
	+\beta(x,t) s(x,t) \, \ta_*(x,t)
\end{align}
for all~$(x,t)\in\mathcal{N}_r(\trLine)\cap\mathrm{im}(\Psi)$.

Analogously, one defines preliminary extensions $\widetilde B'\colon\mathcal{N}_r(\trLine)\cap\mathrm{im}(\Psi')\to\mathbb{R}^3$ 
as well as $\widetilde B''\colon\mathcal{N}_r(\trLine)\cap\mathrm{im}(\Psi'')\to\mathbb{R}^3$ of the normal velocity
vector fields~$(H'\no')|_{\barI'}$ and $(H''\no'')|_{\barI''}$, respectively.
\hfill$\diamondsuit$
\end{construction}
%
%

Note carefully that even away from the triple line
we do not introduce a tangential velocity in $\taTrJ_*$-direction.
As the proof of the following result shows, this will entail
that the gradients of the auxiliary velocities~$\widetilde B$, $\widetilde B'$
and~$\widetilde B''$ do not fully match along the triple line.
However, the only mismatch appears in, at least for our purposes, inessential components. More precisely,
in terms of, say, $\nabla \widetilde B$ the only non-matching
terms result from its $\taTrJ_*\otimes\ta_*$ resp.\ $\taTrJ_*\otimes\no$ component. In view of
the desired evolution equation~\eqref{TransportEquationXi}
and the fact that~$\widetilde\xi \perp \taTrJ_*$ due to~\eqref{eq:AnsatzAuxiliaryXiHalfSpace},
this specific component of~$\nabla \widetilde B$ is intrinsically irrelevant
for a gradient-flow calibration
(this argument turns out to be robust even with respect to the interpolation construction from
Subsection~\ref{subsection:step_2}).

\begin{lemma}
\label{lem:firstOrderCompVelocity}
Let~$(\widetilde B,\widetilde B',\widetilde B'')$
be the preliminary extensions from Construction~\ref{AnsatzAuxiliaryVelocityHalfSpace}
of the normal velocity vector fields~$((H\no)|_{\barI}, (H'\no')|_{\barI'}, (H''\no'')|_{\barI''})$.

Then it holds $\widetilde B \in C^0_tC^2_x(\mathcal{N}_r(\trLine)\cap\mathrm{im}(\Psi))$
with corresponding estimate
\begin{align}
\label{eq:regEstimateAuxiliaryExtensionVelTripleLine}
|\widetilde B| + |\nabla \widetilde B| + |\nabla^2 \widetilde B|
\leq C \quad\text{in } \mathcal{N}_r(\trLine)\cap\mathrm{im}(\Psi),
\end{align}
where the constant~$C>0$ only depends on the data of the
smoothly evolving regular double bubble~$(\bar\Omega_1,\bar\Omega_2,\bar\Omega_3)$ on~$[0,T]$.
Analogous claims hold true for~$\widetilde B'$ resp.\ $\widetilde B''$
throughout~$\mathcal{N}_r(\trLine)\cap\mathrm{im}(\Psi')$
resp.\ $\mathcal{N}_r(\trLine)\cap\mathrm{im}(\Psi'')$.

Moreover, the constructions are essentially compatible to first order 
in the sense that along the triple line~$\trLine$ it holds
\begin{align}
\label{eq:compatibilityTripleLineVelocity}
\widetilde B &= \widetilde B' = \widetilde B'' = \vec{V}_{\trLine},
\\
\label{eq:compatibilityTripleLineVelocityGradient}
(\mathrm{Id}{-}\taTrJ\otimes\taTrJ)(\nabla\widetilde B) &= 
(\mathrm{Id}{-}\taTrJ'\otimes\taTrJ')(\nabla\widetilde B') = 
(\mathrm{Id}{-}\taTrJ''\otimes\taTrJ'')(\nabla\widetilde B''),
\end{align}
for which one should also recall that $\taTrJ=\taTrJ'=\taTrJ''$ along~$\trLine$,
cf.\ \eqref{eq:orientationTangentTripleLine}.
\end{lemma}

Note that here the projection $\mathrm{Id}{-}\taTrJ\otimes\taTrJ$ acts on the components of $\widetilde B$, not $\nabla$.

\begin{proof}
\textit{Step 1: Regularity estimates.} 
Due to the definition~\eqref{eq:coefficientBeta}, the regularity
estimates~\eqref{eq:regEstimateRotations} for the gauge rotations,
the regularity of the frame~$(\no,\ta,\taTrJ)$, see~\eqref{eq:regularityNormalCurvature}
and~\eqref{eq:regTangentFields}, the regularity~\eqref{eq:regularityNormalCurvature}
of the extended scalar mean curvatures, and finally the
expansion ansatz~\eqref{eq:AnsatzAuxiliaryVelocityHalfSpace}
it suffices to prove that
\begin{align}
\label{eq:regEstimateTangentialVelGeneral}
|\alpha_{\mathrm{vel}}| + |(\nabla,\nabla^2)\alpha_{\mathrm{vel}}|
\leq C \quad\text{in } \mathcal{N}_r(\trLine)\cap\mathrm{im}(\Psi),
\end{align}
where~$C>0$ is a constant which depends only on the data of the
smoothly evolving regular double bubble~$(\bar\Omega_1,\bar\Omega_2,\bar\Omega_3)$ on~$[0,T]$.

The estimate~\eqref{eq:regEstimateTangentialVelGeneral} in turn follows directly
from explicitly integrating (in each time slice) the ODEs~\eqref{ODE_alpha2}
along the integral lines of the tangent field~$\ta_*$,
and exploiting as before the regularity of the associated geometric quantities.

\textit{Step 2: Zeroth order compatibility at triple line.}
The condition~\eqref{eq:compatibilityTripleLineVelocity} is
immediate from the definition~\eqref{eq:AnsatzAuxiliaryVelocityHalfSpace},
the identities~\eqref{eq:uniqueVelocityAtTripleLine},
and the specific 
choices~\eqref{eq:coefficientTangentialVelocity}--\eqref{ODE_alpha2}.

\textit{Step 3: First order compatibility at triple line.}
We proceed with the proof of~\eqref{eq:compatibilityTripleLineVelocityGradient}.
Observe that we have on the interface~$\barI$ 
by direct analogy to the proofs of~\eqref{eq:gradientNormal} and~\eqref{eq:gradientTau} that
\begin{align}
\label{eq:gradientNormal2}
\nabla\no &= -\kappa_{\ta_*\ta_*}\,\ta_*\otimes\ta_* -\kappa_{\taTrJ_*\taTrJ_*}\,\taTrJ_*\otimes\taTrJ_*
-\kappa_{\ta_*\taTrJ_*}\,(\taTrJ_*\otimes\ta_* + \ta_*\otimes\taTrJ_*),
\\
\label{eq:gradientTau2}
\nabla\ta_* &= \kappa_{\ta_*\ta_*}\,\no\otimes\ta_* - (\nabla\cdot\taTrJ_*)\,\taTrJ_*\otimes\ta_*
+ \kappa_{\ta_*\taTrJ_*}\,\no\otimes\taTrJ_* 
+ (\nabla\cdot\ta_*)\,\taTrJ_*\otimes\taTrJ_*
\\&~~~\nonumber
+ (\no\cdot\nabla)\ta_*\otimes\no.
\end{align}
It follows directly from the definitions~\eqref{eq:frameInterface}
resp.\ \eqref{eq:defGaugedOrthonormalFrame} of our orthonormal frames,
the definitions~\eqref{eq:defGaugeRotation}--\eqref{ODE_omega} of the gauge rotations,
as well as the formula~\eqref{eq:derivativeRotationNormal2} being valid along the interface~$\barI$ that
\begin{align*}
(\no\cdot\nabla)\ta_* = R_{\no}^{(2)}\big((\no\cdot\nabla)R_{\no}^{(1)}\big)\ta
= \kappa_{\ta\taTrJ}R_{\no}\taTrJ = \kappa_{\ta\taTrJ}\taTrJ_*
\quad\text{along } \barI.
\end{align*}

Starting now from the definition~\eqref{eq:AnsatzAuxiliaryVelocityHalfSpace}, the previous display,
the choices~\eqref{eq:coefficientTangentialVelocity}--\eqref{eq:coefficientBeta}
of the coefficient functions, as well as the formulas~\eqref{eq:gradientNormal2}
and~\eqref{eq:gradientTau2} directly entail along the interface~$\barI$
\begin{align}
\nonumber
\nabla \widetilde B
&= \beta\,\ta_*\otimes\no + \big((\ta_*\cdot\nabla)H 
+ \widehat\alpha_{\mathrm{vel}}\kappa_{\ta_*\ta_*}\big)\,\no\otimes\ta_*
\\&~~~ \nonumber
+ \big((\ta_*\cdot\nabla)\widehat\alpha_{\mathrm{vel}} - H\kappa_{\ta_*\ta_*}\big)\,\ta_*\otimes\ta_*
\\&~~~ \nonumber
+ (\taTrJ_*\cdot\nabla)\widetilde B\otimes\taTrJ_*
\\&~~~ \nonumber
- \big(H\kappa_{\ta_*\taTrJ_*} + \widehat\alpha_{\mathrm{vel}}(\nabla\cdot\taTrJ_*)\big)\,\taTrJ_*\otimes\ta_*
+ \kappa_{\ta\taTrJ}\taTrJ_*\otimes\no
\\& \label{eq:derivAuxiliaryVelocityTripleLine}
= \beta\,\ta_* \wedge \no + (\taTrJ_*\cdot\nabla)\widetilde B\otimes\taTrJ_*
\\&~~~ \nonumber 
- \big(H\kappa_{\ta_*\taTrJ_*} + \widehat\alpha_{\mathrm{vel}}(\nabla\cdot\taTrJ_*)\big)\,\taTrJ_*\otimes\ta_*
+ \kappa_{\ta\taTrJ}\taTrJ_*\otimes\no.
\end{align}
Hence, the already established zeroth order condition~\eqref{eq:compatibilityTripleLineVelocity}
together with the compatibility conditions~\eqref{eq:compConditionTripleLine7} and~\eqref{eq:compConditionTripleLine8}
in form of~$\beta=\beta'=\beta''$ along~$\trLine$ imply~\eqref{eq:compatibilityTripleLineVelocityGradient}.
\end{proof}

The following result provides the approximate evolution equations for our 
auxiliary constructions $(\widetilde\xi,\widetilde R'_{\barI}\,\widetilde\xi,\widetilde R''_{\barI}\,\widetilde\xi\,)$
in terms of the associated auxiliary velocity~$\widetilde B$, which will eventually 
lead us to~\eqref{TransportEquationXi}--\eqref{Dissip}.

\begin{lemma}
\label{lem:propertiesInterfaceWedgesCalibration}
Let~$(\widetilde\xi,\widetilde\xi',\widetilde\xi'')$
be the initial extensions from Construction~\ref{AnsatzAuxiliaryXiHalfSpace} 
of the normal vector fields~$(\no|_{\barI}, \no'|_{\barI'}, \no''|_{\barI''})$. Moreover,
let~$(\widetilde R'_{\barI},\widetilde R''_{\barI})$, $(\widetilde R_{\barI'},\widetilde R''_{\barI'})$
and~$(\widetilde R_{\barI''}, \widetilde R'_{\barI''})$ be the gauged Herring rotations
as provided by Construction~\ref{RotationsTripleLine}, respectively.
Finally, let~$(\widetilde B,\widetilde B',\widetilde B'')$
be the initial extensions from Construction~\ref{AnsatzAuxiliaryVelocityHalfSpace}
of the normal velocity vector fields~$((H\no)|_{\barI}, (H'\no')|_{\barI'}, (H''\no'')|_{\barI''})$.

Then there exists a constant~$C>0$, which depends only on the data of the
smoothly evolving regular double bubble~$(\bar\Omega_1,\bar\Omega_2,\bar\Omega_3)$ on~$[0,T]$,
such that for each rotation $\mathcal{R}\in\{\mathrm{Id},\widetilde R'_{\barI},\widetilde R''_{\barI}\}$
it holds
\begin{align}
\label{eq:lengthControlXiHalfspace}
\big|1-|\mathcal{R}\widetilde\xi|^2\big| 
&\leq C\dist^4(\cdot,\bar I),
\\
\label{eq:gradientLengthXi}
\big|\nabla|\mathcal{R}\widetilde\xi|^2\big|
&\leq C\dist^3(\cdot,\bar I),
\\
\label{eq:timeDerivLengthXi}
\big|\partial_t|\mathcal{R}\widetilde\xi|^2\big|
&\leq C\dist^3(\cdot,\bar I),
\\
\label{eq:evolEquXiHalfspace}
\big|\partial_t\mathcal{R}\widetilde\xi
+ (\widetilde B\cdot\nabla)\mathcal{R}\widetilde\xi
+ (\nabla\widetilde B)^\mathsf{T}\mathcal{R}\widetilde\xi\,\big|
&\leq C\begin{cases}
			 \dist(\cdot,\bar I) &\text{if } \mathcal{R}=\mathrm{Id}, 
			 \\
			 \dist(\cdot,\trLine) &\text{else},
			 \end{cases}
\\
\label{eq:divConstraintXiHalfspace}
\big|\nabla\cdot\mathcal{R}\widetilde\xi 
+ \widetilde B\cdot \mathcal{R}\widetilde\xi\,\big|
&\leq C\begin{cases}
			 \dist(\cdot,\bar I) &\text{if } \mathcal{R}=\mathrm{Id}, 
			 \\
			 \dist(\cdot,\trLine) &\text{else},
			 \end{cases}
\end{align}
throughout the domain~$\mathcal{N}_r(\trLine)\cap\mathrm{im}(\Psi)$.

Analogous estimates hold true throughout the domain~$\mathcal{N}_r(\trLine)\cap\mathrm{im}(\Psi')$
in terms of the vector fields~$(\mathcal{R}\widetilde\xi',\widetilde B')$
for each rotation $\mathcal{R}\in\{\widetilde R_{\barI'},\mathrm{Id},\widetilde R''_{\barI'}\}$,
as well as throughout~$\mathcal{N}_r(\trLine)\cap\mathrm{im}(\Psi'')$
in terms of~$(\mathcal{R}\widetilde\xi'',\widetilde B'')$
for each $\mathcal{R}\in\{\widetilde R_{\barI''},\widetilde R'_{\barI''},\mathrm{Id}\}$.
\end{lemma}

\begin{proof}
Fix a rotation $\mathcal{R}\in\{\mathrm{Id},\widetilde R'_{\barI},\widetilde R''_{\barI}\}$,
and for the purposes of the proof abbreviate~$\alpha_{\trLine}(\cdot,t)
:= \alpha(P_{\trLine}(\cdot,t),t)$, $t\in [0,T]$.

\textit{Step 1: Proof of~\emph{\eqref{eq:lengthControlXiHalfspace}--\eqref{eq:timeDerivLengthXi}}.}
It follows immediately from the ansatz \eqref{eq:AnsatzAuxiliaryXiHalfSpace}
and the orthogonality~$\ta_*\cdot\no=0$ that
\begin{align*}
|\mathcal{R}\widetilde\xi|^2 = |\widetilde\xi|^2 =
\Big(1-\frac{1}{2}\alpha_{\trLine}^2 s^2\Big)^2 + \alpha_{\trLine}^2s^2 
= 1 + \frac{1}{4}\alpha_{\trLine}^4s^4.
\end{align*}
The previous display of course immediately implies the 
estimates~\eqref{eq:lengthControlXiHalfspace}--\eqref{eq:timeDerivLengthXi}.

\textit{Step 2: Proof of~\eqref{eq:divConstraintXiHalfspace}.} By the regularity
estimates~\eqref{eq:regEstimateAuxiliarxExtensionsNormal} 
and~\eqref{eq:regEstimateAuxiliaryExtensionVelTripleLine}, 
it suffices to show that~\eqref{eq:divConstraintXiHalfspace}
is exact~\textit{on} the interface~$\barI$ if~$\mathcal{R}=\mathrm{Id}$,
or otherwise that~\eqref{eq:divConstraintXiHalfspace}
is exact~\textit{on} the triple line~$\trLine$.  
To this end, let us first assume that $\mathcal{R}_{\taTrJ}=\mathrm{Id}$. Then we 
also have $\mathrm{R} =\mathrm{Id}$ and hence we may directly infer from the 
definitions \eqref{eq:AnsatzAuxiliaryXiHalfSpace} and \eqref{eq:AnsatzAuxiliaryVelocityHalfSpace} 
of $\widetilde \xi$ and $\widetilde B$, respectively, that $\nabla \cdot \widetilde \xi = 
H = \widetilde \xi \cdot \widetilde B$ on the interface $\barI$.
In the remaining cases, we express $\mathcal{R}=R_{\no}\mathcal{R}_{\taTrJ}R_{\no}^\mathsf{T}$
in terms of the associated Herring rotation $\mathcal{R}_{\taTrJ}\in\{
R'_{\taTrJ},R''_{\taTrJ}\}$, and then simply read off from~\eqref{eq:derivGaugedRotatedAnsatz1},
\eqref{eq:auxRepDerivRotatedXiDiv}, \eqref{eq:derivativeRotationNormal1},
\eqref{eq:derivativeRotationNormal3} and~\eqref{eq:derivativeRotationNormal5} that
\begin{align*}
\nabla\cdot\mathcal{R}\widetilde\xi = 
- H (\mathcal{R}_{\taTrJ}\no\cdot\no) + (\nabla\cdot\ta) (\mathcal{R}_{\taTrJ}\no\cdot\ta)
- \alpha_{\trLine} (\mathcal{R}_{\taTrJ}\no\cdot\ta)
\end{align*}
along the triple line~$\trLine$. Moreover, the definitions~\eqref{eq:AnsatzAuxiliaryXiHalfSpace} 
and~\eqref{eq:AnsatzAuxiliaryVelocityHalfSpace} directly imply that
\begin{align*}
\widetilde B \cdot \mathcal{R}\widetilde\xi
= H (\mathcal{R}_{\taTrJ}\no\cdot\no) 
+ \alpha_{\mathrm{vel}} (\mathcal{R}_{\taTrJ}\no\cdot\ta)
\end{align*}
holds true on the interface~$\barI$. Hence, the estimate~\eqref{eq:divConstraintXiHalfspace}
follows from the previous two displays in combination with the choice~\eqref{eq:alpha_coefficient_xi}.

\textit{Step 3: Proof of~\eqref{eq:evolEquXiHalfspace}.} It suffices
again to check that~\eqref{eq:evolEquXiHalfspace} is exact on 
the interface~$\barI$ if~$\mathcal{R}=\mathrm{Id}$,
or otherwise that~\eqref{eq:evolEquXiHalfspace}
is exact on the triple line~$\trLine$.  Let us also
again express $\mathcal{R}=R_{\no}\mathcal{R}_{\taTrJ}R_{\no}^\mathsf{T}$
in terms of the associated Herring rotation $\mathcal{R}_{\taTrJ}\in\{\mathrm{Id},
R'_{\taTrJ},R''_{\taTrJ}\}$.

Using that the vector field~$\mathcal{R}\no=R_{\no}\mathcal{R}_{\taTrJ}\no$ 
lies in the $(\no,R_{\no}\ta)$-plane and has constant coefficients in this frame, we compute
along the interface~$\barI$ relying also on~\eqref{eq:AnsatzAuxiliaryXiHalfSpace}
\begin{align}
\nonumber
&\partial_t\mathcal{R}\widetilde\xi
+ (\widetilde B\cdot\nabla)\mathcal{R}\widetilde\xi
+ (\nabla\widetilde B)^\mathsf{T}\mathcal{R}\widetilde\xi 
\\&
\label{eq:aux}
=(R_{\no}\mathcal{R}_{\taTrJ}\no\cdot\no)
\big(\partial_t\no + (\widetilde B\cdot \nabla)\no + (\nabla\widetilde B)^\mathsf{T}\no\big)
\\&~~~
\nonumber
+(R_{\no}\mathcal{R}_{\taTrJ}\no\cdot R_{\no}\ta)
\big(\partial_t\ta_* + (\widetilde B\cdot \nabla)\ta_* + (\nabla\widetilde B)^\mathsf{T}\ta_*\big)
\\&~~~
\nonumber
+ \alpha_{\trLine} (\partial_ts + (\widetilde B\cdot\nabla)s)\,\ta_*.
\end{align}
The last right hand side term of~\eqref{eq:aux} 
vanishes due to $(\widetilde B\cdot\nabla)s = H$
and~\eqref{eq:evolutionSignedDistance}.
Differentiating this equation in space yields because of~$\nabla s = \no$  
\begin{align*}
0 &= \nabla\big(\partial_ts + (\widetilde B\cdot\nabla)s\big)
= \partial_t\no + (\widetilde B\cdot \nabla)\no + (\nabla\widetilde B)^\mathsf{T}\no.
\end{align*}
Hence, also the first right hand side term of~\eqref{eq:aux} vanishes.
Since~$\mathcal{R}_{\taTrJ}=\mathrm{Id}$ if and only if~$\mathcal{R}=\mathrm{Id}$,
the estimate~\eqref{eq:evolEquXiHalfspace} already follows from these
arguments in the case~$\mathcal{R}=\mathrm{Id}$. Hence, let us restrict
to the case~$\mathcal{R}\neq \mathrm{Id}$ in the following. Recall from
the claim~\eqref{eq:evolEquXiHalfspace} that it then suffices to
estimate in terms of the distance to the triple line.

It follows from~$|\ta_*|=1$ that $\ta_*\cdot\big(\partial_t\ta_* + 
(\widetilde B\cdot\nabla)\ta_*\big)=0$. Furthermore, the ansatz for the
velocity field $\widetilde B$ is arranged such that
$\ta_*\otimes\ta_*:\nabla\widetilde B=0$; cf.\ the identity~\eqref{eq:derivAuxiliaryVelocityTripleLine}.
Hence, in the evolution equation for the tangent vector field $\ta_*$
we may neglect the $\ta_*$-component. The $\no$-component also 
vanishes as a consequence of the orthogonality $\ta_\ast \cdot \no =0$, the skew-symmetry 
$\ta_*\otimes\no:\nabla\widetilde B = -\no\otimes\ta_*:\nabla\widetilde B$,
cf.\ again~\eqref{eq:derivAuxiliaryVelocityTripleLine},
and the already established evolution equation for the unit normal vector field~$\no$
\begin{align*}
&\no\cdot \big(\partial_t\ta_* + (\widetilde B\cdot \nabla)\ta_* + (\nabla\widetilde B)^\mathsf{T}\ta_*\big)
= -\ta_*\cdot \big(\partial_t\no + (\widetilde B\cdot \nabla)\no + (\nabla\widetilde B)^\mathsf{T}\no\big)
= 0.
\end{align*}
It therefore suffices to check that the velocity field $\widetilde B$ correctly captures
the translation and rotation of the tangent vector field $\ta_*$ in $\taTrJ_*$-direction
\textit{on} the triple line~$\trLine$, i.e.,
$\taTrJ_*\cdot \big(\partial_t\ta_* + (\widetilde B\cdot \nabla)\ta_* 
+ (\nabla\widetilde B)^\mathsf{T}\ta_*\big) = 0$,
or equivalently by exploiting the orthogonality $\ta_*\cdot\taTrJ_*=0$ that
\begin{align}
\label{eq:auxclaim}
\ta_*\cdot \big(\partial_t\taTrJ_* + (\widetilde B\cdot \nabla)\taTrJ_*\big)
= \taTrJ_*\cdot(\nabla\widetilde B)^\mathsf{T}\ta_*
\end{align}
along the triple line~$\trLine$.

In order to prove \eqref{eq:auxclaim}, we start by noticing that as a consequence 
of the definition~\eqref{eq:AnsatzAuxiliaryVelocityHalfSpace}, as well as the 
formulas~\eqref{eq:gradientNormal2} and \eqref{eq:gradientTau2} we have
\begin{align}
\label{eq:auxRightHandSide1}
\taTrJ_*\cdot(\nabla\widetilde B)^\mathsf{T}\ta_*
= \ta_* \cdot (\taTrJ_* \cdot \nabla ) \widetilde B
=-H\kappa_{\ta_*\taTrJ_*} + (\taTrJ_*\cdot\nabla)\alpha_{\mathrm{vel}} \quad \text{on } \barI.
\end{align}
That this expression equals~$\ta_*\cdot \big(\partial_t\taTrJ_* + (\widetilde B\cdot \nabla)\taTrJ_*\big)$
on the triple line~$\trLine$ is a consequence of the following considerations.
Let $\psi_{\trLine}(\cdot,t)\colon\trLine^0{\times} [0,T] \to \trLine(t)$, $t\in [0,T]$,
be a normal parametrization of the triple line, i.e., it holds
$\partial_t\psi_{\trLine}(x_0,t) = \vec{V}_{\trLine}(\psi_{\trLine}(x_0,t),t)$ 
for all~$(x_0,t)\in\trLine^0{\times}[0,T]$. Choose moreover a $C^5$ diffeomorphic
parametrization $\varphi_0\colon [0,1] \to \trLine^0$ of the initial triple line,
and define for all~$t\in [0,T]$ the dynamic parametrizations
\begin{align*}
\varphi \colon [0,1]{\times}[0,T] \to \trLine(t),
\quad (s,t) \mapsto \psi_{\trLine}(\varphi_0(s),t).
\end{align*}
Observe then that due to the zeroth order compatibility condition~\eqref{eq:compatibilityTripleLineVelocity} 
and the definition of $\widetilde B$ \eqref{eq:AnsatzAuxiliaryVelocityHalfSpace} it holds
for all~$(s,t)\in [0,1]{\times}[0,T]$
\begin{align}
\label{eq:evolDynamicParametrization}
\partial_t \varphi(s,t) = \widetilde B(\varphi(s,t),t)
= (H\no)(\varphi(s,t),t) + (\alpha_{\mathrm{vel}}\ta_*)(\varphi(s,t),t).
\end{align}
Define finally the differential operator~$\partial_v
:=\frac{\partial_s}{|\partial_s\varphi|}$.
Note that~$\partial_v\varphi(\cdot,t)$ is a unit tangent vector field
along the triple line~$\trLine(t)$ for all~$t\in [0,T]$,
and we may choose the orientation such that $\partial_v\varphi(\cdot,t)
= \taTrJ_*(\varphi(\cdot,t),t)$ for all~$t\in [0,T]$.
A straightforward computation now yields
\begin{align*}
\partial_t\partial_v\varphi = \partial_v\partial_t\varphi
-(\partial_v\partial_t\varphi\cdot\partial_v\varphi)\partial_v\varphi.
\end{align*}
In particular, the commutator $[\partial_t\partial_v,\partial_v\partial_t]\varphi$ 
vanishes in $\ta_*$-direction along the triple line. Using the chain rule and 
the first identity in~\eqref{eq:evolDynamicParametrization}, we thus obtain 
for all~$(s,t)\in [0,1]{\times}[0,T]$, by the orthogonality of the 
frame~$(\no,\ta_*,\taTrJ_*)$, the second identity in~\eqref{eq:evolDynamicParametrization},
as well as~\eqref{eq:gradientNormal2} and~\eqref{eq:gradientTau2}
\begin{align*}
&\big(\ta_*\cdot \big(\partial_t\taTrJ_* 
+ (\widetilde B\cdot\nabla)\taTrJ_*\big)\big)(\varphi(s,t),t)
\\
&= \ta_*(\varphi(s,t),t)\cdot\partial_t\partial_v\varphi(s,t)
\\
&= \ta_*(\varphi(s,t),t)\cdot\partial_v\partial_t\varphi(s,t)
\\&
= \big(\ta_* \cdot (\taTrJ_*\cdot \nabla) (H\no + \alpha_{\mathrm{vel}} \ta_*)\big)(\varphi(s,t),t)
\\&
= -(H\kappa_{\ta_*\taTrJ_*})(\varphi(s,t),t)
+ \big((\taTrJ_*\cdot\nabla)\alpha_{\mathrm{vel}}\big)(\varphi(s,t),t).
\end{align*}
Hence, we may obtain~\eqref{eq:auxclaim} by~\eqref{eq:auxRightHandSide1}, which concludes the proof.
\end{proof}

\subsection{Global construction by interpolation}
\label{subsection:step_2}
Throughout this whole subsection, let the assumptions of Proposition~\ref{prop:gradientFlowCalibrationTripleLine}
and the notation of Section~\ref{sec:localCalibrationInterface}, Subsection~\ref{sec:notationTripleLine} 
and Subsection~\ref{subsection:step_1} be in place. The next results provide the last missing
ingredient for the construction of a local gradient-flow calibration at the triple line. 
We refer to Definition \ref{def:locRadius} and Figure \ref{fig:slice} to recall the geometric setup.

\begin{lemma}
\label{lemma:existenceInterpolationFunctions}
%
Let~$i,j,k\in\{1,2,3\}$ such that~$\{i,j,k\}=\{1,2,3\}$.
For each interpolation wedge~$W_{\bar\Omega_i}$
there exists a pair of associated interpolation functions
\begin{align*}
\lambda_{\bar\Omega_i}^{\barI_{i,j}},\lambda_{\bar\Omega_i}^{\barI_{k,i}}\colon\bigcup_{t\in [0,T]} 
\big(\overline{W_{\bar\Omega_i}(t)}\setminus\trLine(t)\big)
{\times}\{t\}\to [0,1]
\end{align*}
of class $(C^0_tC^1_x \cap C^1_tC^0_x)\big(\bigcup_{t\in [0,T]} 
\big(\overline{W_{\bar\Omega_i}(t)}\setminus\trLine(t)\big)
{\times}\{t\}\big)$ such that $\lambda_{\bar\Omega_i}^{\barI_{k,i}}
=1{-}\lambda_{\bar\Omega_i}^{\barI_{i,j}}$,
and where~$\lambda_{\bar\Omega_i}^{\barI_{i,j}}$ is subject to the following additional requirements:
\begin{itemize}[leftmargin=0.7cm]
\item[i)] On the boundary of the interpolation wedge $W_{\bar\Omega_i}$,
					the values of~$\lambda_{\bar\Omega_i}^{\barI_{i,j}}$ and its derivatives are given by
					\begin{align*}
					\lambda_{\bar\Omega_i}^{\barI_{i,j}}(\cdot,t) &= 0, &&\text{ on } 
					\big(\partial W_{\bar\Omega_i}(t)\cap\partial W_{\barI_{k,i}}(t)\big)\setminus\trLine(t),
					\\
					\lambda_{\bar\Omega_i}^{\barI_{i,j}}(\cdot,t) &= 1, &&\text{ on }
					\big(\partial W_{\bar\Omega_i}(t)\cap\partial W_{\barI_{i,j}}(t)\big)\setminus\trLine(t),
					\\
					\nabla\lambda_{\bar\Omega_i}^{\barI_{i,j}}(\cdot,t) &= 0, \quad
					\partial_t\lambda_{\bar\Omega_i}^{\barI_{i,j}}(\cdot,t) = 0, &&\text{ on } 
					\big(B_r(\trLine(t))\cap\partial W_{\bar\Omega_i}(t)\big)\setminus\trLine(t),
					\end{align*}
					for all $t\in [0,T]$.
\item[ii)] There exists a constant $C>0$, which depends only on the data of the
					 smoothly evolving regular double bubble~$(\bar\Omega_1,\bar\Omega_2,\bar\Omega_3)$ on~$[0,T]$, 
					 such that the estimate
					 \begin{align}
					 \label{eq:regEstimateInterpolFunction}
					 |\partial_t\lambda_{\bar\Omega_i}^{\barI_{i,j}}| + |\nabla\lambda_{\bar\Omega_i}^{\barI_{i,j}}|
					 &\leq C\dist^{-1}(\cdot,\trLine)
					 \end{align}
					 holds true on~$\bigcup_{t\in [0,T]}\big(\overline{W_{\bar\Omega_i}(t)}\setminus\trLine(t)\big){\times}\{t\}$.
\item[iii)] Denoting again by~$\vec{V}_{\trLine}$
						the normal velocity vector field of the triple line~$\trLine$,
						we have an improved estimated on the advective derivative
						\begin{align}
						\label{eq:advDerivInterpolFunction}
						\big|\partial_t\lambda_{\bar\Omega_i}^{\barI_{i,j}}(\cdot,t) 
						+ \big(\vec{V}_{\trLine}(P_{\trLine}(\cdot,t),t)\cdot\nabla\big)\lambda_{\bar\Omega_i}^{\barI_{i,j}}(\cdot,t)\big|
						\leq C
						\end{align}
						on~$\overline{W_{\bar\Omega_i}(t)}\setminus\trLine(t)$ for all~$t\in [0,T]$. 
						The constant~$C>0$ depends only on the data of the
					  smoothly evolving regular double bubble~$(\bar\Omega_1,\bar\Omega_2,\bar\Omega_3)$ on~$[0,T]$.
\end{itemize}
\end{lemma}

\begin{proof}
Let~$i,j,k\in\{1,2,3\}$ such that~$\{i,j,k\}=\{1,2,3\}$. For the
construction of the interpolation function~$\lambda_{\bar\Omega_i}^{\barI_{i,j}}=:1{-}\lambda_{\bar\Omega_i}^{\barI_{k,i}}$,
we first choose a smooth function $\widetilde\lambda\colon\mathbb{R}\to [0,1]$
such that $\widetilde\lambda\equiv 0$ on $[\frac{2}{3},\infty)$
and $\widetilde\lambda\equiv 1$ on $(-\infty,\frac{1}{3}]$.
Denote next by~$\theta_i\in (0,\pi)$ the constant opening
angle of the interpolation wedge~$W_i$, cf.\ the representation~\eqref{eq:defInterpolWedge}.
We then define $\lambda_i\colon [-1,1]\to [0,1]$
by $\lambda_i(u):=\widetilde\lambda(\frac{1{-}u}{1{-}\cos(\theta_i)})$,
and based on this auxiliary map an interpolation function 
\begin{align*}
\lambda_i^{+}(x,t) := \lambda_i\bigg(X_{\bar\Omega_i}^+\big(P_{\trLine}(x,t),t\big)
\cdot\frac{x{-}P_{\trLine}(x,t)}{|x{-}P_{\trLine}(x,t)|}\bigg),
\,t\in[0,T],\,x\in \overline{W_{\bar\Omega_i}(t)}\setminus\trLine(t).
\end{align*}
The interpolation function~$\lambda_{\bar\Omega_i}^{\barI_{i,j}}$ is then either
defined by~$\lambda_i^{+}$ or by~$1{-}\lambda_i^{+}$,
depending on the right choice of ``orientation'' to satisfy
the first item of~\eqref{lemma:existenceInterpolationFunctions},
which in turn is then an immediate consequence of the definitions.
For the proof of~\eqref{eq:regEstimateInterpolFunction}
and~\eqref{eq:advDerivInterpolFunction}, it anyhow suffices
to work on the level of the interpolation function~$\lambda_i^{+}$.

The qualitative regularity of~$\lambda_i^{+}$ and the
corresponding regularity estimate~\eqref{eq:regEstimateInterpolFunction}
follow directly from the chain rule, the definition of~$\lambda_i^{+}$,
and the regularity requirements of Definition~\ref{def:locRadius}.
For the improved estimate~\eqref{eq:advDerivInterpolFunction} on the
advective derivative, we need an appropriate representation of $\partial_t P_{\trLine}$
in~$\mathcal{N}_{r}(\trLine)$. Abbreviating $g(x,t):=\frac{1}{2}\mathrm{\dist}^2(x,\trLine(t))$
as well as~$g_{\trLine}(x,t):=g(P_{\trLine}(x,t),t)$ for all~$(x,t)\in\mathcal{N}_r(\trLine)$,
we obtain by the chain rule
\begin{align*}
0 &= \frac{\mathrm{d}}{\mathrm{d}t} \Big(\nabla g_{\trLine}(x,t) \Big)
\\&
= (\nabla\partial_t g)(y,t)\big|_{y=P_{\trLine}(x,t)}
+(\nabla^2 g)(y,t)\big|_{y=P_{\trLine}(x,t)}\partial_tP_{\trLine}(x,t),
\quad (x,t)\in\mathcal{N}_{r}(\trLine).
\end{align*}
However, it is a well-known fact that $-\nabla\partial_t g$ evaluated along~$\trLine$ 
precisely represents the normal velocity of~$\trLine$
(cf.\ \cite[Theorem~7~\textit{ii)}, p.\ 18]{Ambrosio2000}). Hence, the previous display
updates to
\begin{align*}
\vec{V}_{\trLine}\big(P_{\trLine}(x,t),t\big) = 
\nabla^2 g(y,t)\big|_{y=P_{\trLine}(x,t)}
\partial_tP_{\trLine}(x,t)
\end{align*}
for all~$(x,t)\in\mathcal{N}_{r}(\trLine)$.
Moreover, $\nabla^2g(\cdot,t)$ evaluated along the triple line~$\trLine(t)$ represents for all~$t\in [0,T]$ 
the projection onto the normal bundle~$\mathrm{Tan}^\perp\trLine(t)$ for all $t\in[0,T]$
(cf.\ \cite[Theorem~2~\textit{ii)}, p.\ 12]{Ambrosio2000}).
In other words,
\begin{align}
\label{eq:timeEvolutionProjContactLine}
\vec{V}_{\trLine}\big(P_{\trLine}(x,t),t\big)
= (\mathrm{Id}{-}\taTrJ\otimes\taTrJ)(y,t)
\big|_{y=P_{\trLine}(x,t)}
\partial_tP_{\trLine}(x,t)
\end{align}
for all $(x,t)\in\mathcal{N}_{r}(\trLine)$.

Abbreviating $u^+_i:=u^+_i(x,t):=X_{\bar\Omega_i}^+\big(P_{\trLine}(x,t),t\big)
\cdot\frac{x{-}P_{\trLine}(x,t)}{|x{-}P_{\trLine}(x,t)|}$
we may now compute by an application of the chain rule
\begin{align*}
&\partial_t\lambda^+_i(x,t) 
\\
&= \lambda'_i(u^+_i)X_{\bar\Omega_i}^+\big(P_{\trLine}(x,t),t\big)
\cdot\partial_t\frac{x{-}P_{\trLine}(x,t)}{|x{-}P_{\trLine}(x,t)|}
\\&~~~
+ \lambda'_i(u^+_i)\frac{x{-}P_{\trLine}(x,t)}{|x{-}P_{\trLine}(x,t)|}
\cdot \big((\taTrJ\cdot\nabla)X_{\bar\Omega_i}^+\big)(y,t)\big|_{y=P_{\trLine}(x,t)}
\big(\taTrJ(y,t)|_{y=P_{\trLine}(x,t)}\cdot\partial_tP_{\trLine}(x,t)\big)
\\&~~~
+ \lambda'_i(u^+_i)\frac{x{-}P_{\trLine}(x,t)}{|x{-}P_{\trLine}(x,t)|}
\cdot\big(\partial_tX_{\bar\Omega_i}^+\big)(y,t)\big|_{y=P_{\trLine}(x,t)}
\end{align*}
for all~$(x,t)\in \bigcup_{t\in [0,T]} \big(\overline{W_{\bar\Omega_i}(t)}\setminus\trLine(t)\big){\times}\{t\}$.
Observe that the last two right hand side terms in the previous display are bounded by the regularity 
of the projection~$P_{\trLine}$ and the regularity of the vector field~$X^+_{\bar\Omega_i}$, 
cf.\ Definition~\ref{def:locRadius}. Next, for all~$(x,t)\in\mathcal{N}_r(\trLine)\setminus \trLine$
\begin{align*}
\partial_t\frac{x{-}P_{\trLine}(x,t)}{|x{-}P_{\trLine}(x,t)|}
= -\frac{1}{|x{-}y|}\Big(\mathrm{Id}{-}\frac{x{-}y}{|x{-}y|}
\otimes\frac{x{-}y}{|x{-}y|}\Big)\bigg|_{y=P_{\trLine}(x,t)}
\partial_t P_{\trLine}(x,t),
\end{align*}
so that together with~\eqref{eq:timeEvolutionProjContactLine},
$X_{\bar\Omega_i}^+(y,t),\vec{V}_{\trLine}(y,t)\in\mathrm{Tan}_y^\perp\trLine(t)$ 
for all $(y,t)\in\trLine$, as well as $\nabla P_{\trLine}(x,t)
=\big(\taTrJ(y,t)|_{y=P_{\trLine}(x,t)}\cdot\nabla\big) P_{\trLine}(x,t)
\otimes \taTrJ(y,t)|_{y=P_{\trLine}(x,t)}$ for all~$(x,t)\in\mathcal{N}_r(\trLine)$
\begin{align*}
&\partial_t\lambda^+_i(x,t)
\\
&= -\lambda'_i(u^+_i)\frac{1}{|x{-}y|}\Big(\mathrm{Id}{-}\frac{x{-}y}{|x{-}y|}
\otimes\frac{x{-}y}{|x{-}y|}\Big)X_{\bar\Omega_i}^+(y,t)
\bigg|_{y=P_{\trLine}(x,t)}
\cdot \partial_tP_{\trLine}(x,t) + O(1)
\\&
= -\lambda'_i(u^+_i)\frac{1}{|x{-}y|}\Big(\mathrm{Id}{-}\frac{x{-}y}{|x{-}y|}
\otimes\frac{x{-}y}{|x{-}y|}\Big)X_{\bar\Omega_i}^+(y,t)
\cdot \vec{V}_{\trLine}(y,t)\bigg|_{y=P_{\trLine}(x,t)} + O(1)
\\&
=-\big(\vec{V}_{\trLine}\big(P_{\trLine}(x,t),t\big)\cdot\nabla\big)\lambda_i^+(x,t) + O(1)
\end{align*}
for all~$(x,t)\in \bigcup_{t\in [0,T]} \big(\overline{W_{\bar\Omega_i}(t)}\setminus\trLine(t)\big){\times}\{t\}$
as asserted.
\end{proof}

We may now provide the desired extensions~$(\xi_{i,j})_{i,j\in\{1,2,3\},i\neq j}$ 
for the unit normal vector fields as well as the desired extension~$B$ of the velocity vector field
within a space-time tubular neighborhood~$\mathcal{N}_{\hat r}(\trLine)$ of the evolving triple line~$\trLine$,
where the radius~$\hat r>0$ has to be chosen suitably and is potentially
smaller than the admissible localization radius~$r$.  

\begin{construction}[Gradient-flow calibration at triple line]
\label{def:gradientFlowCalibrationTripleLine}
Let~$(\widetilde\xi,\widetilde\xi',\widetilde\xi'')$
be the preliminary extensions from Construction~\ref{AnsatzAuxiliaryXiHalfSpace} 
of the normal vector fields $(\no|_{\barI}, \no'|_{\barI'}, \no''|_{\barI''})$. 
Let~$(\widetilde R'_{\barI},\widetilde R''_{\barI})$, $(\widetilde R_{\barI'},\widetilde R''_{\barI'})$
and~$(\widetilde R_{\barI''}, \widetilde R'_{\barI''})$ be the gauged Herring rotations
as provided by Construction~\ref{RotationsTripleLine}, and
let~$(\widetilde B,\widetilde B',\widetilde B'')$
be the preliminary extensions of the normal velocity vector fields 
from Construction~\ref{AnsatzAuxiliaryVelocityHalfSpace}. We also introduce
the abbreviations~$\bar\Omega:=\bar\Omega_1$, $\bar\Omega':=\bar\Omega_2$ and~$\bar\Omega'':=\bar\Omega_3$.

With these ingredients in place, we first define a scale~$\hat r:=r\wedge (2C)^{-\frac14}$, 
where $C>0$ denotes the (maximum of the) constant(s) from the estimate(s)~\eqref{eq:lengthControlXiHalfspace}.
This choice of~$\hat r\in (0,r]$ then entails due to~\eqref{eq:lengthControlXiHalfspace} that
\begin{align}
\label{eq:lengthNonDegenerate}
|\widetilde\xi|^2  &\in \big[{\textstyle\frac{1}{2},\frac{3}{2}}\big] 
&&\text{in } \mathcal{N}_{\hat r}(\trLine)\cap\mathrm{im}(\Psi),
\\
\label{eq:lengthNonDegenerate2}
|\widetilde\xi'|^2  &\in \big[{\textstyle\frac{1}{2},\frac{3}{2}}\big] 
&&\text{in } \mathcal{N}_{\hat r}(\trLine)\cap\mathrm{im}(\Psi'),
\\
\label{eq:lengthNonDegenerate3}
|\widetilde\xi''|^2  &\in \big[{\textstyle\frac{1}{2},\frac{3}{2}}\big]
&&\text{in } \mathcal{N}_{\hat r}(\trLine)\cap\mathrm{im}(\Psi'').
\end{align}
Based on these non-degeneracy conditions and the 
properties~\eqref{eq:decompTripleLine}--\eqref{eq:interpolWedgeHalfspace}
from the wedge decomposition of~$\mathcal{N}_r(\trLine)$,
we construct a well-defined set of vector fields
\begin{align}
\label{eq:calibrationTripleLineXi}
\xi,\xi',\xi''\colon&\mathcal{N}_{\hat r}(\trLine) \to \overline{B_1(0)},
\\
B\colon&\mathcal{N}_{\hat r}(\trLine) \to \Rd[3]
\end{align}
by the following procedure: 
On the closure of the interface wedges we define
\begin{align}
\label{eq:defXiInterfaceWedgeFinal}
(\xi,\xi',\xi'') &:= \big|\widetilde\xi\,\big|^{-1}\big(\widetilde\xi,\,
\widetilde R'_{\barI}\,\widetilde\xi,\,\widetilde R''_{\barI}\,\widetilde\xi\,\big)
&& \text{on } \overline{W_{\barI}},
\\
(\xi,\xi',\xi'') &:= \big|\widetilde\xi'\big|^{-1}\big(\widetilde R_{\barI'}\,\widetilde\xi',\,
\widetilde\xi',\,\widetilde R''_{\barI'}\,\widetilde\xi'\big)
&& \text{on } \overline{W_{\barI'}},
\\
\label{eq:defXiInterfaceWedgeFinal3}
(\xi,\xi',\xi'') &:= \big|\widetilde\xi''\big|^{-1}\big(\widetilde R_{\barI''}\,\widetilde\xi'',\,
\widetilde R'_{\barI''}\,\widetilde\xi'',\,\widetilde\xi''\big)
&& \text{on } \overline{W_{\barI''}},
\end{align}
as well as
\begin{align}
\label{eq:defVelInterfaceWedgeFinal}
B &:= \widetilde B \text{ on }  \overline{W_{\barI}},
\quad B := \widetilde B' \text{ on }  \overline{W_{\barI'}},
\quad B := \widetilde B'' \text{ on }  \overline{W_{\barI''}}.
\end{align}
On the interpolation wedges, say~$W_{\bar\Omega}$, we define
\begin{align}
\label{eq:defXiInterpolWedgeFinal}
\xi &:= \lambda^{\barI}_{\bar\Omega} \big|\widetilde\xi\,\big|^{-1} \widetilde\xi
								+ \lambda^{\barI''}_{\bar\Omega} \big|\widetilde\xi''\big|^{-1} 
								  \widetilde R_{\barI''}\,\widetilde\xi'',
\\
\xi' &:= \lambda^{\barI}_{\bar\Omega} \big|\widetilde\xi\,\big|^{-1} \widetilde R'_{\barI}\,\widetilde\xi
									+ \lambda^{\barI''}_{\bar\Omega} \big|\widetilde\xi''\big|^{-1} \widetilde R'_{\barI''}\,\widetilde\xi'', 
\\
\label{eq:defXiInterpolWedgeFinal3}
\xi'' &:= \lambda^{\barI}_{\bar\Omega} \big|\widetilde\xi\,\big|^{-1} \widetilde R''_{\barI}\,\widetilde\xi
									 + \lambda^{\barI''}_{\bar\Omega} \big|\widetilde\xi''\big|^{-1} \widetilde\xi'',
\\
\label{eq:defVelInterpolWedgeFinal}
B &:= \lambda^{\barI}_{\bar\Omega}\widetilde B + \lambda^{\barI''}_{\bar\Omega}\widetilde B''.
\end{align}
On the remaining two interpolation wedges~$W_{\bar\Omega'}$ and~$W_{\bar\Omega''}$, 
one proceeds analogously for the definition of these vector fields.
\hfill$\diamondsuit$
\end{construction}

\subsection{Proof of Proposition~\ref{prop:gradientFlowCalibrationTripleLine}}
\label{sec:proofGradientFlowCalibrationTripleLine}
Let $(\xi,\xi',\xi'',B)$ be the vector fields 
from Construction~\ref{def:gradientFlowCalibrationTripleLine}.
We aim to show that this tuple of vector fields gives rise to a local gradient
flow calibration at the triple line~$\trLine$ in the sense of 
Proposition~\ref{prop:gradientFlowCalibrationTripleLine} after defining
\begin{align}
\label{eq:calibrationTripleLineXiIndices}
\xi_{1,2} := \xi,\quad
\xi_{2,3} := \xi',\quad
\xi_{3,1} := \xi''
\quad\text{in } \mathcal{N}_{\hat r}(\trLine),
\end{align}
as well as~$\xi_{j,i}:=-\xi_{i,j}$ for the remaining set of distinct phases~$i,j\in\{1,2,3\}$.
The proof is now split into several steps.

In \textit{Step~1}
of the proof, we will derive the following useful
compatibility estimates valid throughout interpolation
wedges and which are needed in all subsequent steps: 
\begin{align}
\label{eq:compEstimateInterpolWedgeXi}
\bigg|\frac{\widetilde\xi}{|\widetilde\xi|}{-}
\frac{\widetilde R_{\barI''}\widetilde\xi''}{|\widetilde\xi''|}\bigg|
+
\bigg|\frac{\widetilde R'_{\barI}\,\widetilde\xi}{|\widetilde\xi|}{-}
\frac{\widetilde R'_{\barI''}\widetilde\xi''}{|\widetilde\xi''|}\bigg|
+
\bigg|\frac{\widetilde R''_{\barI}\,\widetilde\xi}{|\widetilde\xi|}{-}
\frac{\widetilde\xi''}{|\widetilde\xi''|}\bigg| 
\leq C\dist^2(\cdot,\trLine)
\end{align}
in $W_{\bar\Omega}\cap\mathcal{N}_{\hat r}(\trLine)$,
with analogous estimates being satisfied in the other two interpolation wedges.
Moreover, the constant~$C>0$ only depends on the data of the smoothly
evolving regular double bubble $(\bar\Omega_1,\bar\Omega_2,\bar\Omega_3)$ on~$[0,T]$.

In \textit{Step~2}, we will verify that $(\xi,\xi',\xi'',B)$
are continuous vector fields throughout~$\mathcal{N}_{\hat r}(\trLine)$, that
the extensions of the unit normals $(\xi,\xi',\xi'')$
are of class $(C^0_tC^1_x \cap C^1_tC^0_x)(\mathcal{N}_{\hat r}(\trLine)\setminus\trLine)$ whereas
the extended velocity~$B$ is of class $C^0_tC^1_x(\mathcal{N}_{\hat r}(\trLine)\setminus\trLine)$,
and that there exists a constant~$C>0$ depending only on the data of the smoothly
evolving regular double bubble $(\bar\Omega_1,\bar\Omega_2,\bar\Omega_3)$ on~$[0,T]$
such that the estimate
\begin{align}
\label{eq:regEstimateAlmostFinalCalibration}
|(\partial_t,\nabla)(\xi,\xi',\xi'')|
+ |B| + |\nabla B| \leq C
\end{align}
holds true throughout~$\mathcal{N}_{\hat r}(\trLine)\setminus\trLine$. Moreover, we will
show that
\begin{align}
\label{eq:extensionPropertyAux1}
\xi &= \no|_{\barI} &&\text{along } \barI\cap\mathcal{N}_{\hat r}(\trLine),
\\
\label{eq:extensionPropertyAux4}
B &= \vec{V}_{\trLine} &&\text{along } \trLine,
\\
\label{eq:extensionPropertyAux5}
\sigma\xi + \sigma'\xi' + \sigma''\xi'' &= 0
&&\text{in } \mathcal{N}_{\hat r}(\trLine),
\end{align}
where property~\eqref{eq:extensionPropertyAux1} is also satisfied 
in terms of~$(\xi',\no'|_{\barI'})$ along~$\barI'\cap\mathcal{N}_{\hat r}(\trLine)$,
or in terms of~$(\xi'',\no''|_{\barI''})$ along~$\barI''\cap\mathcal{N}_{\hat r}(\trLine)$.

\textit{Step~3} of the proof is then devoted to the verification of the 
approximate evolution equation
\begin{align}
\label{eq:evolEquAlmostFinalCalibration}
|\partial_t\xi + (B\cdot\nabla)\xi + (\nabla B)^\mathsf{T}\xi|
\leq C\dist(\cdot,\barI) \quad\text{in } \mathcal{N}_{\hat r}(\trLine)\setminus\trLine,
\end{align}
whereas in \textit{Step~4} we will prove the estimate
\begin{align}
\label{eq:evolMCFAlmostFinalCalibration}
|\nabla\cdot\xi + B\cdot\xi|
\leq C\dist(\cdot,\barI) \quad\text{in } \mathcal{N}_{\hat r}(\trLine)\setminus\trLine.
\end{align}
We finally conclude in \textit{Step~5} by deducing the estimate
\begin{align}
\label{eq:evolLengthXiCalibrationTripleLine}
(\partial_t + B\cdot\nabla)|\xi|^2 &\leq C\dist^2(\cdot,\barI)
\quad\text{in } \mathcal{N}_{\hat r}(\trLine).
\end{align}
We record for completeness that analogous estimates with respect 
to~\eqref{eq:evolEquAlmostFinalCalibration}--\eqref{eq:evolLengthXiCalibrationTripleLine}
are satisfied for~$(\xi',B)$, resp.\ $(\xi'',B)$,
in terms of~$\dist(\cdot,\barI')$, resp.\ $\dist(\cdot,\barI'')$,
and that the constant~$C>0$ again only depends on the data of the smoothly
evolving regular double bubble $(\bar\Omega_1,\bar\Omega_2,\bar\Omega_3)$ on~$[0,T]$.

\textit{Step 1: Proof of~\eqref{eq:compEstimateInterpolWedgeXi}.}
Adding zero, making use of the reverse triangle inequality
and recalling the non-degeneracy condition~\eqref{eq:lengthNonDegenerate}--\eqref{eq:lengthNonDegenerate3}, 
we may estimate
\begin{align*}
\bigg|\frac{\widetilde\xi}{|\widetilde\xi|}-
\frac{\widetilde R_{\barI''}\widetilde\xi''}{|\widetilde\xi''|}\bigg|
&\leq \frac{1}{|\widetilde\xi|}\big|\widetilde\xi-\widetilde R_{\barI''}\widetilde\xi''\big|
+ \bigg|\frac{1}{|\widetilde\xi|}-\frac{1}{|\widetilde R_{\barI''}\widetilde\xi''|}\bigg|
\big|\widetilde R_{\barI''}\widetilde\xi''\big|
\\&
\leq \frac{1}{|\widetilde\xi|}\big|\widetilde\xi-\widetilde R_{\barI''}\widetilde\xi''\big|
+ \frac{1}{|\widetilde\xi|}\Big|\big|\widetilde\xi\big|-\big|\widetilde R_{\barI''}\widetilde\xi''\big|\Big|
\leq 2\sqrt{2}\big|\widetilde\xi-\widetilde R_{\barI''}\widetilde\xi''\big|.
\end{align*}
Due to the compatibility conditions~\eqref{eq:compatibilityTripleLineXi}
and~\eqref{eq:compatibilityTripleLineXi2} as well as the regularity
estimates~\eqref{eq:regEstimateAuxiliarxExtensionsNormal}, the
previous estimate then easily upgrades to
\begin{align*}
\bigg|\frac{\widetilde\xi}{|\widetilde\xi|}{-}
\frac{\widetilde R_{\barI''}\widetilde\xi''}{|\widetilde\xi''|}\bigg|
\leq C\dist^2(\cdot,\trLine)
\quad\text{in } W_{\bar\Omega}\cap\mathcal{N}_{\hat r}(\trLine)
\end{align*}
by inserting a second-order Taylor expansion with base point located at the unique
nearest point on the triple line~$\trLine$. The other two terms
on the left hand side of~\eqref{eq:compEstimateInterpolWedgeXi}
are treated analogously.

\textit{Step 2: Proof of 
\emph{\eqref{eq:regEstimateAlmostFinalCalibration}--\eqref{eq:extensionPropertyAux5}}.}
In terms of the asserted qualitative regularity, we observe that
the first item of Lemma~\ref{lemma:existenceInterpolationFunctions}
together with the definitions from Construction~\ref{def:gradientFlowCalibrationTripleLine}
ensure that the vector fields $(\xi,\xi',\xi'',B)$
and their required derivatives are continuous across the boundaries of the
interpolation wedges (away from the triple line). Continuity of~$B$ throughout
the whole space-time neighborhood~$\mathcal{N}_r(\trLine)$ with 
the asserted representation~\eqref{eq:extensionPropertyAux4}
along the triple line~$\trLine$ follows from the
compatibility condition~\eqref{eq:compatibilityTripleLineVelocity}.
The unit normal extensions $(\xi,\xi',\xi'')$
are continuous throughout~$\mathcal{N}_r(\trLine)$ due to
the compatibility estimates~\eqref{eq:compEstimateInterpolWedgeXi}.
The representation~\eqref{eq:extensionPropertyAux1} along the
associated interface in turn is a consequence of
the expansion ansatz~\eqref{eq:AnsatzAuxiliaryXiHalfSpace} and the inclusion~\eqref{eq:interfaceWedge}.

Next, on interface wedges the regularity estimate~\eqref{eq:regEstimateAlmostFinalCalibration}
follows directly from the estimates~\eqref{eq:regEstimateAuxiliarxExtensionsNormal} 
and~\eqref{eq:regEstimateAuxiliaryExtensionVelTripleLine}. For the derivation
of~\eqref{eq:regEstimateAlmostFinalCalibration} throughout an interpolation
wedge, say~$W_{\bar\Omega}\cap\mathcal{N}_{\hat r}(\trLine)$, we simply compute by plugging in the definitions
from Construction~\ref{def:gradientFlowCalibrationTripleLine}
and recalling from Lemma~\ref{lemma:existenceInterpolationFunctions}
that $\lambda_{\bar\Omega}^{\barI''}=1{-}\lambda_{\bar\Omega}^{\barI}$
\begin{align*}
(\partial_t,\nabla)\xi &= 
\lambda_{\bar\Omega}^{\barI} (\partial_t,\nabla)\frac{\widetilde\xi}{|\widetilde\xi|}
+ \lambda_{\bar\Omega}^{\barI''} (\partial_t,\nabla)
\frac{\widetilde R_{\barI''}\widetilde\xi''}{|\widetilde\xi''|}
+ \bigg(\frac{\widetilde\xi}{|\widetilde\xi|}{-}
\frac{\widetilde R_{\barI''}\widetilde\xi''}{|\widetilde\xi''|}\bigg)\otimes
(\partial_t,\nabla)\lambda_{\bar\Omega}^{\barI},
\\
\nabla B &= \lambda_{\bar\Omega}^{\barI}\nabla \widetilde B
+ \lambda_{\bar\Omega}^{\barI''}\nabla\widetilde B''
+ (\widetilde B{-}\widetilde B'')\otimes\nabla\lambda_{\bar\Omega}^{\barI}.
\end{align*}
We thus infer~\eqref{eq:regEstimateAlmostFinalCalibration}
from the chain rule in form of $\nabla\frac{1}{|f|}=-\frac{(\nabla f)^\mathsf{T}f}{|f|^3}$, 
the regularity estimates~\eqref{eq:regEstimateAuxiliarxExtensionsNormal},
\eqref{eq:regEstimateAuxiliaryExtensionVelTripleLine} and~\eqref{eq:regEstimateInterpolFunction},
and the compatibility conditions~\eqref{eq:compEstimateInterpolWedgeXi}
and~\eqref{eq:compatibilityTripleLineVelocity}.

We turn to the proof of~\eqref{eq:extensionPropertyAux5}. Recalling
the expansion ansatz~\eqref{eq:AnsatzAuxiliaryXiHalfSpace} and
the definitions~\eqref{eq:defGaugedHerringRotation} resp.\ \eqref{eq:defGaugedHerringRotation2}
of the gauged Herring rotations, we deduce from~\eqref{eq:HerringConditionByRotation} 
\begin{align}
\label{eq:extensionPropertyAux6}
\sigma \widetilde\xi + \sigma'\widetilde R'_{\barI} \widetilde\xi
+ \sigma''\widetilde R''_{\barI} \widetilde\xi = 0
\quad\text{throughout } \mathcal{N}_r(\trLine)\cap\mathrm{im}(\Psi),
\end{align}
and analogously throughout $\mathcal{N}_r(\trLine)\cap\mathrm{im}(\Psi')$
in terms of $(\widetilde R_{\barI'}\widetilde\xi',\widetilde\xi',\widetilde R''_{\barI'}\widetilde\xi')$,
or throughout $\mathcal{N}_r(\trLine)\cap\mathrm{im}(\Psi'')$
in terms of $(\widetilde R_{\barI''}\widetilde\xi'',\widetilde R'_{\barI''}\widetilde\xi'',\widetilde \xi'')$.
Due to the inclusion~\eqref{eq:interfaceWedge}
and the definitions from Construction~\ref{def:gradientFlowCalibrationTripleLine},
we thus obtain from~\eqref{eq:extensionPropertyAux6}
\begin{align*}
\sigma\xi + \sigma'\xi' + \sigma''\xi''
&= |\widetilde\xi|^{-1} \big(\sigma \widetilde\xi + \sigma'\widetilde R'_{\barI} \widetilde\xi
+ \sigma''\widetilde R''_{\barI} \widetilde\xi\,\big) = 0
\quad\text{in } W_{\barI}\cap\mathcal{N}_{\hat r}(\trLine).
\end{align*}
An analogous argument works in the case of the other two interface wedges.

On interpolation wedges, say~$W_{\bar\Omega}$, the extended Herring angle
condition~\eqref{eq:extensionPropertyAux5} follows from a linear combination
of the previous ingredients. More precisely, the definitions from 
Construction~\ref{def:gradientFlowCalibrationTripleLine} and the
cancellations~\eqref{eq:extensionPropertyAux6} directly imply
\begin{align*}
&\sigma\xi + \sigma'\xi' + \sigma''\xi''
\\
&= \lambda_{\bar\Omega}^{\barI}|\widetilde\xi|^{-1} \big(\sigma \widetilde\xi 
+ \sigma'\widetilde R'_{\barI} \widetilde\xi
+ \sigma''\widetilde R''_{\barI} \widetilde\xi\,\big)
+ \lambda_{\bar\Omega}^{\barI''}|\widetilde\xi''|^{-1} 
  \big(\sigma \widetilde R_{\barI''}\widetilde\xi'' 
+ \sigma'\widetilde R'_{\barI''} \widetilde\xi''
+ \sigma''\widetilde\xi''\,\big)
= 0  
\end{align*}
throughout~$W_{\bar\Omega}\cap\mathcal{N}_{\hat r}(\trLine)$ as desired.
This concludes the proof of~\eqref{eq:extensionPropertyAux5}, and thus \textit{Step~2} of the proof,
as on the other interpolation wedges \eqref{eq:extensionPropertyAux5}~follows
analogously. 

\textit{Step 3: Proof of~\eqref{eq:evolEquAlmostFinalCalibration}.}
We first claim that for each rotation $\mathcal{R}\in\{\mathrm{Id},\widetilde R'_{\barI},\widetilde R''_{\barI}\}$
it holds throughout~$\mathcal{N}_{\hat r}(\trLine)\cap\mathrm{im}(\Psi)$
\begin{align}
\label{eq:evolEquXiHalfspaceNormalized}
\Big|\big(\partial_t {+} (\widetilde B\cdot\nabla) {+} (\nabla\widetilde B)^\mathsf{T}\big)
\frac{\mathcal{R}\widetilde\xi}{|\mathcal{R}\widetilde\xi|}\Big|
&\leq C\begin{cases}
			 \dist(\cdot,\bar I) &\text{if } \mathcal{R}=\mathrm{Id}, 
			 \\
			 \dist(\cdot,\trLine) &\text{else},
			 \end{cases}
\\
\label{eq:divConstraintXiHalfspaceNormalized}
\Big|\widetilde B\cdot \frac{\mathcal{R}\widetilde\xi}{|\mathcal{R}\widetilde\xi|}
+ \nabla\cdot\frac{\mathcal{R}\widetilde\xi}{|\mathcal{R}\widetilde\xi|}\Big|
&\leq C\begin{cases}
			 \dist(\cdot,\bar I) &\text{if } \mathcal{R}=\mathrm{Id}, 
			 \\
			 \dist(\cdot,\trLine) &\text{else},
			 \end{cases}
\end{align}
for some constant~$C>0$ which depends only on the smoothly evolving regular 
double bubble~$(\bar\Omega_1,\bar\Omega_2,\bar\Omega_3)$ on~$[0,T]$.
Moreover, analogous estimates hold true throughout the domain~$\mathcal{N}_{\hat r}(\trLine)\cap\mathrm{im}(\Psi')$
in terms of the vector fields~$(\mathcal{R}\widetilde\xi',\widetilde B')$
for each rotation $\mathcal{R}\in\{\widetilde R_{\barI'},\mathrm{Id},\widetilde R''_{\barI'}\}$,
as well as throughout~$\mathcal{N}_{\hat r}(\trLine)\cap\mathrm{im}(\Psi'')$
in terms of~$(\mathcal{R}\widetilde\xi'',\widetilde B'')$
for each $\mathcal{R}\in\{\widetilde R_{\barI''},\widetilde R'_{\barI''},\mathrm{Id}\}$.

The estimate~\eqref{eq:evolEquXiHalfspaceNormalized} follows from the straightforward computation
\begin{align*}
&\big(\partial_t {+} (\widetilde B\cdot\nabla) {+} (\nabla\widetilde B)^\mathsf{T}\big)
\frac{\mathcal{R}\widetilde\xi}{|\mathcal{R}\widetilde\xi|}
\\&
= \frac{1}{|\mathcal{R}\widetilde\xi|}
\big(\partial_t {+} (\widetilde B\cdot\nabla) {+} (\nabla\widetilde B)^\mathsf{T}\big)\mathcal{R}\widetilde\xi
- \frac{\partial_t |\mathcal{R}\widetilde\xi|^2
{+}(\widetilde B\cdot\nabla)|\mathcal{R}\widetilde\xi|^2}
{2|\mathcal{R}\widetilde\xi|^3}\mathcal{R}\widetilde\xi
\end{align*}
together with the condition~\eqref{eq:lengthNonDegenerate}
and the estimates~\eqref{eq:gradientLengthXi}, \eqref{eq:timeDerivLengthXi}
and~\eqref{eq:evolEquXiHalfspace}. The estimate~\eqref{eq:divConstraintXiHalfspaceNormalized}
in turn can be deduced from the same ingredients as well as
\begin{align*}
&\widetilde B\cdot \frac{\mathcal{R}\widetilde\xi}{|\mathcal{R}\widetilde\xi|}
+ \nabla\cdot\frac{\mathcal{R}\widetilde\xi}{|\mathcal{R}\widetilde\xi|}
= \frac{1}{|\mathcal{R}\widetilde\xi|}
\big(B\cdot\mathcal{R}\widetilde\xi + \nabla\cdot \mathcal{R}\widetilde\xi\big)
- \frac{(\mathcal{R}\widetilde\xi\cdot\nabla)|\mathcal{R}\widetilde\xi|^2}{2|\mathcal{R}\widetilde\xi|^3}.
\end{align*}

On interface wedges, facilitated by the inclusion~\eqref{eq:interfaceWedge}
the claim~\eqref{eq:evolEquAlmostFinalCalibration} 
now follows from an application of the estimate~\eqref{eq:evolEquXiHalfspaceNormalized}
and, if needed, a simple post-processing by means of~\eqref{eq:compDistances2}.
Hence, let us directly move on with the verification of~\eqref{eq:evolEquAlmostFinalCalibration}
throughout interpolation wedges, say~$W_{\bar\Omega}\cap\mathcal{N}_{\hat r}(\trLine)$.
Plugging in the definitions~\eqref{eq:defXiInterpolWedgeFinal}--\eqref{eq:defVelInterpolWedgeFinal} 
from Construction~\ref{def:gradientFlowCalibrationTripleLine}
we may compute based on the product rule, adding zero,
and recalling from Lemma~\ref{lemma:existenceInterpolationFunctions}
that $\lambda_{\bar\Omega}^{\barI''}=1{-}\lambda_{\bar\Omega}^{\barI}$
\begin{align}
\nonumber
\big(\partial_t{+}(B\cdot\nabla){+}(\nabla B)^\mathsf{T}\big)\xi
&= \lambda_{\bar\Omega}^{\barI}\big(\partial_t{+}(\widetilde B\cdot\nabla)
{+}(\nabla \widetilde B)^\mathsf{T}\big)\frac{\widetilde\xi}{|\widetilde\xi|} 
\\&~~~ \label{eq:auxRightHandSide2}
+ \big(1{-}\lambda_{\bar\Omega}^{\barI}\big)\big(\partial_t{+}(\widetilde B''\cdot\nabla)
{+}(\nabla \widetilde B'')^\mathsf{T}\big)
\frac{\widetilde R_{\barI''}\widetilde\xi''}{|\widetilde\xi''|} 
\\&~~~ \nonumber
+ \bigg(\frac{\widetilde\xi}{|\widetilde\xi|} 
{-} \frac{\widetilde R_{\barI''}\widetilde\xi''}{|\widetilde\xi''|} \bigg) 
\big(\partial_t{+}(B\cdot\nabla)\big)\lambda_{\bar\Omega}^{\barI}
\\&~~~ \nonumber
+ \lambda_{\bar\Omega}^{\barI}\big((B{-}\widetilde B)\cdot\nabla\big)
\frac{\widetilde\xi}{|\widetilde\xi|}  
+ \big(1{-}\lambda_{\bar\Omega}^{\barI}\big)\big((B{-}\widetilde B'')\cdot\nabla\big)
\frac{\widetilde R_{\barI''}\widetilde\xi''}{|\widetilde\xi''|} 
\\&~~~ \nonumber
+ \lambda_{\bar\Omega}^{\barI}\big(\nabla{B}{-}\nabla\widetilde B\big)^\mathsf{T}
\frac{\widetilde\xi}{|\widetilde\xi|} 
+ \big(1{-}\lambda_{\bar\Omega}^{\barI}\big)\big(\nabla{B}{-}\nabla\widetilde B''\big)^\mathsf{T}
\frac{\widetilde R_{\barI''}\widetilde\xi''}{|\widetilde\xi''|} .
\end{align}

The first two right hand side terms of the previous display
are at least of the order~$O(\dist(\cdot,\trLine))$ due to 
the estimates~\eqref{eq:evolEquXiHalfspaceNormalized}, which in turn are 
available this time due to the inclusion~\eqref{eq:interpolWedgeHalfspace}.
The third, fourth and fifth right hand side terms are of the same order
thanks to the compatibility conditions~\eqref{eq:compEstimateInterpolWedgeXi}
and~\eqref{eq:compatibilityTripleLineVelocity}, the regularity
estimates~\eqref{eq:regEstimateAuxiliarxExtensionsNormal},
\eqref{eq:regEstimateAuxiliaryExtensionVelTripleLine} 
and~\eqref{eq:regEstimateAlmostFinalCalibration},
the estimate~\eqref{eq:advDerivInterpolFunction}
on the advective derivative of an interpolation function,
as well as the non-degeneracy conditions~\eqref{eq:lengthNonDegenerate}--\eqref{eq:lengthNonDegenerate3}.

Regarding the two right hand side terms from the last line of the previous display,
we may argue as follows. Plugging in the definition of~$B$ from Construction~\ref{def:gradientFlowCalibrationTripleLine},
we compute by the product rule, the identity~$\lambda_{\bar\Omega}^{\barI}+\lambda_{\bar\Omega}^{\barI''}=1$
and by carefully noting that $\widetilde\xi\perp\taTrJ_*$ throughout~$\mathcal{N}_r(\trLine)\cap\mathrm{im}(\Psi)$
due to the expansion ansatz~\eqref{eq:AnsatzAuxiliaryXiHalfSpace}
\begin{align*}
(\nabla B {-} \nabla \widetilde B)^\mathsf{T}\widetilde\xi 
&= (1{-}\lambda_{\bar\Omega}^{\barI}) (\nabla \widetilde B'' {-} \nabla \widetilde B)^\mathsf{T}
(\mathrm{Id}{-}\taTrJ_*\otimes\taTrJ_*)\widetilde\xi
\\&~~~
+ \big((\widetilde B{-}\widetilde B'')
\cdot(\mathrm{Id}{-}\taTrJ_*\otimes\taTrJ_*)\widetilde\xi\,\big)
\nabla\lambda_{\bar\Omega}^{\barI}.
\end{align*}
Abbreviating~$\taTrJ_{\trLine}(x,t):=\taTrJ(P_{\trLine}(x,t),t)$
for all~$(x,t)\in\mathcal{N}_r(\trLine)\cap\mathrm{im}(\Psi)$
and recalling the compatibility conditions~\eqref{eq:defGaugedOrthonormalFrame}
resp.\ \eqref{eq:compatibilityTripleLineVelocity} 
as well as the regularity estimate~\eqref{eq:regEstimateInterpolFunction}
for the interpolation function,
we may switch from~$\taTrJ_*$ to~$\taTrJ_{\trLine}$ in the previous
display at the cost of an admissible error:
\begin{align*}
(\nabla B {-} \nabla \widetilde B)^\mathsf{T}\widetilde\xi 
&= (1{-}\lambda_{\bar\Omega}^{\barI}) (\nabla \widetilde B'' {-} \nabla \widetilde B)^\mathsf{T}
(\mathrm{Id}{-}\taTrJ_{\trLine}\otimes\taTrJ_{\trLine})\widetilde\xi
\\&~~~
+ \big((\widetilde B{-}\widetilde B'')
\cdot(\mathrm{Id}{-}\taTrJ_{\trLine}\otimes\taTrJ_{\trLine})\widetilde\xi\,\big)
\nabla\lambda_{\bar\Omega}^{\barI}
+ O(\dist(\cdot,\trLine)).
\end{align*}
It then follows from the compatibility conditions~\eqref{eq:orientationTangentTripleLine},
\eqref{eq:compatibilityTripleLineVelocity} and~\eqref{eq:compatibilityTripleLineVelocityGradient}, 
and again the regularity estimate~\eqref{eq:regEstimateInterpolFunction}
for the interpolation function that
\begin{align*}
(\nabla B {-} \nabla \widetilde B)^\mathsf{T}\widetilde\xi = O(\dist(\cdot,\trLine)).
\end{align*}
One may argue similarly for the second term
after replacing~$|\widetilde\xi''|^{-1}\widetilde R_{\barI''}\widetilde\xi''$ 
by~$|\widetilde\xi|^{-1}\widetilde\xi$
using the compatibility estimate~\eqref{eq:compEstimateInterpolWedgeXi}.

In summary, the asserted estimate~\eqref{eq:evolEquAlmostFinalCalibration}
in terms of~$\xi$ now follows from the previously derived estimates
for the right hand side terms of~\eqref{eq:auxRightHandSide2} and
a subsequent post-processing of them by means of~\eqref{eq:compDistances}.
We finally remark that the argument proceeds analogously for the other two vector fields~$\xi'$
and~$\xi''$, respectively.

\textit{Step 4: Proof of~\eqref{eq:evolMCFAlmostFinalCalibration}.}
Thanks to the inclusion~\eqref{eq:interfaceWedge},
the estimate~\eqref{eq:divConstraintXiHalfspaceNormalized},
and, if needed, the estimate~\eqref{eq:compDistances2},
it again suffices to provide additional details only for the
argument for~\eqref{eq:evolMCFAlmostFinalCalibration} 
on interpolation wedges, say~$W_{\bar\Omega}\cap\mathcal{N}_{\hat r}(\trLine)$.
Plugging in the definitions~\eqref{eq:defXiInterpolWedgeFinal}--\eqref{eq:defVelInterpolWedgeFinal} 
from Construction~\ref{def:gradientFlowCalibrationTripleLine},
applying the product rule, recalling from Lemma~\ref{lemma:existenceInterpolationFunctions}
that $\lambda_{\bar\Omega}^{\barI''}=1{-}\lambda_{\bar\Omega}^{\barI}$, and adding zero yields
\begin{align*}
B\cdot\xi + \nabla\cdot\xi
&= \lambda_{\bar\Omega}^{\barI}
\bigg(\widetilde B\cdot\frac{\widetilde\xi}{|\widetilde\xi|} 
+ \nabla\cdot\frac{\widetilde\xi}{|\widetilde\xi|} \,\bigg)
+ \big(1{-}\lambda_{\bar\Omega}^{\barI}\big)
\bigg(\widetilde B''\cdot\frac{\widetilde R_{\barI''}\widetilde\xi''}{|\widetilde\xi''|}
+ \nabla\cdot\frac{\widetilde R_{\barI''}\widetilde\xi''}{|\widetilde\xi''|}\bigg)
\\&~~~
+ \lambda_{\bar\Omega}^{\barI}(B{-}\widetilde B)\cdot
\frac{\widetilde\xi}{|\widetilde\xi|} 
+ \big(1{-}\lambda_{\bar\Omega}^{\barI}\big)(B{-}\widetilde B'')\cdot 
\frac{\widetilde R_{\barI''}\widetilde\xi''}{|\widetilde\xi''|}
\\&~~~
+ \bigg(\frac{\widetilde\xi}{|\widetilde\xi|} 
{-} \frac{\widetilde R_{\barI''}\widetilde\xi''}{|\widetilde\xi''|}\bigg)
\cdot\nabla\lambda^{\barI}_{\bar\Omega}.
\end{align*}
The right hand side terms of the previous display are all at least of order~$O(\dist(\cdot,\trLine))$
---and thus of required order due to~\eqref{eq:compDistances}---by an application
of the inclusion~\eqref{eq:interpolWedgeHalfspace}, the estimates~\eqref{eq:divConstraintXiHalfspaceNormalized}, 
the compatibility conditions~\eqref{eq:compatibilityTripleLineVelocity}
and~\eqref{eq:compEstimateInterpolWedgeXi}, as well as the regularity 
estimate~\eqref{eq:regEstimateInterpolFunction} for the interpolation function.

This proves~\eqref{eq:evolMCFAlmostFinalCalibration} in terms of~$\xi$.
The argument proceeds again analogously for the other two vector fields~$\xi'$
and~$\xi''$.

\textit{Step 5: Proof of~\eqref{eq:evolLengthXiCalibrationTripleLine}.}
There is nothing to prove throughout interface wedges
since the unit normal extensions~$(\xi,\xi',\xi'')$
are unit length vectors, cf.\ the definitions from Construction~\ref{def:gradientFlowCalibrationTripleLine}.
On interpolation wedges, say~$W_{\bar\Omega}\cap\mathcal{N}_{\hat r}(\trLine)$, we may compute
by the definition \eqref{eq:defXiInterpolWedgeFinal} from Construction~\ref{def:gradientFlowCalibrationTripleLine}
\begin{align}
\label{eq:repLengthInterpolatedXiTripleLine}
|\xi|^2 = 1 - \lambda_{\bar\Omega}^{\barI}\lambda_{\bar\Omega}^{\barI''}
\Bigg|\frac{\widetilde\xi}{|\widetilde\xi|} 
{-} \frac{\widetilde R_{\barI''}\widetilde\xi''}{|\widetilde\xi''|}\Bigg|^2.
\end{align}
The estimate~\eqref{eq:evolLengthXiCalibrationTripleLine} ist
thus a consequence of~\eqref{eq:compEstimateInterpolWedgeXi},
\eqref{eq:regEstimateInterpolFunction}, \eqref{eq:regEstimateAlmostFinalCalibration}
and~\eqref{eq:compDistances}.
One may argue analogously for the other two vector fields~$\xi'$ and~$\xi''$, respectively.
\qed

\subsection{Compatibility of local gradient-flow calibrations}
\label{subsec:localCompatibility}
A regular double bubble is built out of two
distinct topological features: the three two-phase interfaces
and the triple line. For each of these topological features, we so far constructed
a tuple of vector fields living in a space-time neighborhood of the feature
and locally mimicking the requirements of a gradient-flow calibration. The remaining step
in the construction consists of pasting 
together these local vector fields into 
globally defined ones. This task will be carried out in Section~\ref{sec:globalCalibrations}.
The key issue is to transfer properties from the local to the global level, 
which turns out to be possible because, among other things, the local
constructions for the two distinct topological features 
can be arranged to be sufficiently compatible. We formalize this as follows.

\begin{proposition}
\label{prop:compatibilityEstimatesLocalConstructions}
Let $(\bar\Omega_1,\bar\Omega_2,\bar\Omega_3)$ be a regular double bubble smoothly evolving by MCF in the
sense of Definition~\ref{def:smoothSolution} on a time interval~$[0,T]$. Let~$\hat r\in (0,1]$
be the localization scale of Proposition~\ref{prop:gradientFlowCalibrationTripleLine},
and for each pair of distinct phases~$i,j\in\{1,2,3\}$, denote by~$(\xi^{\barI_{i,j}}_{i,j},B^{\barI_{i,j}})$
the local gradient-flow calibration for the interface~$\barI_{i,j}$ from
Construction~\ref{gradientFlowCalibrationInterface}. 

For all~$i,j\in\{1,2,3\}$ with~$i\neq j$, there exists a choice of the 
tangential component~$\mathcal{Y}_{i,j}$ of $B^{\barI_{i,j}}$
and a local gradient-flow calibration~$((\xi_{i,j}^{\trLine})_{i,j\in\{1,2,3\},i\neq j},B^{\trLine})$
at the triple line in the sense of Proposition~\ref{prop:gradientFlowCalibrationTripleLine}
such that in addition the following compatibility estimates hold true
\begin{align}
\label{eq:compEstimate1}
\big|\xi^{\barI_{i,j}}_{i,j} - \xi^{\trLine}_{i,j}\big| 
+ \big|(\nabla \xi_{i,j}^{\barI_{i,j}})^\mathsf{T}\xi^{\trLine}_{i,j}\big|
+ \big|(\nabla\xi^{\trLine}_{i,j})^\mathsf{T}\xi_{i,j}^{\barI_{i,j}}\big|
&\leq C\dist(\cdot,\barI_{i,j}),
\\
\label{eq:compEstimate2}
\big|(\xi^{\barI_{i,j}}_{i,j} {-} \xi^{\trLine}_{i,j})\cdot\xi^{\barI_{i,j}}_{i,j}\big|
&\leq C\dist^2(\cdot,\barI_{i,j}),
\\
\label{eq:compEstimate3}
\big|B^{\barI_{i,j}} - B^{\trLine}\big|
&\leq C\dist(\cdot,\barI_{i,j}),
\\
\label{eq:compEstimate4}
\big|(\nabla B^{\barI_{i,j}} {-} \nabla B^{\trLine})^\mathsf{T}\xi_{i,j}^{\barI_{i,j}}\big|
&\leq C\dist(\cdot,\barI_{i,j})
\end{align}
throughout~$\mathcal{N}_{\frac{\hat r}{2}}(\trLine)\cap (W_{\barI_{i,j}}\cup W_{\bar\Omega_i}
\cup W_{\bar\Omega_j})$, where $C>0$ is a constant which depends only on the smoothly evolving regular 
double bubble~$(\bar\Omega_1,\bar\Omega_2,\bar\Omega_3)$ on~$[0,T]$.
\end{proposition}

\begin{proof}
Let~$((\xi_{i,j}^{\trLine})_{i,j\in\{1,2,3\},i\neq j},B^{\trLine})$
be the local gradient-flow calibration at the triple line~$\trLine$
as constructed in the proof of Proposition~\ref{prop:gradientFlowCalibrationTripleLine},
and let~$i,j\in\{1,2,3\}$ be distinct phases.

\textit{Step 1: Proof of~\eqref{eq:compEstimate1}.}
The estimate $|\xi^{\barI_{i,j}}_{i,j} - \xi^{\trLine}_{i,j}| \leq C\dist(\cdot,\barI_{i,j})$
is an immediate consequence of the regularity estimates~\eqref{eq:regXiInterface}
and~\eqref{eq:regXiTripleLine}, the inclusions~\eqref{eq:interfaceWedge}--\eqref{eq:interpolWedgeHalfspace}, 
as well as the extension property~$\xi^{\barI_{i,j}}_{i,j}=\no_{i,j}
=\xi^{\trLine}_{i,j}$ along~$\barI_{i,j}\cap\mathcal{N}_{\hat r}(\trLine)$.

The estimate $|(\nabla \xi_{i,j}^{\barI_{i,j}})^\mathsf{T}\xi^{\trLine}_{i,j}|
\leq C\dist(\cdot,\barI_{i,j})$ follows from adding zero, the already established estimate for the 
first left hand side term of~\eqref{eq:compEstimate1}, and~$\xi_{i,j}^{\barI_{i,j}}$
being a unit length vector field due to~\eqref{eq:localGradFlowCalibrationInterface}.

For the remaining part of~\eqref{eq:compEstimate1}, it suffices
to estimate~$\frac{1}{2}\nabla|\xi^{\trLine}_{i,j}|^2$ due to~\eqref{eq:regXiTripleLine}
and the already established estimate for the first left hand side term of~\eqref{eq:compEstimate1}.
Throughout the interpolation wedge $W_{\barI_{i,j}}\cap\mathcal{N}_{\hat r}(\trLine)$,
we have~$|\xi^{\trLine}_{i,j}|\equiv 1$ in view of the 
definitions~\eqref{eq:defXiInterfaceWedgeFinal}--\eqref{eq:defXiInterfaceWedgeFinal3},
so that the desired estimate is satisfied for trivial reasons. Within
the relevant interpolation wedges, one may employ the representation~\eqref{eq:repLengthInterpolatedXiTripleLine}
and then deduce~$\big|\frac{1}{2}\nabla|\xi^{\trLine}_{i,j}|^2\big|\leq C\dist(\cdot,\barI_{i,j})$
from~\eqref{eq:compEstimateInterpolWedgeXi}, \eqref{eq:regEstimateInterpolFunction}
and~\eqref{eq:compDistances}.

\textit{Step 2: Proof of~\eqref{eq:compEstimate2}.}
Denote by~$\widetilde\xi^{\,\barI_{i,j}}$ the auxiliary extension
of the unit normal~$\no_{i,j}|_{\barI_{i,j}}$ from Construction~\ref{AnsatzAuxiliaryXiHalfSpace}.
Due to~\eqref{eq:defXiInterfaceWedgeFinal}--\eqref{eq:defXiInterfaceWedgeFinal3}, \eqref{eq:defXiInterpolWedgeFinal}--\eqref{eq:defXiInterpolWedgeFinal3}, and the compatibility estimates~\eqref{eq:compEstimateInterpolWedgeXi} it holds
\begin{align*}
\xi^{\trLine}_{i,j} = |\widetilde\xi^{\,\barI_{i,j}}|^{-1} 
\widetilde\xi^{\,\barI_{i,j}} + O(\dist^2(\cdot,\barI_{i,j}))
\quad\text{in } \mathcal{N}_{\hat r}(\trLine)
\cap (W_{\barI_{i,j}}\cup W_{\bar\Omega_i}\cup W_{\bar\Omega_j}).
\end{align*}
Making use of the non-degeneracy conditions~\eqref{eq:lengthNonDegenerate}--\eqref{eq:lengthNonDegenerate3}
and the estimate~\eqref{eq:lengthControlXiHalfspace} we also obtain
\begin{align*}
|\widetilde\xi^{\,\barI_{i,j}}|^{-1} - 1
= \frac{1-|\widetilde\xi^{\,\barI_{i,j}}|^2}
{|\widetilde\xi^{\,\barI_{i,j}}|\big(1{+}|\widetilde\xi^{\,\barI_{i,j}}|\big)}
= O(\dist^2(\cdot,\barI_{i,j}))
\quad\text{in } \mathcal{N}_{\hat r}(\trLine) \cap \mathrm{im}(\Psi_{i,j}).
\end{align*}
Recalling the precise representations~\eqref{eq:localGradFlowCalibrationInterface}
and~\eqref{eq:AnsatzAuxiliaryXiHalfSpace}, we thus infer from the 
previous two displays that
\begin{align*}
\big|(\xi^{\barI_{i,j}}_{i,j} {-} \xi^{\trLine}_{i,j})\cdot\xi^{\barI_{i,j}}_{i,j}\big|
\leq \big| 1 - |\widetilde\xi^{\,\barI_{i,j}}|^{-1} \big| + O(\dist^2(\cdot,\barI_{i,j}))
=  O(\dist^2(\cdot,\barI_{i,j}))
\end{align*}
throughout $\mathcal{N}_{\hat r}(\trLine)
\cap (W_{\barI_{i,j}}\cup W_{\bar\Omega_i}\cup W_{\bar\Omega_j})$
as asserted.

\textit{Step 3: Construction of the tangential component~$\mathcal{Y}_{i,j}$ of~$B^{\barI_{i,j}}$.}
Let~$\theta$ be a smooth and even cutoff function with $\theta(r)=1$ for $|r|\leq\frac{1}{2}$ and $\theta(r)= 0$
for $|r|\geq 1$. Denote by~$\widetilde B^{\barI_{i,j}}$ the auxiliary
local velocity field from Construction~\ref{AnsatzAuxiliaryVelocityHalfSpace}
with respect to the interface~$\barI_{i,j}$.
The tangential component~$\mathcal{Y}_{i,j}$ of~$B^{\barI_{i,j}}$
is then simply defined by means of
\begin{align}
\label{eq:choiceTangentialComponentVelInterface}
\mathcal{Y}_{i,j} := \theta\Big(\frac{\dist(\cdot,\trLine)}{\hat r}\Big)
(\mathrm{Id}{-}\no_{i,j}\otimes\no_{i,j}) \widetilde B^{\barI_{i,j}} \quad\text{in } \mathrm{im}(\Psi_{i,j}).
\end{align}
Note that~$\mathcal{Y}_{i,j}\in C^0_tC^1_x(\mathrm{im}(\Psi_{i,j}))$
as required by Construction~\ref{gradientFlowCalibrationInterface}
due to the regularity~\eqref{eq:regularityNormalCurvature} of the normal~$\no_{i,j}$ and
the regularity estimate~\eqref{eq:regEstimateAuxiliaryExtensionVelTripleLine}
for~$\widetilde B^{\barI_{i,j}}$.

\textit{Step 4: Proof of~\emph{\eqref{eq:compEstimate3}--\eqref{eq:compEstimate4}}.}
It follows from the expansion ansatz~\eqref{eq:AnsatzAuxiliaryVelocityHalfSpace},
the definitions~\eqref{eq:defVelInterfaceWedgeFinal} and~\eqref{eq:localGradFlowCalibrationInterface},
the choice of the tangential component~\eqref{eq:choiceTangentialComponentVelInterface},
as well as the choice of the cutoff~$\theta$ from the previous step that
\begin{align}
\label{eq:compInterfaceTripleLineVel}
B^{\barI_{i,j}} = B^{\trLine} \quad
\text{throughout } W_{\barI_{i,j}} \cap \mathcal{N}_{\frac{\hat r}{2}}(\trLine).
\end{align}
More precisely, denoting by~$\widetilde B^{\barI_{i,j}}$ the auxiliary
local velocity field from Construction~\ref{AnsatzAuxiliaryVelocityHalfSpace}
with respect to the interface~$\barI_{i,j}$, we in fact have
\begin{align}
\label{eq:compInterfaceTripleLineVel2}
B^{\barI_{i,j}} = \widetilde B^{\barI_{i,j}} \quad
\text{throughout } ( W_{\barI_{i,j}} \cup W_{\bar\Omega_i} \cup W_{\bar\Omega_j} )
\cap \mathcal{N}_{\frac{\hat r}{2}}(\trLine).
\end{align}
Now \eqref{eq:compEstimate3} follows directly from a Taylor-expansion
argument exploiting the regularity estimates~\eqref{eq:regVelocityInterface}
resp.\ \eqref{eq:regVelocityTripleLine} as well as the 
inclusions~\eqref{eq:interfaceWedge}--\eqref{eq:interpolWedgeHalfspace}.

In view of~\eqref{eq:compInterfaceTripleLineVel}, the estimate~\eqref{eq:compEstimate4}
is satisfied for trivial reasons throughout the interface 
wedge~$\mathcal{N}_{\frac{\hat r}{2}}(\trLine) \cap W_{\barI_{i,j}}$.
Within the relevant interpolation wedges, say for 
concreteness~$\mathcal{N}_{\frac{\hat r}{2}}(\trLine) \cap W_{\bar\Omega_i}$,
we make use of~\eqref{eq:compInterfaceTripleLineVel2}. Let $k\in\{1,2,3\}\setminus\{i,j\}$
denote the third phase. It then follows from~\eqref{eq:compInterfaceTripleLineVel2}
and expressing the definition~\eqref{eq:defVelInterpolWedgeFinal}
in form of $B^{\trLine} = \lambda_{\bar\Omega_i}^{\barI_{i,j}} \widetilde B^{\barI_{i,j}} 
+ (1{-}\lambda_{\bar\Omega_i}^{\barI_{i,j}}) \widetilde B^{\barI_{k,i}}$
\begin{align*}
\nabla B^{\barI_{i,j}} - \nabla B^{\trLine}
= \big(1{-}\lambda_{\bar\Omega_i}^{\barI_{i,j}}\big)
\big(\nabla\widetilde B^{\barI_{i,j}} {-} \nabla \widetilde B^{\barI_{k,i}}\big)
- \big(\widetilde B^{\barI_{i,j}} {-} \widetilde B^{\barI_{k,i}}\big) \otimes \nabla\lambda_{\bar\Omega_i}^{\barI_{i,j}}.
\end{align*}
Since $\xi^{\barI_{i,j}}_{i,j}=\no_{i,j}$ due to~\eqref{eq:localGradFlowCalibrationInterface},
the estimate~\eqref{eq:compEstimate4} now follows throughout
the interpolation wedge~$\mathcal{N}_{\frac{\hat r}{2}}(\trLine) \cap W_{\bar\Omega_i}$
by the same argument which deals with estimating the last two right hand side terms of~\eqref{eq:auxRightHandSide2}.
We recall for convenience that the essential input for the latter is given by the
compatibility conditions~\eqref{eq:compatibilityTripleLineVelocity} 
and~\eqref{eq:compatibilityTripleLineVelocityGradient} for the auxiliary velocity
fields~$\widetilde B^{\barI_{i,j}}$ and~$\widetilde B^{\barI_{k,i}}$. 
\end{proof}

\section{Gradient-flow calibrations for double bubbles}
\label{sec:globalCalibrations}

\subsection{Localization of topological features}
We start by introducing a family of suitable cutoff functions localizing
around the interfaces and the triple line in a smoothly evolving regular double bubble. 
This family will be used to provide the construction of a gradient-flow calibration
by means of gluing together the local constructions from the previous two sections. 

\begin{lemma}
\label{lemma:partitionOfUnity}
Let $(\bar\Omega_1,\bar\Omega_2,\bar\Omega_3)$ be a regular 
double bubble smoothly evolving by~MCF in the sense of Definition~\ref{def:smoothSolution} 
on a time interval~$[0,T]$. Let the notation of Definition~\ref{def:locRadiusInterface}
resp.\ Definition~\ref{def:locRadius} be in place, and let $\hat r\in (0,1]$ be the radius of
Proposition~\ref{prop:gradientFlowCalibrationTripleLine}. In particular, let~$(r_{i,j})_{i,j\in\{1,2,3\},i\neq j}$
be admissible localization radii for the interfaces in the sense of Definition~\ref{def:locRadiusInterface}
such that $\hat r \leq r_{1,2}\wedge r_{2,3} \wedge r_{3,1}$. We next define 
for each pair~$i,j\in\{1,2,3\}$ with~$i\neq j$ a scale
\begin{align*}
3\ell_{i,j} := \min_{t\in [0,T]}\min_{\substack{k,l\in\{1,2,3\},\,k\neq l,
\\ (k,l)\notin\{(i,j),(j,i)\}}}
\dist\big(\barI_{i,j}(t)\setminus B_{\hat r}(\trLine(t)),\barI_{k,l}(t)\big)> 0,
\end{align*}
and based on these a localization 
scale~$\bar r\in (0,r_{1,2}\wedge r_{2,3}\wedge r_{3,1}]$ by means of
\begin{align}
\label{eq:defLocScale}
2\bar r := \hat r \wedge \min_{i,j\in\{1,2,3\},\,i\neq j} \ell_{i,j}. 
\end{align}

There then exists a collection 
of continuous cutoff functions
\begin{align*}
\eta_{\trLine},\,\eta_{\barI_{1,2}},\,\eta_{\barI_{2,3}},\,\eta_{\barI_{3,1}} 
\colon \mathbb{R}^3 \times [0,T] \to [0,1]
\end{align*}
satisfying the following properties: 
\begin{itemize}[leftmargin=0.7cm]
\item[i)] The cutoff functions are of class $(C^0_tC^1_x\cap C^1_tC^0_x)
					(\mathbb{R}^3{\times}[0,T]\setminus\trLine)$ with corresponding
					regularity estimates
					\begin{align}
					\label{eq:regEstimateCutoffs}
					|(\partial_t,\nabla)(\eta_{\trLine},\eta_{\barI_{1,2}},\eta_{\barI_{2,3}},\eta_{\barI_{3,1}})|
					\leq C \quad\text{in } \Rd[3]{\times}[0,T]\setminus\trLine
					\end{align}
					for some constant~$C>0$ depending only on the data of the smoothly evolving regular double 
					bubble~$(\bar\Omega_1,\bar\Omega_2,\bar\Omega_3)$ on~$[0,T]$.
\item[ii)] The family $(\eta_{\trLine},(\eta_{\barI_{i,j}})_{i,j\in\{1,2,3\},i\neq j})$ 
				  is a partition of unity for the evolving
					surface cluster in the sense that
					$\eta_{\trLine}+\eta_{\barI_{1,2}}
					+\eta_{\barI_{2,3}}+\eta_{\barI_{3,1}}\equiv 1$ 
					holds true on the surface cluster $\mathcal{I}:=\bigcup_{i,j\in\{1,2,3\},i\neq j}\barI_{i,j}$.
					
					Moreover, for all pairwise distinct~$i,j,k\in\{1,2,3\}$ it holds
					\begin{align}
					\label{eq:coercivityInterfaceCutoff}
					\eta_{\barI_{k,i}} &\leq C(\dist^2(\cdot,\barI_{i,j}) \wedge 1) 
					&&\text{in } \Rd[3]{\times} [0,T],
					\\ \label{eq:coercivityGradientInterfaceCutoff}
					|\nabla\eta_{\barI_{k,i}}| &\leq C(\dist(\cdot,\barI_{i,j}) \wedge 1) 
					&&\text{in } \Rd[3]{\times} [0,T]\setminus\trLine,
					\\ \label{eq:coercivityTimeDerivativeInterfaceCutoff}
					|\partial_t\eta_{\barI_{k,i}}| &\leq C(\dist(\cdot,\barI_{i,j}) \wedge 1) 
					&&\text{in } \Rd[3]{\times} [0,T]\setminus\trLine.
					\end{align}
					Defining
					$\eta_{\mathrm{bulk}}:=1-\eta_{\trLine}-\eta_{\barI_{1,2}}-\eta_{\barI_{2,3}}
					-\eta_{\barI_{3,1}}$ we have $\eta_{\mathrm{bulk}}\in [0,1]$ on $\Rd[3]\times[0,T]$,
					and the bulk cutoff is subject to the estimates
					\begin{align}
					\label{eq:coercivityBulkCutoff}
					\frac{1}{C}(\dist^2(\cdot,\mathcal{I}) \wedge 1) &\leq
					\eta_{\mathrm{bulk}} \leq C(\dist^2(\cdot,\mathcal{I}) \wedge 1) 
					&&\text{in } \Rd[3]{\times} [0,T],
					\\ \label{eq:coercivityGradientBulkCutoff}
					|\nabla\eta_{\mathrm{bulk}}| &\leq C(\dist(\cdot,\mathcal{I}) \wedge 1) 
					&&\text{in } \Rd[3]{\times} [0,T]\setminus\trLine,
					\\ \label{eq:coercivityTimeDerivBulkCutoff}
					|\partial_t\eta_{\mathrm{bulk}}| &\leq C(\dist(\cdot,\mathcal{I}) \wedge 1) 
					&&\text{in } \Rd[3]{\times} [0,T]\setminus\trLine.
					\end{align}
					The constant~$C\geq 1$ in the 
					estimates~\emph{\eqref{eq:coercivityInterfaceCutoff}--\eqref{eq:coercivityTimeDerivBulkCutoff}}
					depends only on the data of the smoothly evolving regular double 
					bubble~$(\bar\Omega_1,\bar\Omega_2,\bar\Omega_3)$ on~$[0,T]$.
					
\item[iii)] For all pairwise distinct~$i,j,k\in\{1,2,3\}$ and all~$t\in[0,T]$ it holds
					 \begin{align}
					 \label{eq:supportCutOffInterfaces}
					 \supp\eta_{\barI_{i,j}}(\cdot,t) &\subset 
					 \Psi_{i,j}(\barI_{i,j}(t){\times}\{t\}{\times}[-\bar r,\bar r]),
					 \\
					 \label{eq:supportCutOffInterfacesNearTripleLine}
					 B_{\hat r}(\trLine(t))\cap\supp\eta_{\barI_{i,j}}(\cdot,t)
					 &\subset B_{\hat r}(\trLine(t))\cap\big(W_{\barI_{i,j}}(t)
					 \cup W_{\bar\Omega_i}(t)\cup W_{\bar\Omega_j}(t)\big),
					 \\
					 \label{eq:supportCutOffTwoInterfaces}
					 \supp\eta_{\barI_{i,j}}(\cdot,t)\cap\supp\eta_{\barI_{j,k}}(\cdot,t)
					 &\subset B_{\hat r}(\trLine(t))\cap W_{\bar\Omega_j}(t),
					 \\
					 \label{eq:supportCutOffTripleLine}
					 \supp\eta_{\trLine}(\cdot,t) &\subset B_{\hat r/2}(\trLine(t)).
					 \end{align}
\end{itemize}
\end{lemma}

\begin{proof}
The proof is split into several steps.

\textit{Step 1: Definition of building blocks.}
Let $\theta$ be a smooth and even cutoff function with $\theta(r)=1$ for $|r|\leq\frac{1}{2}$ and $\theta(r) = 0$
for $|r|\geq 1$. Then define a smooth quadratic profile~$\zeta\colon\mathbb{R}\to [0,1]$ by means of
\begin{align}
\label{QuadraticInterfaceCutOff}
\zeta(r) = (1-r^2)\theta(r^2), \quad r\in\Rd[].
\end{align}
Let~$\delta\in (0,1]$ be a constant
whose value will be determined in subsequent steps of the proof. For all
distinct~$i,j\in\{1,2,3\}$ we define auxiliary cutoff functions
\begin{align}
\label{eq:interfaceCutoff}
\zeta_{\barI_{i,j}} &:= \zeta\Big(
\frac{s_{i,j}}{\delta\bar r}\Big)
&&\text{in } \mathrm{im}(\Psi_{i,j}),
\\
\label{eq:tripleLineCutoff}
\zeta_{\trLine} &:= \zeta\Big(
\frac{\dist(\cdot,\trLine)}{\hat r/2}\Big)
&&\text{in } \Rd[3] {\times} [0,T].
\end{align}
Note that as a consequence of the regularity~\eqref{eq:regProjectionSignedDistance}
of the signed distance, expressing $\dist(x,\trLine(t))=|x{-}P_{\trLine}(x,t)|$
for all~$x\in B_{\hat r}(\trLine(t))$ and all~$t\in [0,T]$,
the regularity of the projection~$P_{\trLine}$ onto
the triple line~$\trLine$ from Definition~\ref{def:locRadius},
and~\eqref{QuadraticInterfaceCutOff} it holds
\begin{align}
\label{eq:regEstimateInterfaceCutoff}
|(\partial_t,\nabla)\zeta_{\barI_{i,j}}| & \leq C\dist(\cdot,\barI_{i,j})
&&\text{in } \mathrm{im}(\Psi_{i,j}),
\\
\label{eq:regEstimateTripleLineCutoff}
|(\partial_t,\nabla)\zeta_{\trLine}| & \leq C\dist(\cdot,\trLine)
&&\text{in } \Rd[3] {\times} [0,T].
\end{align}

\textit{Step 2: Definition of interface cutoffs.} Fix distinct~$i,j\in\{1,2,3\}$.
We define the cutoff~$\eta_{\barI_{i,j}}\colon\Rd[3]{\times}[0,T]\to [0,1]$
for the two-phase interface~$\barI_{i,j}$ by means of
\begin{align}
\label{eq:defInterfaceCutoffAwayTripleLine}
\eta_{\barI_{i,j}}(\cdot,t) &:= \zeta_{\barI_{i,j}}(\cdot,t)
&&\text{in } \mathrm{im}(\Psi_{i,j}(t)) \setminus B_{\hat r}(\trLine(t)),
\\
\label{eq:defInterfaceCutoffInterfaceWedge}
\eta_{\barI_{i,j}}(\cdot,t) &:= (1{-}\zeta_{\trLine}(\cdot,t))\zeta_{\barI_{i,j}}(\cdot,t)
&&\text{in } B_{\hat r}(\trLine(t)) \cap W_{\barI_{i,j}}(t),
\\
\label{eq:defInterfaceCutoffInterpolWedge1}
\eta_{\barI_{i,j}}(\cdot,t) &:= \lambda_{\bar\Omega_i}^{\barI_{i,j}}(\cdot,t)(1{-}\zeta_{\trLine}(\cdot,t))
\zeta_{\barI_{i,j}}(\cdot,t)
&&\text{in } B_{\hat r}(\trLine(t)) \cap W_{\bar\Omega_i}(t),
\\
\label{eq:defInterfaceCutoffInterpolWedge2}
\eta_{\barI_{i,j}}(\cdot,t) &:= \lambda_{\bar\Omega_j}^{\barI_{i,j}}(\cdot,t)(1{-}\zeta_{\trLine}(\cdot,t))
\zeta_{\barI_{i,j}}(\cdot,t)
&&\text{in } B_{\hat r}(\trLine(t)) \cap W_{\bar\Omega_j}(t),
\\
\label{eq:defInterfaceCutoffElse}
\eta_{\barI_{i,j}}(\cdot,t) &:= 0
&&\text{else}
\end{align}
for all~$t\in [0,T]$. Here, the maps~$\lambda^{\barI_{i,j}}_{\bar\Omega_i}$ resp.\ $\lambda^{\barI_{i,j}}_{\bar\Omega_j}$
are the interpolation functions of Lemma~\ref{lemma:existenceInterpolationFunctions} on
the interpolation wedges~$W_{\bar\Omega_i}$ resp.\ $W_{\bar\Omega_j}$.
Observe that~\eqref{eq:defInterfaceCutoffInterfaceWedge} is well-defined
because of~\eqref{eq:interfaceWedge}, and that~\eqref{eq:defInterfaceCutoffInterpolWedge1}
resp.\ \eqref{eq:defInterfaceCutoffInterpolWedge2} are well-defined
as a consequence of~\eqref{eq:interpolWedgeHalfspace}. In particular, the 
properties~\eqref{eq:supportCutOffInterfaces}--\eqref{eq:supportCutOffTwoInterfaces} 
are immediate consequences of the definitions~\eqref{eq:defInterfaceCutoffAwayTripleLine}--\eqref{eq:defInterfaceCutoffElse}
and the choice~\eqref{eq:defLocScale} of the localization scale~$\bar r$.
Finally, in order to ensure continuity of~$\eta_{\barI_{i,j}}$ throughout~$\Rd[3]{\times}[0,T]$
(i.e., compatibility of the definition~\eqref{eq:defInterfaceCutoffAwayTripleLine} 
resp.\ the definition \eqref{eq:defInterfaceCutoffElse} 
with the definitions~\eqref{eq:defInterfaceCutoffInterfaceWedge}--\eqref{eq:defInterfaceCutoffInterpolWedge2})
we choose the constant~$\delta\in (0,\frac{1}{2}]$ small enough such that for all~$t\in [0,T]$
and all distinct~$i,j\in\{1,2,3\}$ it holds
\begin{align}
\label{eq:choiceDelta1}
\partial B_{\hat r}(\trLine(t)) \cap
\overline{\Psi_{i,j}(\barI_{i,j}(t){\times}\{t\}{\times}[-\delta\hat r,\delta\hat r])}
&\subset\subset W_{\bar I_{i,j}}(t).
\end{align}

\textit{Step 3: Definition of triple line cutoff.}
We construct a cutoff for the triple line $\eta_{\trLine}\colon\Rd[3]{\times}[0,T]\to [0,1]$
as follows: for all distinct~$i,j,k\in\{1,2,3\}$ and all~$t\in[0,T]$ we define
\begin{align}
\label{eq:defTripleLineCutoffInterfaceWedge}
\eta_{\trLine}(\cdot,t) &:= \zeta_{\trLine}(\cdot,t)\zeta_{\barI_{i,j}}(\cdot,t)
&&\text{in } B_{\hat r}(\trLine(t))\cap W_{\barI_{i,j}}(t),
\\
\label{eq:defTripleLineCutoffInterpolWedge}
\eta_{\trLine}(\cdot,t) &:= \lambda_{\bar\Omega_{i}}^{\barI_{i,j}}(\cdot,t)
\zeta_{\trLine}(\cdot,t)\zeta_{\barI_{i,j}}(\cdot,t)
\\\nonumber
&~~~~
+ \lambda_{\bar\Omega_{i}}^{\barI_{k,i}}(\cdot,t)
\zeta_{\trLine}(\cdot,t)\zeta_{\barI_{k,i}}(\cdot,t)
&&\text{in } B_{\hat r}(\trLine(t))\cap W_{\bar\Omega_i}(t),
\\
\label{eq:defTripleLineCutoffAwayTripleLine}
\eta_{\trLine}(\cdot,t) &:= 0
&&\text{in } \Rd[3]\setminus B_{\hat r}(\trLine(t)).
\end{align}
Because of~\eqref{eq:decompTripleLine}, the 
definitions~\eqref{eq:defTripleLineCutoffInterfaceWedge}--\eqref{eq:defTripleLineCutoffAwayTripleLine}
provide a definition of~$\eta_{\trLine}$ on the whole space-time domain~$\Rd[3]{\times}[0,T]$.
Property~\eqref{eq:supportCutOffTripleLine} is obviously satisfied in view of~\eqref{eq:defTripleLineCutoffAwayTripleLine}.
Since~$\lambda_{\bar\Omega_{i}}^{\barI_{i,k}}=1{-}\lambda_{\bar\Omega_{i}}^{\barI_{i,j}}$
on interpolation wedges~$W_{\bar\Omega_i}$, we indeed have~$\eta_{\trLine}(x,t)\in [0,1]$
for all~$(x,t)\in\Rd[3]{\times}[0,T]$.

\textit{Step 4: Partition of unity property along the surface cluster}.
Define the bulk cutoff $\eta_{\mathrm{bulk}}:=1-\eta_{\trLine}-\eta_{\barI_{1,2}}
-\eta_{\barI_{2,3}}-\eta_{\barI_{3,1}}$. We claim that
\begin{align}
\label{eq:partitionOfUnity}
\eta_{\mathrm{bulk}} = 0 \quad\text{along } \mathcal{I} = \bigcup_{i,j\in\{1,2,3\},\,i\neq j}\barI_{i,j}.
\end{align}
Fix~$t\in [0,T]$ and a point~$x\in\mathcal{I}(t)\setminus B_{\hat r}(\trLine(t))$.
There exists a unique pair of distinct
phases~$i,j\in\{1,2,3\}$ such that~$x\in\barI_{i,j}(t)$
and, because of the localization properties~\eqref{eq:supportCutOffTwoInterfaces}
and~\eqref{eq:supportCutOffTripleLine},~$\eta_{\mathrm{bulk}}(x,t)=1-\eta_{\barI_{i,j}}(x,t)$. It then
follows from the definitions~\eqref{eq:defInterfaceCutoffAwayTripleLine} 
and~\eqref{eq:interfaceCutoff} that~$\eta_{\mathrm{bulk}}(x,t)=0$.

Now fix~$t\in [0,T]$ and consider a point~$x\in\mathcal{I}(t)\cap B_{\hat r}(\trLine(t))$. 
Let~$i,j\in\{1,2,3\}$ be the unique pair of distinct
phases such that~$x\in\barI_{i,j}(t)$. As a consequence of~\eqref{eq:interfaceWedge},
the localization properties~\eqref{eq:supportCutOffInterfacesNearTripleLine}--\eqref{eq:supportCutOffTripleLine},
and the definitions~\eqref{eq:defInterfaceCutoffInterfaceWedge} resp.\
\eqref{eq:defTripleLineCutoffInterfaceWedge}, we obtain that
$\eta_{\mathrm{bulk}}(x,t) = 1 - \eta_{\trLine}(x,t) - \eta_{\barI_{i,j}}(x,t)
= 1 - \zeta_{\barI_{i,j}}(x,t)$. Hence, $\eta_{\mathrm{bulk}}(x,t)=0$ due to
the definition~\eqref{eq:interfaceCutoff}. This concludes the proof of~\eqref{eq:partitionOfUnity}.

\textit{Step 5: Regularity of cutoff functions.} 
Fix~$i,j\in\{1,2,3\}$ such that~$i\neq j$.
The required derivatives of~$\eta_{\barI_{i,j}}$
exist in~$\Rd[3]\setminus \overline{B_{\hat r}(\trLine(t))}$
resp.\ in~$B_{\hat r}(\trLine(t))\setminus\trLine(t)$
in a pointwise sense for all~$t\in [0,T]$ due to
the definition of~$\eta_{\barI_{i,j}}$ from \textit{Step~2} of this proof,
the definitions~\eqref{eq:interfaceCutoff} and~\eqref{eq:tripleLineCutoff},
the properties of the interpolation functions from Lemma~\ref{lemma:existenceInterpolationFunctions},
and the regularity~\eqref{eq:regEstimateInterfaceCutoff}
and~\eqref{eq:regEstimateTripleLineCutoff} of the auxiliary cutoff functions.
By the choice~\eqref{eq:choiceDelta1}
of the scale~$\delta\in (0,1]$, these derivatives do not jump across the boundary
of~$B_{\hat r}(\trLine(t))$. Hence, $\partial_t\eta_{\barI_{i,j}}$ and~$\nabla \eta_{\barI_{i,j}}$
exist in a pointwise sense in~$\Rd[3]{\times}[0,T]\setminus\trLine$.

In terms of the required bounds~\eqref{eq:regEstimateCutoffs} 
for these derivatives, the only possibly critical cases are those
for which at least one derivative hits an interpolation function present in
the definitions~\eqref{eq:defInterfaceCutoffInterpolWedge1}
resp.\ \eqref{eq:defInterfaceCutoffInterpolWedge2}. The blow-up of these derivatives (see 
Lemma~\ref{lemma:existenceInterpolationFunctions}), however, is always cured by 
the presence of the term~$1{-}\zeta_{\trLine}$.
In summary, $\eta_{\barI_{i,j}}\in (C^0_tC^1_x\cap C^1_tC^0_x)
(\mathbb{R}^3{\times}[0,T]\setminus\trLine)$
and~\eqref{eq:regEstimateCutoffs} holds true.

Along similar lines, one checks that~$\partial_t\eta_{\trLine}$ and~$\nabla \eta_{\trLine}$
exist in a pointwise sense in~$\Rd[3]{\times}[0,T]\setminus\trLine$.
The required cancellations to counteract the blow-up of derivatives of the interpolation
parameter in interpolation wedges this time comes from 
recalling~$\lambda_{\bar\Omega_{i}}^{\barI_{k,i}}=1{-}\lambda_{\bar\Omega_{i}}^{\barI_{i,j}}$,
which in turn ensures that potentially critical terms always involve the term
$\zeta_{\barI_{i,j}}-\zeta_{\barI_{k,i}}$.
As the latter vanishes to first order at the triple line and has a bounded
second-order spatial derivative within interpolation wedges, it follows that
$\eta_{\trLine}\in (C^0_tC^1_x\cap C^1_tC^0_x)(\mathbb{R}^3{\times}[0,T]\setminus\trLine)$,
and that~\eqref{eq:regEstimateCutoffs} holds true.

\textit{Step 6: Estimates for the bulk cutoff.}
By construction it holds~$\eta_{\mathrm{bulk}}(\cdot,t)\equiv 1$ outside of the space-time 
domain~$B_{\hat{r}}(\trLine(t))\cup\bigcup_{i,j\in\{1,2,3\},i\neq j}
\Psi_{i,j}(\barI_{i,j}(t){\times}\{t\}{\times}[-\bar r,\bar r])$
for all~$t\in [0,T]$. Hence, for a proof of~$\eta_{\mathrm{bulk}}\in [0,1]$
and the estimates~\eqref{eq:coercivityBulkCutoff}--\eqref{eq:coercivityTimeDerivBulkCutoff},
we may restrict our attention to~$\bigcup_{i,j\in\{1,2,3\},i\neq j}
\Psi_{i,j}(\barI_{i,j}(t){\times}\{t\}{\times}[-\bar r,\bar r])\setminus B_{\hat{r}}(\trLine(t))$
and~$B_{\hat{r}}(\trLine(t))$ for all~$t\in [0,T]$.

In view of the choice~\eqref{eq:defLocScale} of the localization scale~$\bar r$, 
one may argue separately on $\Psi_{i,j}(\barI_{i,j}(t){\times}\{t\}{\times}[-\bar r,\bar r])
\setminus B_{\hat{r}}(\trLine(t))$ for each pair of distinct
phases~$i,j\in\{1,2,3\}$ and all~$t\in [0,T]$. Because of the localization
properties~\eqref{eq:supportCutOffTwoInterfaces}
and~\eqref{eq:supportCutOffTripleLine} it holds
\begin{align}
\nonumber
\eta_{\mathrm{bulk}}(\cdot,t) &= 1 - \eta_{\barI_{i,j}}(\cdot,t)
\\&\label{eq:bulkCutoffAwayTripleLine}
= 1 - \zeta_{\barI_{i,j}}(\cdot,t)
\quad\text{in } \Psi_{i,j}(\barI_{i,j}(t){\times}\{t\}{\times}[-\bar r,\bar r])
\setminus B_{\hat{r}}(\trLine(t))
\end{align}
for all~$t\in [0,T]$. Hence, $\eta_{\mathrm{bulk}}\in [0,1]$
and the estimates~\eqref{eq:coercivityBulkCutoff}--\eqref{eq:coercivityTimeDerivBulkCutoff}
follow from the definitions~\eqref{eq:defInterfaceCutoffAwayTripleLine} 
and~\eqref{eq:interfaceCutoff} in combination with the quadratic behaviour around the origin
of the profile~\eqref{QuadraticInterfaceCutOff}. Note in this context that~\eqref{eq:defLocScale}
precisely ensures that the error can be expressed in terms of~$\dist(\cdot,\mathcal{I})$ as required.

We move on to the argument in the ball~$B_{\hat r}(\trLine(t))$ for all~$t\in [0,T]$.
On interface wedges, we infer from the localization properties~\eqref{eq:supportCutOffInterfacesNearTripleLine}
and~\eqref{eq:supportCutOffTwoInterfaces} as well as the definitions~\eqref{eq:defInterfaceCutoffInterfaceWedge}
and~\eqref{eq:defTripleLineCutoffInterfaceWedge} that
\begin{align}
\nonumber
\eta_{\mathrm{bulk}}(\cdot,t) &= 1 
- \eta_{\trLine}(\cdot,t) - \eta_{\barI_{i,j}}(\cdot,t)
\\&\label{eq:bulkCutoffInterfaceWedge}
= 1 - \zeta_{\barI_{i,j}}(\cdot,t)
\quad\text{in } B_{\hat{r}}(\trLine(t)) \cap W_{\barI_{i,j}}(t)
\end{align}
for all~$t\in [0,T]$, so that the asserted bounds follow as in the previous case
together with the bound~\eqref{eq:compDistances3} to express the error in terms of~$\dist(\cdot,\mathcal{I})$.

On interpolation wedges, we may compute based on~\eqref{eq:supportCutOffInterfacesNearTripleLine}
and~\eqref{eq:supportCutOffTwoInterfaces} as well as~\eqref{eq:defInterfaceCutoffInterpolWedge1}
and~\eqref{eq:defTripleLineCutoffInterpolWedge} that
(recall the relation~$\lambda_{\bar\Omega_{i}}^{\barI_{k,i}}=1{-}\lambda_{\bar\Omega_{i}}^{\barI_{i,j}}$)
\begin{align}
\nonumber
\eta_{\mathrm{bulk}}(\cdot,t) &= 1 - \eta_{\trLine}(\cdot,t)
- \eta_{\barI_{i,j}}(\cdot,t) - \eta_{\barI_{k,i}}(\cdot,t)
\\&\label{eq:bulkCutoffInterpolWedge}
= \lambda_{\bar\Omega_{i}}^{\barI_{i,j}}(1{-}\zeta_{\barI_{i,j}})(\cdot,t)
+ (1{-}\lambda_{\bar\Omega_{i}}^{\barI_{i,j}})(1{-}\zeta_{\barI_{k,i}})(\cdot,t)
\quad\text{in } B_{\hat{r}}(\trLine(t)) \cap W_{\bar\Omega_{i}}(t)
\end{align}
for all~$t\in [0,T]$. It follows immediately that~$\eta_{\mathrm{bulk}}(\cdot,t)\in [0,1]$.
Moreover, the definition~\eqref{eq:interfaceCutoff}, the quadratic behavior around the origin
of the profile~\eqref{QuadraticInterfaceCutOff}, and the estimate~\eqref{eq:compDistances} directly imply~\eqref{eq:coercivityBulkCutoff}.
Finally, since
\begin{align}
\label{eq:bulkCutoffGradientInterpolWedge}
\nabla\eta_{\mathrm{bulk}}(\cdot,t) &= 
- \lambda_{\bar\Omega_{i}}^{\barI_{i,j}}(\cdot,t)\nabla\zeta_{\barI_{i,j}}(\cdot,t)
- (1{-}\lambda_{\bar\Omega_{i}}^{\barI_{i,j}})(\cdot,t)\nabla\zeta_{\barI_{k,i}}(\cdot,t)
\\&~~~\nonumber
- (\zeta_{\barI_{i,j}}{-}\zeta_{\barI_{k,i}})(\cdot,t)\nabla\lambda_i^{\barI_{i,j}}(\cdot,t)
\quad\text{in } B_{\hat{r}}(\trLine(t)) \cap W_{\bar\Omega_{i}}(t)
\end{align}
for all~$t\in [0,T]$, we obtain~\eqref{eq:coercivityGradientBulkCutoff} 
and~\eqref{eq:coercivityTimeDerivBulkCutoff}
because the blow-up of~$\nabla\lambda_{\bar\Omega_{i}}^{\barI_{i,j}}$, see 
Lemma~\ref{lemma:existenceInterpolationFunctions}, is cancelled to required
order by the term~$\zeta_{\barI_{i,j}}{-}\zeta_{\barI_{k,i}}$. Indeed, the 
latter vanishes to first order at the triple line and has a bounded
second-order spatial derivative within interpolation wedges.

\textit{Step 7: Error estimates for interface cutoffs.}
The bounds~\eqref{eq:coercivityInterfaceCutoff}--\eqref{eq:coercivityTimeDerivativeInterfaceCutoff}
are trivially fulfilled outside of~$B_{\hat r}(\trLine(t))$
for all~$t\in[0,T]$ by construction and the choice~\eqref{eq:defLocScale} of the
localization scale~$\bar r$. In view of the 
definitions~\eqref{eq:defInterfaceCutoffInterfaceWedge}--\eqref{eq:defInterfaceCutoffInterpolWedge2}
and the definition~\eqref{eq:tripleLineCutoff}, we also have
$\eta_{\barI_{k,i}}(\cdot,t) \leq 1 - \zeta_{\trLine}(\cdot,t)
\leq C\dist^2(\cdot,\trLine(t))$ in~$B_{\hat r}(\trLine(t))\cap\big(W_{\barI_{k,i}}(t)
\cup W_{\bar\Omega_i}(t)\cup W_{\bar\Omega_k}(t)\big)$
for all~$t\in[0,T]$. Recalling the bounds~\eqref{eq:compDistances} and~\eqref{eq:compDistances2},
this in turn implies~\eqref{eq:coercivityInterfaceCutoff}
throughout~$B_{\hat r}(\trLine(t))$ for all~$t\in[0,T]$.

For a proof of~\eqref{eq:coercivityGradientInterfaceCutoff}
and~\eqref{eq:coercivityTimeDerivativeInterfaceCutoff}, note that
\begin{align*}
|(\partial_t,\nabla) \eta_{\barI_{k,i}}(\cdot,t)| \leq C(1 {-} \zeta_{\trLine}(\cdot,t)) +
C|(\partial_t,\nabla)\dist(\cdot,\trLine(t))|\dist(\cdot,\trLine(t))
\end{align*}
in~$B_{\hat r}(\trLine(t))\cap\big(W_{\barI_{k,i}}(t)
\cup W_{\bar\Omega_i}(t)\cup W_{\bar\Omega_k}(t)\big)$
for all~$t\in[0,T]$. The first right hand side term is estimated as before,
while the second one is of required order due to the bounds~\eqref{eq:compDistances} resp.\ \eqref{eq:compDistances2}
and the regularity of the projection onto the triple line~$\trLine$,
see Definition~\ref{def:locRadius}, which in turn one may employ throughout~$B_{\hat r}(\trLine(t))$ based
on the representation~$|x{-}P_{\trLine}(x,t)|=\dist(x,\trLine(t))$. 
\end{proof}

\subsection{Construction of a gradient-flow calibration}
We have everything in place to provide the construction
of a gradient-flow calibration for a regular double bubble
smoothly evolving by MCF.
We first introduce a global definition for the vector fields $\xi_{i,j}$
extending the unit normal vector fields $\no_{i,j}|_{\barI_{i,j}}$ of the interfaces $\barI_{i,j}$.

\begin{construction}[Global extensions of the unit normal vector fields~$\no_{i,j}|_{\barI_{i,j}}$]
\label{globalXi}
Let $(\bar\Omega_1,\bar\Omega_2,\bar\Omega_3)$ be a regular 
double bubble smoothly evolving by MCF in the sense of Definition~\ref{def:smoothSolution} 
on a time interval~$[0,T]$. Let $(\eta_{\trLine},(\eta_{\barI_{i,j}})_{i,j\in\{1,2,3\},i\neq j})$
be the partition of unity from the proof of Lemma~\ref{lemma:partitionOfUnity}. Fix $i,j\in\{1,2,3\}$
with $i\neq j$. We then define a family of vector fields
\begin{align}
\label{eq:localXiInterface}
\xi_{i,j}^{\barI_{k,l}}\colon\bigcup_{t\in [0,T]}
\supp\eta_{\barI_{k,l}}(\cdot,t)\times\{t\}&\to \overline{B_1(0)},
\quad k,l\in\{1,2,3\},\, k\neq l,
\\
\label{eq:localXiTripleLine}
\xi_{i,j}^{\trLine}\colon\bigcup_{t\in [0,T]}
\supp\eta_{\trLine}(\cdot,t)\times\{t\}&\to \overline{B_1(0)}
\end{align}
by means of the following procedure:

For $k,l\in\{1,2,3\}$ with $(k,l)\in\{(i,j),(j,i)\}$ we let
$\xi_{i,j}^{\barI_{k,l}}$ be the corresponding vector field
from Construction~\ref{gradientFlowCalibrationInterface}
for the interface $\barI_{k,l}$. For $k,l\in\{1,2,3\}$ with $(k,l)\notin\{(i,j),(j,i)\}$ and $k\neq l$ we define
$\xi^{\barI_{k,l}}_{i,j}:=\frac{1}{2}(\frac{\sigma_{l,i}-\sigma_{l,j}}{\sigma_{i,j}}\xi^{\barI_{k,l}}_{k,l}
+\frac{\sigma_{k,i}-\sigma_{k,j}}{\sigma_{i,j}}\xi^{\barI_{k,l}}_{l,k})$, which
is well-defined reversing the roles of $i,j$ and $k,l$ in the previous step.
Finally, we denote by~$\xi_{i,j}^{\trLine}$ the corresponding vector field
from the proof of Proposition~\ref{prop:compatibilityEstimatesLocalConstructions}.

With this family of local vector fields in place, we now define 
a global vector field~$\xi_{i,j}\colon \Rd[3]\times [0,T] \to \Rd[3]$ by means of
\begin{align}
\label{eq:globalXi}
\xi_{i,j} &:=  \eta_{\trLine} \xi^{\trLine}_{i,j}
+ \eta_{\barI_{1,2}} \xi_{i,j}^{\barI_{1,2}}
+ \eta_{\barI_{2,3}} \xi_{i,j}^{\barI_{2,3}}
+ \eta_{\barI_{3,1}} \xi_{i,j}^{\barI_{3,1}}
\end{align}
for all distinct pairs of phases~$i,j\in\{1,2,3\}$.
\hfill$\diamondsuit$
\end{construction}

We proceed by showing that the vector fields from the previous construction
satisfy the structural assumption~\eqref{Calibrations} and the coercivity estimate~\eqref{LengthControlXi}
of a gradient-flow calibration.

\begin{lemma}
\label{lemma:calibrationProperty}
Let the assumptions and notation of Construction~\ref{globalXi} be in place.
Fix~$i,j\in\{1,2,3\}$ such that~$i\neq j$.
The vector field $\xi_{i,j}$ is then
subject to the following list of properties:
\begin{itemize}[leftmargin=0.7cm]
\item[i)] It holds $\xi_{i,j} \in (C^0_tC^1_x\cap C^1_tC^0_x)(\mathbb{R}^3{\times}[0,T]\setminus\trLine)$,
					and there exists a constant~$C>0$ which depends only on the data of the smoothly evolving regular double 
					bubble~$(\bar\Omega_1,\bar\Omega_2,\bar\Omega_3)$ on~$[0,T]$ such that
					\begin{align}
					\label{eq:regEstimateGlobalExtensionXi}
					|(\partial_t,\nabla)\xi_{i,j}| \leq C \quad\text{in } \mathbb{R}^3{\times}[0,T]\setminus\trLine.
					\end{align}
					Moreover, it holds $\xi_{i,j}=\no_{i,j}$ along~$\barI_{i,j}$.
\item[ii)] For each phase~$i\in\{1,2,3\}$, there exists a vector 
					 field~$\xi_{i}\colon\mathbb{R}^3{\times}[0,T]\to\Rd[3]$
					 of class $(C^0_tC^1_x\cap C^1_tC^0_x)(\mathbb{R}^3{\times}[0,T]\setminus\trLine)$
					 such that $\sigma_{i,j}\xi_{i,j}=\xi_{i}-\xi_{j}$ holds true on~$\mathbb{R}^3{\times}[0,T]$.
\item[iii)] There exists a constant~$c\in (0,1)$, which depends only on the data of the smoothly evolving regular double 
						bubble~$(\bar\Omega_1,\bar\Omega_2,\bar\Omega_3)$ on~$[0,T]$, such that
						\begin{align}
						\label{eq:quadraticLengthControl}
						c(\dist^2(\cdot,\barI_{i,j}) \wedge 1) \leq 1 - |\xi_{i,j}| 
						\quad\text{in } \mathbb{R}^3{\times}[0,T].
						\end{align}
\end{itemize}
\end{lemma}

\begin{proof}
The proof is performed in three steps.

\textit{Step 1: Regularity and structural properties.} The asserted qualitative regularity
of the vector fields $\xi_{i,j}$
together with the estimate~\eqref{eq:regEstimateGlobalExtensionXi}
follows from the definition~\eqref{eq:globalXi},
the regularity~\eqref{eq:regEstimateCutoffs} of the cutoff functions,
as well as the regularity of the local building blocks~\eqref{eq:localXiInterface} and~\eqref{eq:localXiTripleLine}
in form of
\begin{align}
\label{eq:auxRegEstimatesLocalXi}
\big|(\partial_t,\nabla)(\xi_{i,j}^{\barI_{k,l}},\xi^{\trLine}_{i,j})\big| 
\leq C \quad\text{in } \Rd[3]{\times} [0,T],
\end{align}
which in turn is a consequence of the definitions from Construction~\ref{globalXi}
and the regularity estimates~\eqref{eq:regXiInterface} and~\eqref{eq:regXiTripleLine}.
The property~$\xi_{i,j}|_{\barI_{i,j}}\equiv\no_{i,j}$
is immediate from the definition~\eqref{eq:globalXi}, the fact that 
$(\eta_{\trLine},(\eta_{\barI_{i,j}})_{i,j\in\{1,2,3\},i\neq j})$
constitutes a partition of unity along the network~$\mathcal{I}$,
and the corresponding property in terms of the local constructions
from Lemma~\ref{lemma:gradientFlowCalibrationInterface} and Proposition~\ref{prop:gradientFlowCalibrationTripleLine}.

The existence of vector fields $(\xi_i)_{i\in\{1,2,3\}}$ of class 
$(C^0_tC^1_x\cap C^1_tC^0_x)(\mathbb{R}^3{\times}[0,T]\setminus\trLine)$
such that $\sigma_{i,j}\xi_{i,j}=\xi_{i}-\xi_{j}$ holds true on~$\mathbb{R}^3{\times}[0,T]$
follows from the following considerations.
Let $i,j,k\in\{1,2,3\}$ be pairwise distinct.
We define $\xi^{\trLine}_{i}:=\frac{1}{3}(\sigma_{i,j}\xi^{\trLine}_{i,j}+\sigma_{i,k}\xi^{\trLine}_{i,k})$.
Since $\sigma_{1,2}\xi^{\trLine}_{1,2}+\sigma_{2,3}\xi^{\trLine}_{2,3}+\sigma_{3,1}\xi^{\trLine}_{3,1}=0$
holds true in the support of $\eta_{\trLine}$, see Proposition~\ref{prop:gradientFlowCalibrationTripleLine},
we indeed obtain $\sigma_{i,j}\xi^{\trLine}_{i,j}=\xi^{\trLine}_{i}-\xi^{\trLine}_{j}$.
Next, fix $k,l\in\{1,2,3\}$ with $k\neq l$, and let $i\in\{1,2,3\}$.
We may then define $\xi^{\barI_{k,l}}_i := \frac{1}{2}
(\sigma_{l,i}\xi^{\barI_{k,l}}_{k,l}+\sigma_{k,i}\xi^{\barI_{k,l}}_{l,k})$.
Again, plugging in the definitions immediately shows
$\sigma_{i,j}\xi^{\barI_{k,l}}_{i,j}=\xi^{\barI_{k,l}}_{i}-\xi^{\barI_{k,l}}_{j}$
for all $i,j\in\{1,2,3\}$ such that~$i\neq j$. Defining
$\xi_i := \eta_{\trLine}\xi^{\trLine}_i + \eta_{\barI_{1,2}}\xi^{\barI_{1,2}}_i 
+ \eta_{\barI_{2,3}}\xi^{\barI_{2,3}}_i + \eta_{\barI_{3,1}}\xi^{\barI_{3,1}}_i$
therefore entails the desired conclusion.

\textit{Step 2: A coercivity condition.}
As a preparation for the proof of~\eqref{eq:quadraticLengthControl},
we claim that there exists a constant $\varepsilon=\varepsilon(\sigma)\in (0,1)$ such that for all $i,j\in\{1,2,3\}$
with $i\neq j$, as well as all $k,l\in\{1,2,3\}$ with $(k,l)\notin\{(i,j),(j,i)\}$ and $k\neq l$
it holds
\begin{align}
\label{eq:coercivityAbsentPhase}
\big|\xi^{\barI_{k,l}}_{i,j}\big| \leq \varepsilon < 1.
\end{align}
Indeed, the estimate~\eqref{eq:coercivityAbsentPhase} is an 
immediate consequence of the definition
of the vector field~$\xi^{\barI_{k,l}}_{i,j}=\frac{1}{2}(\frac{\sigma_{l,i}-\sigma_{l,j}}{\sigma_{i,j}}\xi^{\barI_{k,l}}_{k,l}
+\frac{\sigma_{k,i}-\sigma_{k,j}}{\sigma_{i,j}}\xi^{\barI_{k,l}}_{l,k})$, see Construction~\ref{globalXi},
and the fact that $|\frac{\sigma_{l,i}-\sigma_{l,j}}{\sigma_{i,j}}|<1$
resp.\ $|\frac{\sigma_{k,i}-\sigma_{k,j}}{\sigma_{i,j}}|<1$, which in turn
is true since the matrix of surface tensions satisfies the strict triangle inequality
by assumption.

\textit{Step 3: Proof of the estimate~\eqref{eq:quadraticLengthControl}.}
Fix~$i,j\in\{1,2,3\}$ such that~$i\neq j$.
By the localization properties~\eqref{eq:supportCutOffInterfaces}--\eqref{eq:supportCutOffTripleLine}
and the choice~\eqref{eq:defLocScale} of the localization scale~$\bar r$,
it suffices to establish the desired estimate throughout~$\supp\eta_{\barI_{k,l}}(\cdot,t)\setminus B_{\hat r}(\trLine(t))$,
$B_{\hat r}(\trLine(t))\cap W_{\barI_{k,l}}(t)$ or $B_{\hat r}(\trLine(t))\cap W_{\bar\Omega_l}(t)$ 
for all distinct phases~$k,l\in\{1,2,3\}$ and all~$t\in [0,T]$. Hence, fix such~$k,l\in\{1,2,3\}$
with~$k\neq l$ and~$t\in [0,T]$, and then observe that due to the definition~\eqref{eq:globalXi}
and the localization properties~\eqref{eq:supportCutOffInterfaces}--\eqref{eq:supportCutOffTripleLine}
it holds
\begin{align}
\label{eq:representationGlobalXi}
\xi_{i,j} = \begin{cases}
						\eta_{\barI_{k,l}}\xi^{\barI_{k,l}}_{i,j} 
						& \text{on } \supp\eta_{\barI_{k,l}}(\cdot,t)\setminus B_{\hat r}(\trLine(t)),
						\\[0.5ex]
						\eta_{\trLine}\xi^{\trLine}_{i,j} + \eta_{\barI_{k,l}}\xi^{\barI_{k,l}}_{i,j} 
						& \text{on } B_{\hat r}(\trLine(t))\cap W_{\barI_{k,l}}(t),
						\\[0.5ex]
						\eta_{\trLine}\xi^{\trLine}_{i,j} + \eta_{\barI_{k,l}}\xi^{\barI_{k,l}}_{i,j}
						+ \eta_{\barI_{l,m}}\xi^{\barI_{l,m}}_{i,j}
						& \text{on } B_{\hat r}(\trLine(t))\cap W_{\bar\Omega_l}(t),\,m\in\{1,2,3\}\setminus\{k,l\}.
						\end{cases}
\end{align}
Based on~\eqref{eq:representationGlobalXi}, we now distinguish between two cases.

\textit{Substep 3.1:} 
Assume that~$(k,l)\in\{(i,j),(j,i)\}$. In other words, both the phases~$k$ and~$l$
are present at the interface~$\barI_{i,j}$. In this case, observe first that throughout the 
three domains represented in~\eqref{eq:representationGlobalXi} it holds 
due to~\eqref{eq:compDistances}, \eqref{eq:compDistances3} and~\eqref{eq:defLocScale} that
the distance to~$\mathcal{I}$ is comparable to the distance to~$\barI_{i,j}$:
$\frac{1}{C}\dist(\cdot,\barI_{i,j})\leq\dist(\cdot,\mathcal{I})\leq C\dist(\cdot,\barI_{i,j})$
for some constant~$C\geq 1$. Furthermore, it follows from~\eqref{eq:representationGlobalXi}
and the triangle inequality that~$|\xi_{i,j}|\leq\mathrm 1{-}\eta_{\mathrm{bulk}}$
throughout the three domains represented in~\eqref{eq:representationGlobalXi}.
Hence, the bound~\eqref{eq:quadraticLengthControl} follows from the lower bound in~\eqref{eq:coercivityBulkCutoff}.

\textit{Substep 3.2:} Assume that~$(k,l)\notin\{(i,j),(j,i)\}$.
In the first case of~\eqref{eq:representationGlobalXi}, the estimate~\eqref{eq:quadraticLengthControl}
follows immediately from the coercivity condition~\eqref{eq:coercivityAbsentPhase}.
In the third case of~\eqref{eq:representationGlobalXi}, we may additionally
assume that~$(l,m)\notin\{(i,j),(j,i)\}$; otherwise, we are again in the setting of the
argument from \textit{Substep~3.1} above. Plugging in the definitions~\eqref{eq:defInterfaceCutoffInterpolWedge1},
\eqref{eq:defInterfaceCutoffInterpolWedge2} and~\eqref{eq:defTripleLineCutoffInterpolWedge}, 
as well as exploiting the coercivity condition~\eqref{eq:coercivityAbsentPhase}
for both the vector fields~$\xi^{\barI_{k,l}}_{i,j}$ and~$\xi^{\barI_{l,m}}_{i,j}$
(which is admissible due to our assumptions), we may estimate from below
\begin{align*}
1-|\xi_{i,j}| &\geq 1 {-} \big(\eta_{\trLine} + \varepsilon\eta_{\barI_{k,l}} + \varepsilon\eta_{\barI_{l,m}}\big) 
\\&
\geq (1-\varepsilon)(1-\zeta_{\trLine}) \geq (1-\varepsilon)(\dist^2(\cdot,\trLine) \wedge 1)
\end{align*}
on~$B_{\hat r}(\trLine(t))\cap W_{\bar\Omega_l}(t)$ for all~$t\in [0,T]$,
so that~\eqref{eq:quadraticLengthControl} follows again. Since
the argument proceeds similarly in the second case of~\eqref{eq:representationGlobalXi},
we may conclude the proof.
%
%
\end{proof}

The next step consists of providing the global definition
of a suitable velocity field along which a smoothly
evolving regular double bubble and our associated constructions
are transported.

\begin{construction}[Global extension of velocity vector field]
\label{globalVelocity}
Let $(\bar\Omega_1,\bar\Omega_2,\bar\Omega_3)$ be a regular 
double bubble smoothly evolving by MCF in the sense of Definition~\ref{def:smoothSolution} 
on a time interval~$[0,T]$. Let $(\eta_{\trLine},(\eta_{\barI_{i,j}})_{i,j\in\{1,2,3\},i\neq j})$
be the partition of unity from the proof of Lemma~\ref{lemma:partitionOfUnity}. We then introduce a family of vector fields
\begin{align}
\label{eq:localVelocityInterface}
B^{\barI_{i,j}}\colon\bigcup_{t\in [0,T]}
\supp\eta_{\barI_{i,j}}(\cdot,t)\times\{t\}&\to\Rd[3]
\quad\text{for all } i,j\in\{1,2,3\},\, i\neq j,
\\
\label{eq:localVelocityTripleLine}
B^{\trLine}\colon\bigcup_{t\in [0,T]}
\supp\eta_{\trLine}(\cdot,t)\times\{t\}&\to\Rd[3]
\end{align}
as follows: the velocity field~$B^{\trLine}$ denotes the
corresponding vector field from the proof of 
Proposition~\ref{prop:compatibilityEstimatesLocalConstructions},
whereas~$B^{\barI_{i,j}}$ is the velocity field from Construction~\ref{gradientFlowCalibrationInterface}
with tangential component chosen as in the proof of Proposition~\ref{prop:compatibilityEstimatesLocalConstructions}.

With this family of local vector fields in place, we now define 
a global velocity field by means of
\begin{align}
\label{eq:globalVelocity}
B &:=  \eta_{\trLine} B^{\trLine}
+ \eta_{\barI_{1,2}} B^{\barI_{1,2}}
+ \eta_{\barI_{2,3}} B^{\barI_{2,3}}
+ \eta_{\barI_{3,1}} B^{\barI_{3,1}}
\end{align}
throughout~$\Rd[3]{\times}[0,T]$.
\hfill$\diamondsuit$
\end{construction}

A crucial ingredient for the proof of the estimates~\eqref{TransportEquationXi} and~\eqref{LengthConservation}
are the following bounds on the advective derivatives of the partition of unity
from Lemma~\ref{lemma:partitionOfUnity}.

\begin{lemma}
\label{lemma:advectionCutOff}
Let the assumptions and notation of Construction~\ref{globalVelocity} be in place.
In particular, $(\eta_{\trLine},(\eta_{\barI_{i,j}})_{i,j\in\{1,2,3\},i\neq j})$
denotes the partition of unity from the proof of Lemma~\ref{lemma:partitionOfUnity}. 
Then $B\in C^0_tC^1_x(\mathbb{R}^3{\times}[0,T]\setminus\trLine)$ with corresponding estimate
\begin{align}
\label{eq:regEstimateGlobalVel}
|B| + |\nabla B| \leq C \quad\text{in } \Rd[3]{\times}[0,T]\setminus\trLine.
\end{align}
Moreover, the velocity field~$B$ gives rise to an improved estimate
on the advective derivative of the bulk cutoff in form of
\begin{align}
\label{eq:advectionCutoffBulk}
|\partial_t\eta_{\mathrm{bulk}} + (B\cdot\nabla)\eta_{\mathrm{bulk}}|
&\leq C(\dist^2(\cdot,\mathcal{I})\wedge 1)
\quad\text{in } \Rd[3]\times [0,T],
\end{align}
and similarly for all pairwise distinct phases~$i,j,k\in\{1,2,3\}$
\begin{align}
\label{eq:advectionCutoffInterface}
|\partial_t\eta_{\barI_{k,i}} + (B\cdot\nabla)\eta_{\barI_{k,i}}|
&\leq C(\dist^2(\cdot,\barI_{i,j})\wedge 1)
\quad\text{in } \Rd[3]\times [0,T].
\end{align}
The constant~$C > 0$ in the 
estimates~\emph{\eqref{eq:regEstimateGlobalVel}--\eqref{eq:advectionCutoffInterface}}
depends only on the data of the smoothly evolving regular double 
bubble~$(\bar\Omega_1,\bar\Omega_2,\bar\Omega_3)$ on~$[0,T]$.
\end{lemma}

\begin{proof}
The proof is decomposed into three steps.

\textit{Step 1: Regularity estimates.}
The asserted qualitative regularity of the velocity field~$B$ 
together with the associated estimate~\eqref{eq:regEstimateGlobalVel}
follow from its definition~\eqref{eq:globalVelocity},
the regularity~\eqref{eq:regEstimateCutoffs} of the cutoff functions,
as well as the regularity of the local building blocks~\eqref{eq:localVelocityInterface} 
and~\eqref{eq:localVelocityTripleLine} in form of
\begin{align}
\label{eq:regEstimateLocalVelocitiesAux}
\big|(B^{\trLine},B^{\barI_{i,j}})\big| 
+ \big|\nabla (B^{\trLine},B^{\barI_{i,j}})\big| \leq C
\quad\text{in } \Rd[3]{\times}[0,T],
\end{align}
which is a consequence of~\eqref{eq:regVelocityInterface}
and~\eqref{eq:regVelocityTripleLine}.

\textit{Step 2: Proof of~\eqref{eq:advectionCutoffBulk}.}
It holds~$\eta_{\mathrm{bulk}}(\cdot,t)\equiv 1$ outside of the space-time 
domain~$B_{\hat{r}}(\trLine(t))\cup\bigcup_{i,j\in\{1,2,3\},i\neq j}
\Psi_{i,j}(\barI_{i,j}(t){\times}\{t\}{\times}[-\bar r,\bar r])$
for all~$t\in [0,T]$ by construction. Hence, for a proof of the
estimate~\eqref{eq:advectionCutoffBulk},
we may restrict our attention to $\bigcup_{i,j\in\{1,2,3\},i\neq j}
\Psi_{i,j}(\barI_{i,j}(t){\times}\{t\}{\times}[-\bar r,\bar r])\setminus B_{\hat{r}}(\trLine(t))$
and~$B_{\hat{r}}(\trLine(t))$ for all~$t\in [0,T]$.
By the choice~\eqref{eq:defLocScale} of the localization scale~$\bar r$, 
one may even argue separately on $\Psi_{i,j}(\barI_{i,j}(t){\times}\{t\}{\times}[-\bar r,\bar r])
\setminus B_{\hat{r}}(\trLine(t))$ for each pair of distinct
phases~$i,j\in\{1,2,3\}$ and all~$t\in [0,T]$.

\textit{Substep 2.1: Proof of~\eqref{eq:advectionCutoffBulk}
on~$\Psi_{i,j}(\barI_{i,j}(t){\times}\{t\}{\times}[-\bar r,\bar r])
\setminus B_{\hat{r}}(\trLine(t))$.} 
It follows from the representation~\eqref{eq:bulkCutoffAwayTripleLine}
and the definition~\eqref{eq:globalVelocity}
that $B=\eta_{\barI_{i,j}}B^{\barI_{i,j}}$ and
\begin{align}
\label{eq:auxAdvection1}
|\partial_t\eta_{\mathrm{bulk}} + (B\cdot\nabla)\eta_{\mathrm{bulk}}|
\leq \big|\partial_t\zeta_{\barI_{i,j}} + (B^{\bar{I}_{i,j}}\cdot\nabla)\zeta_{\barI_{i,j}}\big|
+ \eta_{\mathrm{bulk}}\big|(B^{\bar{I}_{i,j}}\cdot\nabla)\zeta_{\barI_{i,j}}\big|
\end{align}
throughout~$\Psi_{i,j}(\barI_{i,j}(t){\times}\{t\}{\times}[-\bar r,\bar r])
\setminus B_{\hat{r}}(\trLine(t))$ for all~$t\in [0,T]$.

Recall that the signed distance~$s_{i,j}$ satisfies
\begin{align}
\label{eq:evolSignedDistanceAux1000}
\partial_t s_{i,j}+(B^{\barI_{i,j}}\cdot\nabla)s_{i,j}=0
\quad\text{in } \mathrm{im}(\Psi_{i,j})
\end{align}
as a consequence of the choice of the local velocity~$B^{\barI_{i,j}}$, cf.\ Construction~\ref{globalVelocity},
Construction~\ref{gradientFlowCalibrationInterface} and~\eqref{eq:evolutionSignedDistance}. 
Hence, we infer from the definition~\eqref{eq:interfaceCutoff} 
and an application of the chain rule that
\begin{align}
\label{eq:evolCutoffProfile}
\partial_t\zeta_{\barI_{i,j}} + (B^{\bar{I}_{i,j}}\cdot\nabla)\zeta_{\barI_{i,j}} = 0
\quad\text{in } \mathrm{im}(\Psi_{i,j}).
\end{align}
For an estimate of the second right hand side term of~\eqref{eq:auxAdvection1},
we simply make use of the upper bound for the bulk cutoff~\eqref{eq:coercivityBulkCutoff}
as well as the regularity estimates~\eqref{eq:regEstimateLocalVelocitiesAux} and~\eqref{eq:regEstimateInterfaceCutoff}
of~$B^{\barI_{i,j}}$ and~$\zeta_{\barI_{i,j}}$, respectively.

\textit{Substep 2.2: Proof of~\eqref{eq:advectionCutoffBulk}
on~$B_{\hat{r}}(\trLine(t))\cap W_{\barI_{i,j}}(t)$.}
Throughout the interface wedge~$W_{\barI_{i,j}}(t)\cap B_{\hat r}(\trLine(t))$, it holds $B=\eta_{\trLine}B^{\trLine}
+ \eta_{\barI_{i,j}}B^{\barI_{i,j}}$ thanks to the representation~\eqref{eq:bulkCutoffInterfaceWedge}
and the definition~\eqref{eq:globalVelocity}. We may then estimate,
making use again of~\eqref{eq:bulkCutoffInterfaceWedge},
\begin{align}
\label{eq:auxAdvection2}
|\partial_t\eta_{\mathrm{bulk}} + (B\cdot\nabla)\eta_{\mathrm{bulk}}|
&\leq
\big|\partial_t\zeta_{\barI_{i,j}} + (B^{\bar{I}_{i,j}}\cdot\nabla)\zeta_{\barI_{i,j}}\big|
\\&~~~\nonumber
+ \eta_{\mathrm{bulk}}\big|(B^{\bar{I}_{i,j}}\cdot\nabla)\zeta_{\barI_{i,j}}\big|
+ \eta_{\trLine}\big|B^{\trLine} - B^{\bar{I}_{i,j}}\big| |\nabla\zeta_{\barI_{i,j}}|
\end{align}
on~$W_{\barI_{i,j}}(t)\cap B_{\hat r}(\trLine(t))$ for all~$t\in [0,T]$.
Thanks to~\eqref{eq:interfaceWedge}, the identity~\eqref{eq:evolCutoffProfile}
is still applicable on an interface wedge. In particular, the first two right hand side
terms of~\eqref{eq:auxAdvection2} can be estimated along the same lines as in
\textit{Substep~2.1}. The third right hand side term is of required
order due to the compatibility estimate~\eqref{eq:compEstimate3}, the bound~\eqref{eq:compDistances3},
and the regularity estimate~\eqref{eq:regEstimateInterfaceCutoff}.

\textit{Substep 2.3: Proof of~\eqref{eq:advectionCutoffBulk}
on~$B_{\hat{r}}(\trLine(t))\cap W_{\bar\Omega_i}(t)$.}
Throughout~$W_{\bar\Omega_{i}}(t)\cap B_{\hat r}(\trLine(t))$,
we may represent, as a consequence of the identity~\eqref{eq:bulkCutoffInterpolWedge}, 
the global velocity defined by~\eqref{eq:globalVelocity}
in form of $B=\eta_{\trLine}B^{\trLine}+ \eta_{\barI_{i,j}}B^{\barI_{i,j}}+\eta_{\barI_{k,i}}B^{\barI_{k,i}}$.
Plugging in~\eqref{eq:bulkCutoffInterpolWedge} and adding zero twice then entails
\begin{align}
\nonumber
&|\partial_t\eta_{\mathrm{bulk}} {+} (B\cdot\nabla)\eta_{\mathrm{bulk}}|
\\&\leq\label{eq:auxAdvection3}
\big|\partial_t\lambda^{\barI_{i,j}}_{\bar\Omega_i} {+} (B\cdot\nabla)\lambda^{\barI_{i,j}}_{\bar\Omega_i}\big|
|\zeta_{\barI_{i,j}}{-}\zeta_{\barI_{k,i}}|
\\&~~~\nonumber
+\lambda^{\barI_{i,j}}_{\bar\Omega_i} 
\big|\partial_t\zeta_{\barI_{i,j}} {+} (B^{\barI_{i,j}}\cdot\nabla)\zeta_{\barI_{i,j}}\big|
+(1{-}\lambda^{\barI_{i,j}}_{\bar\Omega_i} )
\big|\partial_t\zeta_{\barI_{i,k}} {+} (B^{\barI_{k,i}}\cdot\nabla)\zeta_{\barI_{k,i}}\big|
\\&~~~\nonumber
+\lambda^{\barI_{i,j}}_{\bar\Omega_i} \eta_{\trLine}\big|B^{\barI_{i,j}}{-}B^{\trLine}\big||\nabla\zeta_{\barI_{i,j}}|
+(1{-}\lambda^{\barI_{i,j}}_{\bar\Omega_i} )\eta_{\trLine}\big|B^{\barI_{k,i}}{-}B^{\trLine}\big||\nabla\zeta_{\barI_{k,i}}|
\\&~~~\nonumber
+\lambda^{\barI_{i,j}}_{\bar\Omega_i} \eta_{\barI_{k,i}}\big|B^{\barI_{i,j}}{-}B^{\barI_{k,i}}\big||\nabla\zeta_{\barI_{i,j}}|
+(1{-}\lambda^{\barI_{i,j}}_{\bar\Omega_i} )
\eta_{\barI_{i,j}}\big|B^{\barI_{k,i}}{-}B^{\barI_{i,j}}\big||\nabla\zeta_{\barI_{k,i}}|
\\&~~~\nonumber
+\lambda^{\barI_{i,j}}_{\bar\Omega_i} \eta_{\mathrm{bulk}}\big|(B^{\bar{I}_{i,j}}\cdot\nabla)\zeta_{\barI_{i,j}}\big|
+(1{-}\lambda^{\barI_{i,j}}_{\bar\Omega_i} )\eta_{\mathrm{bulk}}\big|(B^{\bar{I}_{k,i}}\cdot\nabla)\zeta_{\barI_{k,i}}\big|.
\end{align}
The last eight right hand side terms of~\eqref{eq:auxAdvection3} can be estimated by means of the 
same ingredients as in the previous two substeps, relying in the process also 
on~\eqref{eq:interpolWedgeHalfspace} and~\eqref{eq:compDistances}. 
Hence, we focus only on the first right hand side term of~\eqref{eq:auxAdvection3}.
Since the difference~$\zeta_{\barI_{i,j}}-\zeta_{\barI_{k,i}}$ vanishes to first order at the triple line and has a bounded
second-order spatial derivative within interpolation wedges, we have the bound
\begin{align}
\label{eq:aux1000}
|\zeta_{\barI_{i,j}}-\zeta_{\barI_{k,i}}| \leq C\dist^2(\cdot,\trLine)
\end{align}
on~$W_{\bar\Omega_i}(t)\cap B_{\hat r}(\trLine(t))$ for all~$t\in [0,T]$.
Since the advective derivative of the interpolation parameter is bounded within
interpolation wedges in form of~\eqref{eq:advDerivInterpolFunction}, we
may add zero and exploit the property~\eqref{eq:extensionPropertyVelocity}
as well as the regularity estimates~\eqref{eq:regEstimateGlobalVel}
and~\eqref{eq:regEstimateInterpolFunction} to obtain
\begin{align}
\label{eq:advectionInterpolFunctionByGlobalVel}
\big|\partial_t\lambda^{\barI_{i,j}}_{\bar\Omega_i} 
{+} (B\cdot\nabla)\lambda^{\barI_{i,j}}_{\bar\Omega_i}\big| \leq C
\end{align}
throughout~$W_{\bar\Omega_i}(t)\cap B_{\hat r}(\trLine(t))$ for all~$t\in [0,T]$.
Post-processing~\eqref{eq:aux1000} by means of~\eqref{eq:compDistances}
thus entails~\eqref{eq:advectionCutoffBulk} on $W_{\bar\Omega_i}(t)\cap B_{\hat r}(\trLine(t))$ 
for all~$t\in [0,T]$.

\textit{Step 3: Proof of~\eqref{eq:advectionCutoffInterface}.}
Fix~$i,j,k\in\{1,2,3\}$ such that~$\{i,j,k\}=\{1,2,3\}$.
Due to the localization properties~\eqref{eq:supportCutOffInterfaces}--\eqref{eq:supportCutOffTwoInterfaces},
the choice~\eqref{eq:defLocScale} of the localization scale~$\bar r$,
and the regularity estimates~\eqref{eq:regEstimateCutoffs} and~\eqref{eq:regEstimateGlobalVel},
the estimate~\eqref{eq:advectionCutoffInterface} is satisfied for trivial reasons
outside of $B_{\hat r}(\trLine(t))\cap(W_{\bar\Omega_k}(t)\cup W_{\bar\Omega_i}(t)\cup W_{\barI_{k,i}}(t))$
for all~$t\in[0,T]$.

\textit{Substep 3.1: Proof of~\eqref{eq:advectionCutoffInterface} 
on $B_{\hat r}(\trLine(t))\cap W_{\barI_{k,i}}(t)$.}
Based on the representation~\eqref{eq:bulkCutoffInterfaceWedge}
as well as the definition~\eqref{eq:defInterfaceCutoffInterfaceWedge},
it holds $\eta_{\barI_{k,i}}=(1{-}\zeta_{\trLine})(1{-}\eta_{\mathrm{bulk}})$
on $B_{\hat r}(\trLine(t))\cap W_{\barI_{k,i}}(t)$ for all~$t\in [0,T]$.
By an application of the product rule and the already established
estimate~\eqref{eq:advectionCutoffBulk} for the advective derivative
of the bulk cutoff we thus infer
\begin{align*}
\big|\partial_t\eta_{\barI_{k,i}} + (B\cdot\nabla)\eta_{\barI_{k,i}}\big|
\leq |\partial_t\zeta_{\trLine} + (B\cdot\nabla)\zeta_{\trLine}| 
+ C(\dist^2(\cdot,\barI_{i,j})\wedge 1)
\end{align*}
on $B_{\hat r}(\trLine(t))\cap W_{\barI_{k,i}}(t)$ for all~$t\in [0,T]$.
Expressing $\dist(x,\trLine(t))=|x{-}P_{\trLine}(x,t)|$
for all~$x\in B_{\hat r}(\trLine(t))$ and all~$t\in [0,T]$,
as well as recalling the relations~\eqref{eq:timeEvolutionProjContactLine}
and~\eqref{eq:extensionPropertyVelocity}, we may compute
\begin{align}
\nonumber
\partial_t \dist(x,\trLine(t))
&= - \frac{x{-}P_{\trLine}(x,t)}{|x{-}P_{\trLine}(x,t)|} \cdot B(P_{\trLine}(x,t),t)
\\& \label{eq:transportByGlobalVelDistanceTripleLine}
= - \big(B(P_{\trLine}(x,t),t)\cdot\nabla\big)\dist(x,\trLine(t))
\end{align}
for all~$x\in B_{\hat r}(\trLine(t))\setminus\trLine(t)$ and all~$t\in [0,T]$.
It is now a consequence of the chain rule and 
the regularity estimates~\eqref{eq:regEstimateGlobalVel} resp.\ \eqref{eq:regEstimateTripleLineCutoff} that
\begin{align}
\label{eq:transportByGlobalVelTripleLineCutoff}
|\partial_t\zeta_{\trLine} + (B\cdot\nabla)\zeta_{\trLine}| 
\leq C(\dist^2(\cdot,\trLine)\wedge 1)
\end{align}
throughout $B_{\hat r}(\trLine(t))\setminus\trLine(t)$ for all~$t\in [0,T]$.
Post-processing the previous display by means of~\eqref{eq:compDistances2}
then yields~\eqref{eq:advectionCutoffInterface} 
on $B_{\hat r}(\trLine(t))\cap W_{\barI_{k,i}}(t)$ for all~$t\in [0,T]$.

\textit{Substep 3.2: Proof of~\eqref{eq:advectionCutoffInterface} 
on $B_{\hat r}(\trLine(t))\cap W_{\bar\Omega_{i}}(t)$.}
Recall~\eqref{eq:defInterfaceCutoffInterpolWedge1}--\eqref{eq:defInterfaceCutoffInterpolWedge2}, i.e.,
$\eta_{\barI_{k,i}}=\lambda_{\bar\Omega_i}^{\barI_{k,i}}(1{-}\zeta_{\trLine})\zeta_{\barI_{k,i}}$
on $B_{\hat r}(\trLine(t))\cap W_{\bar\Omega_{i}}(t)$ for all~$t\in [0,T]$.
It then directly follows from the product rule, 
the trivial estimate $1{-}\zeta_{\trLine} \leq C(\dist^2(\cdot,\trLine)\wedge 1)$,
the estimate~\eqref{eq:advectionInterpolFunctionByGlobalVel} on the advective derivative
of the interpolation function~$\lambda_{\bar\Omega_i}^{\barI_{k,i}}=1{-}\lambda_{\bar\Omega_i}^{\barI_{i,j}}$,
the regularity estimates~\eqref{eq:regEstimateInterfaceCutoff} and~\eqref{eq:regEstimateGlobalVel},
the estimate~\eqref{eq:transportByGlobalVelTripleLineCutoff}, and finally the bound~\eqref{eq:compDistances} 
that~\eqref{eq:advectionCutoffInterface} holds true
on $B_{\hat r}(\trLine(t))\cap W_{\bar\Omega_{i}}(t)$ for all~$t\in [0,T]$.

This concludes the proof of Lemma~\ref{lemma:advectionCutOff}
since the argument on the other relevant interpolation wedge
proceeds analogously.
\end{proof}

\subsection{Approximate transport equations and motion by mean curvature}
We establish the validity of the estimates~\eqref{TransportEquationXi}--\eqref{Dissip}
in terms of the global extensions~$(\xi_{i,j})_{i,j\in\{1,2,3\},i\neq j}$ of the unit normal vector fields
from Construction~\ref{globalXi} and the global extension~$B$ of the velocity field
from Construction~\ref{globalVelocity}.

\begin{lemma}
\label{lemma:evolutionEquationsGradientFlowCalibration}
Let the assumptions and notation from Construction~\ref{globalXi}
and Construction~\ref{globalVelocity} be in place. There exists 
a constant $C>0$, which depends only on the data of the smoothly evolving regular double 
bubble~$(\bar\Omega_1,\bar\Omega_2,\bar\Omega_3)$ on~$[0,T]$,
such that for all $i,j\in\{1,2,3\}$ with $i\neq j$ it holds throughout~$\mathbb{R}^3{\times} [0,T]$
\begin{align}
\label{eq:auxEvol1}
|\partial_t\xi_{i,j}+(B\cdot\nabla)\xi_{i,j}+(\nabla B)^\mathsf{T}\xi_{i,j}|
&\leq C(\dist(\cdot,\barI_{i,j}) \wedge 1),
\\
\label{eq:auxEvol2}
|B\cdot\xi_{i,j}+\nabla\cdot\xi_{i,j}|
&\leq C(\dist(\cdot,\barI_{i,j}) \wedge 1),
\\
\label{eq:auxEvol3}
|\xi_{i,j}\cdot(\partial_t\xi_{i,j}+(B\cdot\nabla)\xi_{i,j})|
&\leq C(\dist^2(\cdot,\barI_{i,j}) \wedge 1).
\end{align}
\end{lemma}

\begin{proof}
The main point of the proof is the reduction to the corresponding
assertions on the level of the local constructions~$(\xi_{i,j}^{\barI_{i,j}},B^{\barI_{i,j}})$
at two-phase interfaces (see Lemma~\ref{lemma:gradientFlowCalibrationInterface})
and the local construction~$(\xi^{\trLine},B^{\trLine})$ at a triple line
(see Proposition~\ref{prop:gradientFlowCalibrationTripleLine}).
The reduction argument is facilitated by an interplay of 
the estimates~\eqref{eq:coercivityInterfaceCutoff}--\eqref{eq:coercivityTimeDerivBulkCutoff}
resp.\ \eqref{eq:advectionCutoffBulk} and~\eqref{eq:advectionCutoffInterface}
with sufficient compatibility of the local and global constructions.
We list and prove the required compatibility estimates in a first step before starting 
with the proof of the bounds~\eqref{eq:auxEvol1}--\eqref{eq:auxEvol3}.

\textit{Step 1: Compatibility estimates.} We claim that for
all~$i,j\in\{1,2,3\}$ with~$i\neq j$ it holds on~$\Rd[3]\times[0,T]$
\begin{align}
\label{eq:globalComp1}
\mathds{1}_{\supp\eta_{\barI_{i,j}}}\big|\xi_{i,j}{-}\xi_{i,j}^{\barI_{i,j}}\big|
+ \mathds{1}_{\supp\eta_{\trLine}}\big|\xi_{i,j}{-}\xi_{i,j}^{\trLine}\big|
&\leq C(\dist(\cdot,\barI_{i,j})\wedge 1),
\\
\label{eq:globalComp2}
\mathds{1}_{\supp\eta_{\barI_{i,j}}}\big|B{-}B^{\barI_{i,j}}\big|
+ \mathds{1}_{\supp\eta_{\trLine}}\big|B{-}B^{\trLine}\big|
&\leq C(\dist(\cdot,\barI_{i,j})\wedge 1),
\\
\label{eq:globalComp3}
\mathds{1}_{\supp\eta_{\barI_{i,j}}}\big|(\nabla B{-}\nabla B^{\barI_{i,j}})^\mathsf{T}\xi^{\barI_{i,j}}_{i,j}\big|
+ \mathds{1}_{\supp\eta_{\trLine}}\big|(\nabla B&{-}\nabla B^{\trLine})^\mathsf{T}\xi^{\trLine}_{i,j}\big|
\\\nonumber
&\leq C(\dist(\cdot,\barI_{i,j})\wedge 1),
\\
\label{eq:globalComp4}
\mathds{1}_{\supp\eta_{\barI_{i,j}}}
\big|\big(\xi_{i,j}{-}\xi_{i,j}^{\barI_{i,j}}\big)\cdot\xi_{i,j}^{\barI_{i,j}}\big|
+ \mathds{1}_{\supp\eta_{\trLine}}
\big|\big(\xi_{i,j}&{-}\xi_{i,j}^{\trLine}\big)\cdot\xi_{i,j}^{\trLine}\big|
\\\nonumber
&\leq C(\dist^2(\cdot,\barI_{i,j})\wedge 1),
\\
\label{eq:globalComp5}
\mathds{1}_{\supp\eta_{\barI_{i,j}}}
\big|\xi_{i,j}^{\barI_{i,j}}\cdot\big((B{-}B^{\barI_{i,j}})\cdot\nabla\big)\xi^{\barI_{i,j}}_{i,j}\big|
+
\mathds{1}_{\supp\eta_{\trLine}}
&\big|\xi_{i,j}^{\trLine}\cdot\big((B{-}B^{\trLine})\cdot\nabla\big)\xi^{\trLine}_{i,j}\big|
\\\nonumber
&\leq C(\dist^2(\cdot,\barI_{i,j})\wedge 1).
\end{align}
For a proof of these compatibility estimates, we only focus on the respective
first left hand side terms. The proof for the second left hand side terms
follows along the same lines switching the roles of~$\barI_{i,j}$ and~$\trLine$
in the process.

Inserting the definition~\eqref{eq:globalXi}
and exploiting the estimate~\eqref{eq:coercivityInterfaceCutoff}
yields $\xi_{i,j}{-}\xi^{\barI_{i,j}}_{i,j}=\eta_{\trLine}(\xi^{\trLine}_{i,j}{-}\xi^{\barI_{i,j}}_{i,j})
-\eta_{\mathrm{bulk}}\xi^{\barI_{i,j}}_{i,j} +O(\dist^2(\cdot,\barI_{i,j})\wedge 1)$
on~$\supp\eta_{\barI_{i,j}}$.
Hence, we obtain the asserted bound~\eqref{eq:globalComp1} 
thanks to the estimates~\eqref{eq:compEstimate1} and~\eqref{eq:coercivityBulkCutoff}.

Next, the definition~\eqref{eq:globalVelocity}
together with the estimates~\eqref{eq:compEstimate3}, \eqref{eq:coercivityInterfaceCutoff},
\eqref{eq:coercivityBulkCutoff} and~\eqref{eq:regEstimateLocalVelocitiesAux} implies
$B{-}B^{\barI_{i,j}}=\eta_{\trLine}(B^{\trLine}{-}B^{\barI_{i,j}})
{-}\eta_{\mathrm{bulk}}B^{\barI_{i,j}}{+}O(\dist(\cdot,\barI_{i,j})\wedge 1)
= O(\dist(\cdot,\barI_{i,j})\wedge 1)$ on~$\supp\eta_{\barI_{i,j}}$ as required.

Moreover, it holds on~$\supp\eta_{\barI_{i,j}}$ as a consequence
of the definition~\eqref{eq:globalVelocity}, the product rule, 
the already established compatibility estimate~\eqref{eq:globalComp2},
as well as the estimates~\eqref{eq:coercivityInterfaceCutoff}, \eqref{eq:coercivityGradientInterfaceCutoff}
and~\eqref{eq:regEstimateLocalVelocitiesAux} that
\begin{align*}
(\nabla B{-}\nabla B^{\barI_{i,j}})^\mathsf{T}\xi^{\barI_{i,j}}_{i,j}
&= \eta_{\trLine}(\nabla B^{\trLine}{-}\nabla B^{\barI_{i,j}})^\mathsf{T}\xi^{\barI_{i,j}}_{i,j}
-\eta_{\mathrm{bulk}}(\nabla B^{\barI_{i,j}})^\mathsf{T}\xi^{\barI_{i,j}}_{i,j}
\\&~~~
+(B^{\trLine}\cdot\xi^{\barI_{i,j}}_{i,j})\nabla\eta_{\trLine}
+(B^{\barI_{i,j}}\cdot\xi_{i,j}^{\barI_{i,j}})\nabla\eta_{\barI_{i,j}}
+ O(\dist(\cdot,\barI_{i,j})\wedge 1)
\\&
=\eta_{\trLine}(\nabla B^{\trLine}{-}\nabla B^{\barI_{i,j}})^\mathsf{T}\xi^{\barI_{i,j}}_{i,j}
\\&~~~
- (B\cdot\xi_{i,j}^{\barI_{i,j}})\nabla\eta_{\mathrm{bulk}}
-\eta_{\mathrm{bulk}}(\nabla B^{\barI_{i,j}})^\mathsf{T}\xi^{\barI_{i,j}}_{i,j}
+ O(\dist(\cdot,\barI_{i,j})\wedge 1).
\end{align*}
The previous display in turn implies~\eqref{eq:globalComp3}
in view of the bounds~\eqref{eq:compEstimate4}, \eqref{eq:coercivityBulkCutoff},
\eqref{eq:coercivityGradientBulkCutoff}, \eqref{eq:regEstimateLocalVelocitiesAux}
and~\eqref{eq:regEstimateGlobalVel}.

By the argument for~\eqref{eq:globalComp1}
we also have $(\xi_{i,j}{-}\xi^{\barI_{i,j}}_{i,j})\cdot\xi_{i,j}^{\barI_{i,j}}
=\eta_{\trLine}(\xi^{\trLine}_{i,j}{-}\xi^{\barI_{i,j}}_{i,j})\cdot\xi_{i,j}^{\barI_{i,j}}
-\eta_{\mathrm{bulk}}|\xi^{\barI_{i,j}}_{i,j}|^2 +O(\dist^2(\cdot,\barI_{i,j})\wedge 1)$
on~$\supp\eta_{\barI_{i,j}}$. Hence, we deduce from~\eqref{eq:compEstimate2}
and~\eqref{eq:coercivityBulkCutoff} that~\eqref{eq:globalComp4} holds true.

Finally, based on the definition~\eqref{eq:globalVelocity}
and the estimates~\eqref{eq:coercivityInterfaceCutoff},
\eqref{eq:auxRegEstimatesLocalXi} and~\eqref{eq:regEstimateLocalVelocitiesAux}, 
we may bound on~$\supp\eta_{\barI_{i,j}}$
\begin{align*}
&\xi_{i,j}^{\barI_{i,j}}\cdot\big((B{-}B^{\barI_{i,j}})\cdot\nabla\big)\xi^{\barI_{i,j}}_{i,j}
\\&
= \eta_{\trLine}\big(\xi_{i,j}^{\barI_{i,j}}{-}\xi^{\trLine}_{i,j}\big)\cdot
\big((B^{\trLine}{-}B^{\barI_{i,j}})\cdot\nabla\big)\xi^{\barI_{i,j}}_{i,j}
+ \eta_{\trLine}\xi_{i,j}^{\trLine}\cdot
\big((B^{\trLine}{-}B^{\barI_{i,j}})\cdot\nabla\big)\xi^{\barI_{i,j}}_{i,j}
\\&~~~
- \eta_{\mathrm{bulk}}\xi_{i,j}^{\barI_{i,j}}\cdot(B^{\barI_{i,j}}\cdot\nabla)\xi^{\barI_{i,j}}_{i,j}
+ O(\dist^2(\cdot,\barI_{i,j})\wedge 1),
\end{align*}
so that~\eqref{eq:compEstimate1}, \eqref{eq:compEstimate3}, \eqref{eq:coercivityBulkCutoff},
\eqref{eq:auxRegEstimatesLocalXi} and~\eqref{eq:regEstimateLocalVelocitiesAux}
entail the desired estimate~\eqref{eq:globalComp5}.

\textit{Step 2: Proof of~\eqref{eq:auxEvol1}.}
For the sake of brevity, from now on we refrain from explicitly
spelling out the application of the regularity estimates~\eqref{eq:regEstimateGlobalExtensionXi},
\eqref{eq:auxRegEstimatesLocalXi}, \eqref{eq:regEstimateGlobalVel}
or~\eqref{eq:regEstimateLocalVelocitiesAux}, and thus solely
concentrate on the error contributions in terms of
the distance to the interface~$\barI_{i,j}$.

We start estimating based on the definition~\eqref{eq:globalXi}, the product rule,
as well as the bounds~\eqref{eq:coercivityInterfaceCutoff}
and~\eqref{eq:coercivityTimeDerivativeInterfaceCutoff}
\begin{align*}
\partial_t\xi_{i,j} =
\eta_{\trLine}\partial_t\xi_{i,j}^{\trLine}
+ \eta_{\barI_{i,j}}\partial_t\xi_{i,j}^{\barI_{i,j}}
+ \xi_{i,j}^{\trLine}\partial_t\eta_{\trLine}
+ \xi_{i,j}^{\barI_{i,j}}\partial_t\eta_{\barI_{i,j}}
+ O(\dist(\cdot,\barI_{i,j})\wedge 1).
\end{align*}
As a consequence of the compatibility estimate~\eqref{eq:globalComp1}
and the bounds~\eqref{eq:coercivityTimeDerivativeInterfaceCutoff},
we may add zero twice and obtain
\begin{align*}
\xi_{i,j}^{\trLine}\partial_t\eta_{\trLine}
+ \xi_{i,j}^{\barI_{i,j}}\partial_t\eta_{\barI_{i,j}}
&= \xi_{i,j}(\partial_t\eta_{\trLine}{+}\partial_t\eta_{\barI_{i,j}})
+ O(\dist(\cdot,\barI_{i,j})\wedge 1)
\\&
= -\xi_{i,j}\partial_t\eta_{\mathrm{bulk}} + O(\dist(\cdot,\barI_{i,j})\wedge 1).
\end{align*}
The previous two displays combine to
\begin{align}
\label{eq:auxEvol1Aux1}
\partial_t\xi_{i,j} =
\eta_{\trLine}\partial_t\xi_{i,j}^{\trLine}
+ \eta_{\barI_{i,j}}\partial_t\xi_{i,j}^{\barI_{i,j}}
-\xi_{i,j}\partial_t\eta_{\mathrm{bulk}}
+ O(\dist(\cdot,\barI_{i,j})\wedge 1).
\end{align}
Replacing the differential operator~$\partial_t$ by~$(B\cdot\nabla)$
in the previous argument entails
\begin{align*}
(B\cdot\nabla)\xi_{i,j} &=
\eta_{\trLine}(B\cdot\nabla)\xi_{i,j}^{\trLine}
+ \eta_{\barI_{i,j}}(B\cdot\nabla)\xi_{i,j}^{\barI_{i,j}}
\\&~~~
-\xi_{i,j}(B\cdot\nabla)\eta_{\mathrm{bulk}}
+ O(\dist(\cdot,\barI_{i,j})\wedge 1).
\end{align*}
Making use of the compatibility estimate~\eqref{eq:globalComp2}
updates the previous display to
\begin{align}
\label{eq:auxEvol1Aux2}
(B\cdot\nabla)\xi_{i,j} &=
\eta_{\trLine}(B^{\trLine}\cdot\nabla)\xi_{i,j}^{\trLine}
+ \eta_{\barI_{i,j}}(B^{\barI_{i,j}}\cdot\nabla)\xi_{i,j}^{\barI_{i,j}}
\\&~~~\nonumber
-\xi_{i,j}(B\cdot\nabla)\eta_{\mathrm{bulk}}
+ O(\dist(\cdot,\barI_{i,j})\wedge 1).
\end{align}
Inserting the definition~\eqref{eq:globalXi}, recalling the estimate~\eqref{eq:coercivityInterfaceCutoff},
and adding zero based on the compatibility estimate~\eqref{eq:globalComp3}
moreover allows to estimate
\begin{align}
\nonumber
(\nabla B)^\mathsf{T}\xi_{i,j} &= 
\eta_{\trLine}(\nabla B)^\mathsf{T}\xi_{i,j}^{\trLine}
+ \eta_{\barI_{i,j}}(\nabla B)^\mathsf{T}\xi_{i,j}^{\barI_{i,j}}
+ O(\dist(\cdot,\barI_{i,j})\wedge 1)
\\&\label{eq:auxEvol1Aux3}
= \eta_{\trLine}(\nabla B^{\trLine})^\mathsf{T}\xi_{i,j}^{\trLine}
+ \eta_{\barI_{i,j}}(\nabla B^{\barI_{i,j}})^\mathsf{T}\xi_{i,j}^{\barI_{i,j}}
+ O(\dist(\cdot,\barI_{i,j})\wedge 1).
\end{align}
The desired estimate~\eqref{eq:auxEvol1} thus follows from~\eqref{eq:auxEvol1Aux1}--\eqref{eq:auxEvol1Aux3},
the estimate~\eqref{eq:advectionCutoffBulk} of the advective derivative of the bulk cutoff,
as well as the local versions~\eqref{eq:evolEquXiInterface} 
and \eqref{eq:timeEvolutionXiTripleLine} of~\eqref{eq:auxEvol1}, respectively.

\textit{Step 3: Proof of~\eqref{eq:auxEvol2}.}
We compute as a consequence of the definition~\eqref{eq:globalXi},
the estimate~\eqref{eq:coercivityInterfaceCutoff}, and the 
compatibility estimate~\eqref{eq:globalComp2}
\begin{align}
\nonumber
B\cdot\xi_{i,j} &= \eta_{\trLine}B\cdot\xi_{i,j}^{\trLine}
+ \eta_{\barI_{i,j}}B\cdot\xi_{i,j}^{\barI_{i,j}}
+ O(\dist(\cdot,\barI_{i,j})\wedge 1)
\\&\label{eq:auxEvol2Aux1}
= \eta_{\trLine}B^{\trLine}\cdot\xi_{i,j}^{\trLine}
+ \eta_{\barI_{i,j}}B^{\barI_{i,j}}\cdot\xi_{i,j}^{\barI_{i,j}}
+ O(\dist(\cdot,\barI_{i,j})\wedge 1).
\end{align}
We also directly estimate by means of  the definition~\eqref{eq:globalXi},
the estimate~\eqref{eq:coercivityGradientInterfaceCutoff},
as well as the compatibility estimate~\eqref{eq:globalComp1}
\begin{align}
\nonumber
\nabla\cdot\xi_{i,j} &= \eta_{\trLine}\nabla\cdot\xi_{i,j}^{\trLine}
{+} \eta_{\barI_{i,j}}\nabla\cdot\xi^{\barI_{i,j}}_{i,j}
{+} (\xi^{\trLine}_{i,j}\cdot\nabla)\eta_{\trLine}
{+} (\xi^{\barI_{i,j}}_{i,j}\cdot\nabla)\eta_{\barI_{i,j}}
{+} O(\dist(\cdot,\barI_{i,j})\wedge 1)
\\&\label{eq:auxEvol2Aux2}
= \eta_{\trLine}\nabla\cdot\xi_{i,j}^{\trLine}
+ \eta_{\barI_{i,j}}\nabla\cdot\xi^{\barI_{i,j}}_{i,j}
- (\xi_{i,j}\cdot\nabla)\eta_{\mathrm{bulk}}
+ O(\dist(\cdot,\barI_{i,j})\wedge 1).
\end{align}
Hence, the estimate~\eqref{eq:auxEvol2} follows by combining~\eqref{eq:auxEvol2Aux1}--\eqref{eq:auxEvol2Aux2},
the estimate~\eqref{eq:coercivityGradientBulkCutoff} for the bulk cutoff, and the local versions
of~\eqref{eq:auxEvol2} given by~\eqref{eq:divConstraintXiInterface} 
and~\eqref{eq:motionByMeanCurvatureTripleLine}, respectively.

\textit{Step 4: Proof of~\eqref{eq:auxEvol3}.} Plugging in the definition~\eqref{eq:globalXi},
recalling the estimate~\eqref{eq:coercivityInterfaceCutoff}, and denoting
by~$k\in\{1,2,3\}\setminus\{i,j\}$ the remaining phase yields
\begin{align*}
\xi_{i,j}\cdot\partial_t\xi_{i,j} &= \eta_{\trLine}\xi^{\trLine}_{i,j}\cdot\partial_t\xi_{i,j}
+ \eta_{\barI_{i,j}}\xi^{\barI_{i,j}}_{i,j}\cdot\partial_t\xi_{i,j} 
+ O(\dist^2(\cdot,\barI_{i,j})\wedge 1)
\\&
= \eta_{\trLine}^2\xi^{\trLine}_{i,j}\cdot\partial_t\xi^{\trLine}_{i,j}
+ \eta_{\barI_{i,j}}^2\xi^{\barI_{i,j}}_{i,j}\cdot\partial_t\xi^{\barI_{i,j}}_{i,j} 
\\&~~~
+ \eta_{\trLine}\eta_{\barI_{i,j}}\xi^{\trLine}_{i,j}\cdot\partial_t\xi^{\barI_{i,j}}_{i,j}
+ \eta_{\trLine}\eta_{\barI_{i,j}}\xi^{\barI_{i,j}}_{i,j}\cdot\partial_t\xi^{\trLine}_{i,j} 
\\&~~~
+ \eta_{\trLine}\xi^{\trLine}_{i,j}\cdot
\big(\xi^{\trLine}_{i,j}\partial_t\eta_{\trLine}
     + \xi^{\barI_{i,j}}_{i,j}\partial_t\eta_{\barI_{i,j}}
		 + \xi^{\barI_{j,k}}_{j,k}\partial_t\eta_{\barI_{j,k}}
		 + \xi^{\barI_{k,i}}_{k,i}\partial_t\eta_{\barI_{k,i}}\big)
\\&~~~
+ \eta_{\barI_{i,j}}\xi^{\barI_{i,j}}_{i,j}\cdot
\big(\xi^{\trLine}_{i,j}\partial_t\eta_{\trLine}
     + \xi^{\barI_{i,j}}_{i,j}\partial_t\eta_{\barI_{i,j}}
		 + \xi^{\barI_{j,k}}_{j,k}\partial_t\eta_{\barI_{j,k}}
		 + \xi^{\barI_{k,i}}_{k,i}\partial_t\eta_{\barI_{k,i}}\big)
\\&~~~
+ O(\dist^2(\cdot,\barI_{i,j})\wedge 1).
\end{align*}
The compatibility estimates~\eqref{eq:globalComp1} and~\eqref{eq:globalComp4}
in combination with the bounds~\eqref{eq:coercivityInterfaceCutoff},
and~\eqref{eq:coercivityBulkCutoff}
provide an upgrade of the previous display in form of
\begin{align}
\label{eq:auxEvol3Aux1}
\xi_{i,j}\cdot\partial_t\xi_{i,j} &=
\eta_{\trLine}^2\xi^{\trLine}_{i,j}\cdot\partial_t\xi^{\trLine}_{i,j}
+ \eta_{\barI_{i,j}}^2\xi^{\barI_{i,j}}_{i,j}\cdot\partial_t\xi^{\barI_{i,j}}_{i,j} 
\\&~~~\nonumber
+ \eta_{\trLine}\eta_{\barI_{i,j}}\xi^{\trLine}_{i,j}\cdot\partial_t\xi^{\barI_{i,j}}_{i,j}
+ \eta_{\trLine}\eta_{\barI_{i,j}}\xi^{\barI_{i,j}}_{i,j}\cdot\partial_t\xi^{\trLine}_{i,j} 
\\&~~~\nonumber
+ \eta_{\trLine}(\xi^{\trLine}_{i,j}\cdot\xi_{i,j})\partial_t(\eta_{\trLine}{+}\eta_{\barI_{i,j}})
\\&~~~\nonumber
+ \eta_{\barI_{i,j}}(\xi^{\barI_{i,j}}_{i,j}\cdot\xi_{i,j})\partial_t(\eta_{\trLine}{+}\eta_{\barI_{i,j}})
\\&~~~\nonumber
+ \eta_{\trLine}\xi^{\trLine}_{i,j}\cdot
\big(\xi^{\barI_{j,k}}_{j,k}\partial_t\eta_{\barI_{j,k}}
		 + \xi^{\barI_{k,i}}_{k,i}\partial_t\eta_{\barI_{k,i}}\big)
\\&~~~\nonumber
+ \eta_{\barI_{i,j}}\xi^{\barI_{i,j}}_{i,j}\cdot
\big(\xi^{\barI_{j,k}}_{j,k}\partial_t\eta_{\barI_{j,k}}
		 + \xi^{\barI_{k,i}}_{k,i}\partial_t\eta_{\barI_{k,i}}\big)
\\&~~~\nonumber
+ O(\dist^2(\cdot,\barI_{i,j})\wedge 1).
\end{align}
Substituting the differential operator~$(B\cdot\nabla)$ for~$\partial_t$
in the previous argument, making use of the compatibility estimates~\eqref{eq:globalComp5}, 
\eqref{eq:globalComp1} and~\eqref{eq:globalComp2}, and exploiting twice the estimate~\eqref{eq:advectionCutoffInterface}
then shows that
\begin{align*}
&\xi_{i,j}\cdot(\partial_t {+} B\cdot\nabla)\xi_{i,j}
\\
&=
 \eta_{\trLine}^2\xi^{\trLine}_{i,j}\cdot(\partial_t {+} B^{\trLine}\cdot\nabla)\xi^{\trLine}_{i,j}
+ \eta_{\barI_{i,j}}^2\xi^{\barI_{i,j}}_{i,j}
\cdot(\partial_t {+} B^{\barI_{i,j}}\cdot\nabla)\xi^{\barI_{i,j}}_{i,j} 
\\&~~~
+ \eta_{\trLine}\eta_{\barI_{i,j}}\xi^{\trLine}_{i,j}
\cdot(\partial_t {+}B^{\barI_{i,j}}\cdot\nabla)\xi^{\barI_{i,j}}_{i,j}
+ \eta_{\trLine}\eta_{\barI_{i,j}}\xi^{\barI_{i,j}}_{i,j}
\cdot(\partial_t {+}B^{\trLine}\cdot\nabla)\xi^{\trLine}_{i,j} 
\\&~~~
- \eta_{\trLine}(\xi^{\trLine}_{i,j}\cdot\xi_{i,j})(\partial_t{+}B\cdot\nabla)\eta_{\mathrm{bulk}}
- \eta_{\barI_{i,j}}(\xi^{\barI_{i,j}}_{i,j}\cdot\xi_{i,j})(\partial_t{+}B\cdot\nabla)\eta_{\mathrm{bulk}}
\\&~~~
+ O(\dist^2(\cdot,\barI_{i,j})\wedge 1).
\end{align*}
Hence, employing the local versions~\eqref{eq:evolEquLengthXiInterface} 
and~\eqref{eq:timeEvolutionLengthXiTripleLine} of~\eqref{eq:auxEvol3}
and making use of the estimate~\eqref{eq:advectionCutoffBulk}
for the bulk cutoff shows that
\begin{align}
\nonumber
&\xi_{i,j}\cdot(\partial_t {+} B\cdot\nabla)\xi_{i,j}
\\\label{eq:auxEvol3Aux2} 
&= \eta_{\trLine}\eta_{\barI_{i,j}}\xi^{\trLine}_{i,j}
\cdot(\partial_t {+}B^{\barI_{i,j}}\cdot\nabla)\xi^{\barI_{i,j}}_{i,j}
+ \eta_{\trLine}\eta_{\barI_{i,j}}\xi^{\barI_{i,j}}_{i,j}
\cdot(\partial_t {+}B^{\trLine}\cdot\nabla)\xi^{\trLine}_{i,j} 
\\&~~~\nonumber
+ O(\dist^2(\cdot,\barI_{i,j})\wedge 1).
\end{align}
Adding zero, making use 
of the local evolution equations~\eqref{eq:evolEquXiInterface} resp.\ \eqref{eq:evolEquLengthXiInterface}, 
and exploiting the compatibility estimates~\eqref{eq:globalComp1} and~\eqref{eq:globalComp3}
further implies that
\begin{align*}
&\eta_{\trLine}\eta_{\barI_{i,j}}\xi^{\trLine}_{i,j}\cdot
\big(\partial_t\xi^{\barI_{i,j}}_{i,j}{+}(B^{\barI_{i,j}}\cdot\nabla)\xi^{\barI_{i,j}}_{i,j}\big)
\\&
= \eta_{\trLine}\eta_{\barI_{i,j}}\xi^{\trLine}_{i,j}\cdot
\big(\partial_t\xi^{\barI_{i,j}}_{i,j}{+}(B^{\barI_{i,j}}\cdot\nabla)\xi^{\barI_{i,j}}_{i,j}
{+}(\nabla B^{\barI_{i,j}})^\mathsf{T}\xi^{\barI_{i,j}}_{i,j}\big)
- \eta_{\trLine}\eta_{\barI_{i,j}}\xi^{\trLine}_{i,j}\cdot
(\nabla B^{\barI_{i,j}})^\mathsf{T}\xi^{\barI_{i,j}}_{i,j}
\\&
=  -\eta_{\trLine}\eta_{\barI_{i,j}}\big(\xi^{\barI_{i,j}}_{i,j}{-}\xi^{\trLine}_{i,j}\big)
(\nabla B)^\mathsf{T}\xi^{\barI_{i,j}}_{i,j}
+ O(\dist^2(\cdot,\barI_{i,j})\wedge 1).
\end{align*}
Switching the roles of~$\trLine$ and~$\barI_{i,j}$ in the argument leading
to the previous display, relying in the process on
the local evolution equations~\eqref{eq:timeEvolutionXiTripleLine}
resp.\ \eqref{eq:timeEvolutionLengthXiTripleLine},
we then in summary obtain together with~\eqref{eq:globalComp1}
\begin{align}
\nonumber
&\eta_{\trLine}\eta_{\barI_{i,j}}\xi^{\trLine}_{i,j}\cdot
\big(\partial_t\xi^{\barI_{i,j}}_{i,j}{+}(B^{\barI_{i,j}}\cdot\nabla)\xi^{\barI_{i,j}}_{i,j}\big)
+ \eta_{\trLine}\eta_{\barI_{i,j}}\xi^{\barI_{i,j}}_{i,j}\cdot
\big(\partial_t\xi^{\trLine}_{i,j}{+}(B^{\trLine}\cdot\nabla)\xi^{\trLine}_{i,j}\big)
\\&\nonumber
= -\eta_{\trLine}\eta_{\barI_{i,j}}\big(\xi^{\barI_{i,j}}_{i,j}{-}\xi^{\trLine}_{i,j}\big)
(\nabla B)^\mathsf{T}\big(\xi^{\barI_{i,j}}_{i,j}{-}\xi^{\trLine}_{i,j}\big)
+ O(\dist^2(\cdot,\barI_{i,j})\wedge 1)
\\&\label{eq:auxEvol3Aux3}
= O(\dist^2(\cdot,\barI_{i,j})\wedge 1).
\end{align}
The combination of the estimates~\eqref{eq:auxEvol3Aux2} and~\eqref{eq:auxEvol3Aux3}
thus entails the bound~\eqref{eq:auxEvol3}.
\end{proof}

\subsection{Existence of a gradient-flow calibration: Proof of Theorem~\ref{prop:existenceGradientFlowCalibration}}
This is only a matter of collecting already established facts.
More precisely, the required regularity for~$((\xi_{i,j})_{i,j\in\{1,2,3\},i\neq j},B)$
is part of Lemma~\ref{lemma:calibrationProperty} and Lemma~\ref{lemma:advectionCutOff}, respectively.
The calibration resp.\ extension property~\eqref{Calibrations} as well as the coercivity estimate~\eqref{LengthControlXi}
for the extensions of the unit normal vector fields follow from Lemma~\ref{lemma:calibrationProperty}.
The estimates~\eqref{TransportEquationXi}--\eqref{Dissip} are finally the content of
Lemma~\ref{lemma:evolutionEquationsGradientFlowCalibration}. \qed

\section{Existence of transported weights: Proof of Proposition~\ref{prop:existenceWeights}}
\label{sec:existenceWeights}
\begin{proof}[Proof of Proposition~\ref{prop:existenceWeights}]
The proof proceeds in several steps.

\textit{Step 1: Construction of an auxiliary family of transported weights.}
We first fix a smooth truncation of the identity.
More precisely, let~$\vartheta\colon\Rd[]\to\Rd[]$ be a 
smooth and non-decreasing map such that $\vartheta(r)=r$ for $|r|\leq\frac{1}{2}$, 
$\vartheta(r)=1$ for $r\geq 1$ and $\vartheta(r)= -1$ for $r\leq -1$.
Let~$\hat r\in (0,1]$ be the localization scale of Proposition~\ref{prop:gradientFlowCalibrationTripleLine},
let $\bar r\in(0,1]$ be the localization scale defined by~\eqref{eq:defLocScale},
and let finally $\delta\in (0,1]$ be the constant from \textit{Step~2} of the proof 
of Lemma~\ref{lemma:partitionOfUnity} (cf.\ the defining property~\eqref{eq:choiceDelta1}
for all~$i,j\in\{1,2,3\}$, $i\neq j$).
We then define building blocks
\begin{align}
\label{eq:interfaceWeight}
\vartheta_{i,j} &:= \vartheta\Big(\frac{s_{i,j}}{\delta\bar r}\Big) 
&&\text{in } \mathrm{im}(\Psi_{i,j}),
\\
\label{eq:extWeight}
\vartheta_{\mathrm{ext}} &:= \vartheta\Big(\frac{\dist(\cdot,\trLine)}{\hat r}\Big) 
&&\text{in } \Rd[3] \times [0,T].
\end{align}
Note that by definition~\eqref{eq:defLocScale} of the localization
scale~$\bar r$, we have for all phases $i\in\{1,2,3\}$ a covering of~$\partial\bar\Omega_i$ in form of
\begin{align}
\label{eq:decompNbhdGrainBoundary}
\partial\bar\Omega_i \subset B_{\hat r}(\trLine(t)) \cup
\bigcup_{j\in\{1,2,3\},j\neq i} \mathrm{im}_{\bar r}(\Psi_{i,j})(t) \setminus B_{\hat r}(\trLine(t))
=: \mathcal{N}^{\partial\bar\Omega_i}_{\hat r,\bar r}(t),
\end{align}
for all~$t\in [0,T]$, and where we abbreviated
\begin{align*}
\mathrm{im}_{\bar r}(\Psi_{i,j})(t) := \Psi_{i,j}(\barI_{i,j}(t){\times}\{t\}{\times}[-\bar r,\bar r]),
\quad t\in [0,T].
\end{align*}
Note that this also implies a disjoint covering of~$\Rd[3]$ by means of
\begin{align}
\label{eq:decompSpaceNbhdGrainBoundary}
\Rd[3] &= \mathcal{N}^{\partial\bar\Omega_i}_{\hat r,\bar r}(t)
\cup \big(\bar\Omega_i(t)\setminus \mathcal{N}^{\partial\bar\Omega_i}_{\hat r,\bar r}(t)\big)
\cup \big((\Rd[3]\setminus\bar\Omega_i(t))\setminus \mathcal{N}^{\partial\bar\Omega_i}_{\hat r,\bar r}(t)\big)
\end{align}
for all~$t\in [0,T]$.

For each phase~$i\in\{1,2,3\}$, denote by~$j,k\in\{1,2,3\}\setminus\{i\}$
the remaining two phases. We then define,
based on the building blocks~\eqref{eq:interfaceWeight} and~\eqref{eq:extWeight},
an auxiliary weight~$\hat\vartheta_i\colon \Rd[3]{\times}[0,T]\to[-1,1]$
by means of
\begin{align}
\label{eq:defWeight1}
\hat\vartheta_i(\cdot,t) &:= \vartheta_{i,\ell}(\cdot,t)
&&\text{in } \mathrm{im}_{\bar r}(\Psi_{i,j})(t)\setminus B_{\hat r}(\trLine(t)),\,\ell\neq i,
\\
\label{eq:defWeight2}
\hat\vartheta_i(\cdot,t) &:= \vartheta_{i,\ell}(\cdot,t) 
&&\text{in } \overline{W_{\barI_{i,j}}(t)} \cap B_{\hat r}(\trLine(t)),\,\ell\neq i,
\\
\label{eq:defWeight3}
\hat\vartheta_i(\cdot,t) &:= \lambda^{\barI_{i,j}}_{\bar\Omega_i}(\cdot,t)\vartheta_{i,j}(\cdot,t)
\\&~~~~\nonumber
+ \lambda^{\barI_{k,i}}_{\bar\Omega_i}(\cdot,t)\vartheta_{i,k}(\cdot,t)
&&\text{in } W_{\bar\Omega_i}(t) \cap B_{\hat r}(\trLine(t)),
\\
\label{eq:defWeight4}
\hat\vartheta_i(\cdot,t) &:= \vartheta_{\mathrm{ext}}(\cdot,t)
&&\text{in } \overline{W_{\barI_{j,k}}(t)} \cap B_{\hat r}(\trLine(t)),
\\
\label{eq:defWeight5}
\hat\vartheta_i(\cdot,t) &:= \lambda^{\barI_{i,j}}_{\bar\Omega_j}(\cdot,t)\vartheta_{i,j}(\cdot,t)
\\&~~~~\nonumber
+ \lambda^{\barI_{j,k}}_{\bar\Omega_j}(\cdot,t) \vartheta_{\mathrm{ext}}(\cdot,t)
&&\text{in } W_{\bar\Omega_j}(t) \cap B_{\hat r}(\trLine(t)),
\\
\label{eq:defWeight5b}
\hat\vartheta_i(\cdot,t) &:= \lambda^{\barI_{k,i}}_{\bar\Omega_k}(\cdot,t)\vartheta_{i,k}(\cdot,t)
\\&~~~~\nonumber
+ \lambda^{\barI_{j,k}}_{\bar\Omega_k}(\cdot,t) \vartheta_{\mathrm{ext}}(\cdot,t)
&&\text{in } W_{\bar\Omega_k}(t) \cap B_{\hat r}(\trLine(t)),
\\
\label{eq:defWeight6}
\hat\vartheta_i(\cdot,t) &:= -1
&&\text{in } \bar\Omega_i(t)\setminus \mathcal{N}^{\partial\bar\Omega_i}_{\hat r,\bar r}(t),
\\
\label{eq:defWeight7}
\hat\vartheta_i(\cdot,t) &:= 1
&&\text{else}
\end{align}
for all~$t\in [0,T]$. For the construction and properties of the interpolation functions,
we refer to Lemma~\ref{lemma:partitionOfUnity}.
Note that~$\hat\vartheta_i$ is well-defined in view of~\eqref{eq:decompNbhdGrainBoundary}, 
\eqref{eq:decompSpaceNbhdGrainBoundary} and~\eqref{eq:decompTripleLine}. Moreover,
due to the defining property~\eqref{eq:choiceDelta1} of
the constant~$\delta\in (0,1]$, we infer that~$\hat\vartheta_i$
is continuous throughout~$\Rd[3]{\times}[0,T]$.

\textit{Step 2: Properties of the auxiliary family of transported weights.}
In this step, we verify that the auxiliary family~$\hat\vartheta=(\hat\vartheta_i)_{i\in\{1,2,3\}}$
satisfies all the requirements of Definition~\ref{def:weights} with the (obvious)
exception that~$\hat\vartheta_i\in L^1(\Rd[3]{\times}[0,T])$.
The $W^{1,\infty}$-regularity on $\Rd[3]{\times}[0,T]$ as well
as the required conditions from item~\textit{ii)} of Definition~\ref{def:weights}
are immediate from the definitions~\eqref{eq:defWeight1}--\eqref{eq:defWeight7}.
Hence, we focus in the following on the deduction of the advection estimate~\eqref{eq:advectionWeights}.

\textit{Substep 2.1: Preliminary estimates.}
We first claim that for all~$i,j\in\{1,2,3\}$ with~$i\neq j$ and all~$t\in [0,T]$ it holds
\begin{align}
\label{eq:advectionInterfaceWeight}
|\partial_t\vartheta_{i,j}{+}(B\cdot\nabla)\vartheta_{i,j}|(\cdot,t) 
&\leq C\dist(\cdot,\partial\bar\Omega_i(t))
\quad\text{in } \mathrm{im}_{\bar r}(\Psi_{i,j})(t) \setminus B_{\hat r}(\trLine(t)),
\\
\label{eq:advectionInterfaceWeight2}
|\partial_t\vartheta_{i,j}{+}(B\cdot\nabla)\vartheta_{i,j}|(\cdot,t) 
&\leq C\dist(\cdot,\partial\bar\Omega_i(t))
\\&\nonumber~
\text{in } B_{\hat r}(\trLine(t)) \cap \big(W_{\barI_{i,j}}(t)\cup W_{\bar\Omega_i}(t)\cup W_{\bar\Omega_j}(t)\big),
\\
\label{eq:advectionExtWeight}
|\partial_t\vartheta_{\mathrm{ext}}{+}(B\cdot\nabla)\vartheta_{\mathrm{ext}}|(\cdot,t)
&\leq C\dist(\cdot,\trLine(t))
\quad\text{in } B_{\hat r}(\trLine(t))\setminus\trLine(t).
\end{align}

We start with a proof of~\eqref{eq:advectionInterfaceWeight}.
It follows from the representation~\eqref{eq:bulkCutoffAwayTripleLine}
and the definition~\eqref{eq:globalVelocity}
that $B=\eta_{\barI_{i,j}}B^{\barI_{i,j}}$ 
in~$\mathrm{im}_{\bar r}(\Psi_{i,j})(t) \setminus B_{\hat r}(\trLine(t))$
for all~$t\in [0,T]$. 
We may then estimate by the chain rule,
the definition~\eqref{eq:interfaceWeight}, the identity \eqref{eq:evolSignedDistanceAux1000}, 
the representation~\eqref{eq:bulkCutoffAwayTripleLine}, 
as well as the estimate~\eqref{eq:coercivityBulkCutoff}
\begin{align*}
|\partial_t\vartheta_{i,j}{+}(B\cdot\nabla)\vartheta_{i,j}|
\leq \eta_{\mathrm{bulk}}|(B^{\barI_{i,j}}\cdot\nabla)\vartheta_{i,j}|
\leq C\dist(\cdot,\bar\Omega_i)
\end{align*}
throughout~$\mathrm{im}_{\bar r}(\Psi_{i,j})(t) \setminus B_{\hat r}(\trLine(t))$
for all~$t\in [0,T]$. 

We next prove~\eqref{eq:advectionInterfaceWeight2}.
Throughout the interface wedge~$W_{\barI_{i,j}}(t)\cap B_{\hat r}(\trLine(t))$, it holds $B=\eta_{\trLine}B^{\trLine}
+ \eta_{\barI_{i,j}}B^{\barI_{i,j}}$ thanks to the representation~\eqref{eq:bulkCutoffInterfaceWedge}
and the definition~\eqref{eq:globalVelocity}. Employing~\eqref{eq:bulkCutoffInterfaceWedge} once more,
we then estimate making also use of the chain rule,
the definition~\eqref{eq:interfaceWeight} and the identity~\eqref{eq:evolSignedDistanceAux1000}
\begin{align*}
|\partial_t\vartheta_{i,j}{+}(B\cdot\nabla)\vartheta_{i,j}|
\leq \eta_{\mathrm{bulk}}|(B^{\barI_{i,j}}\cdot\nabla)\vartheta_{i,j}|
+ \eta_{\trLine}\big|\big((B^{\trLine}{-}B^{\barI_{i,j}})\cdot\nabla\big)\vartheta_{i,j}\big|
\end{align*}
on~$W_{\barI_{i,j}}(t)\cap B_{\hat r}(\trLine(t))$ for all~$t\in [0,T]$.
Post-processing the previous display by means of~\eqref{eq:coercivityBulkCutoff}, \eqref{eq:compEstimate3}
and~\eqref{eq:compDistances3} thus yields~\eqref{eq:advectionInterfaceWeight2}
on~$W_{\barI_{i,j}}(t)\cap B_{\hat r}(\trLine(t))$, $t\in [0,T]$.

Throughout~$W_{\bar\Omega_{i}}(t)\cap B_{\hat r}(\trLine(t))$,
we may write, as a consequence of the representation~\eqref{eq:bulkCutoffInterpolWedge}, 
the global velocity defined by~\eqref{eq:globalVelocity}
in form of $B=\eta_{\trLine}B^{\trLine}+ \eta_{\barI_{i,j}}B^{\barI_{i,j}}+\eta_{\barI_{k,i}}B^{\barI_{k,i}}$.
Hence, based on the same ingredients as in the case of interface wedges we may estimate
\begin{align*}
&|\partial_t\vartheta_{i,j}{+}(B\cdot\nabla)\vartheta_{i,j}|
\\
&\leq \eta_{\mathrm{bulk}}|(B^{\barI_{i,j}}\cdot\nabla)\vartheta_{i,j}|
+ \eta_{\trLine}\big|\big((B^{\trLine}{-}B^{\barI_{i,j}})\cdot\nabla\big)\vartheta_{i,j}\big|
+ \eta_{\barI_{k,i}}\big|\big((B^{\barI_{k,i}}{-}B^{\barI_{i,j}})\cdot\nabla\big)\vartheta_{i,j}\big|
\end{align*}
on~$W_{\bar\Omega_{i}}(t)\cap B_{\hat r}(\trLine(t))$ for all~$t\in [0,T]$.
The previous display in turn upgrades to the desired estimate~\eqref{eq:advectionInterfaceWeight2}
thanks to~\eqref{eq:coercivityBulkCutoff}, \eqref{eq:compEstimate3} and~\eqref{eq:compDistances}.

Finally, the estimate~\eqref{eq:advectionExtWeight} is a direct consequence
of the chain rule, the definition~\eqref{eq:extWeight}, the identity~\eqref{eq:transportByGlobalVelDistanceTripleLine} 
and the regularity estimate~\eqref{eq:regEstimateGlobalVel}. 

\textit{Substep 2.2: Proof of~\eqref{eq:advectionWeights} in terms of~$(\hat\vartheta_i)_{i\in\{1,2,3\}}$.}
We first observe that as a consequence of the definitions~\eqref{eq:defWeight1}--\eqref{eq:defWeight7},
there exists~$C\geq 1$ such that
\begin{align}
\label{eq:coercivityAuxWeight}
\frac{1}{C}|\hat\vartheta_i| \leq
\dist(\cdot,\partial\bar\Omega_i)\leq C|\hat\vartheta_i|
\quad\text{in } \Rd[3]{\times}[0,T].
\end{align}
Modulo this post-processing, the claim~\eqref{eq:advectionWeights} 
in terms of~$\hat\vartheta_i$ is then directly implied 
for the regions~\eqref{eq:defWeight1}, \eqref{eq:defWeight2} and~\eqref{eq:defWeight4}
by the estimates~\eqref{eq:advectionInterfaceWeight}--\eqref{eq:advectionExtWeight}
and~\eqref{eq:compDistances2}. Furthermore,
the only additional ingredients needed in the interpolation
regions~\eqref{eq:defWeight3}, \eqref{eq:defWeight5} and~\eqref{eq:defWeight5b}
are given by the estimate~\eqref{eq:advectionInterpolFunctionByGlobalVel} for the interpolation functions
as well as the bound~\eqref{eq:compDistances}.
Since there is nothing to prove for the regions~\eqref{eq:defWeight6} and~\eqref{eq:defWeight7},
this in turn concludes the proof of~\eqref{eq:advectionWeights} in terms of~$(\hat\vartheta_i)_{i\in\{1,2,3\}}$.

\textit{Step 3: Enforcing integrability of the weights.}
We slightly modify the construction from the previous step to
take care of the integrability issue. To this end, we first choose
a smooth and concave function $\kappa\colon [0,\infty)\to [0,1]$
such that $\kappa(0)=0$ as well as $\kappa(r)=1$ for $r\geq 1$. 
Which we think of as an upper concave approximation
of the map $r\mapsto r\wedge 1$ on the interval $[0,\infty)$. 
Choose a sufficiently large radius~$R>0$ such that
\begin{align}
\label{eq:cutoffScale}
\bigcup_{t\in [0,T]} \bigcup_{i,j\in\{1,2,3\},i\neq j}
B_{\hat r}(\bar I_{i,j}(t))\times\{t\}
\subset\subset B_R(0).
\end{align}
We then define a weight
$\eta_R\in W^{1,\infty}(\Rd[3])\cap W^{1,1}(\Rd[3])$
by means of
\begin{align}
\label{eq:intWeight}
\eta_R(x) := \kappa(\exp(R{-}|x|)),\quad x\in\Rd[3],
\end{align}
with its spatial gradient being bounded in form of
\begin{align}
\label{eq:gradientBoundIntWeight}
|\nabla\eta_R| \leq C|\eta_R| \quad\text{in } \Rd[3].
\end{align}
With all of these ingredients in place, we may finally define $\vartheta_i:=\eta_R\hat\vartheta_i$
for all phases $i\in\{1,2,3\}$. Note that~$\vartheta_i\in W^{1,1}(\Rd[3]{\times}[0,T];[-1,1])$
as desired. Moreover, the weights~$\vartheta_i$ directly inherit all the other 
required properties of Definition~\ref{def:weights} from
the auxiliary weights~$\hat\vartheta_i$ of the previous step, as can be seen from the definitions.
\end{proof}

\section*{Acknowledgments}
This project has received funding from the European Research Council 
(ERC) under the European Union's Horizon 2020 research and innovation 
programme (grant agreement No 948819)
\begin{tabular}{@{}c@{}}\includegraphics[width=8ex]{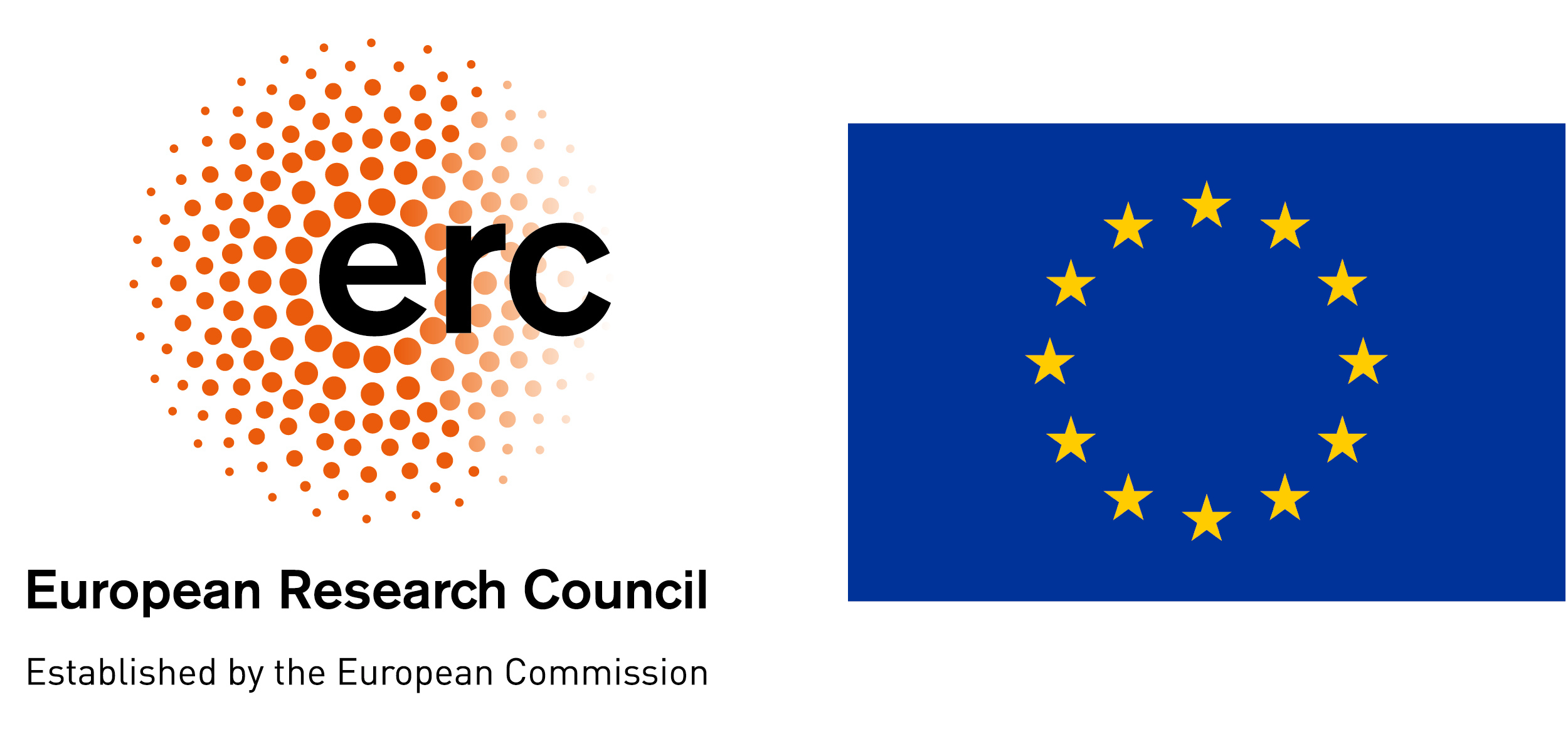}\end{tabular}, and 
from the Deutsche Forschungsgemeinschaft (DFG, German Research Foundation) 
under Germany's Excellence Strategy -- EXC-2047/1 -- 390685813.

\bibliographystyle{abbrv}
\bibliography{stability_double_bubbles}

\end{document}